\documentclass[11pt,oneside,english,reqno,dvipsnames,table]{amsart}
\usepackage[T1]{fontenc}
\usepackage[a4paper]{geometry}
\usepackage{mathrsfs}
\usepackage{amstext}
\usepackage{amsthm}
\usepackage{amssymb}
\usepackage{bm}
\usepackage{color}
\usepackage{hyperref}

\usepackage{enumerate}
\usepackage{verbatim} 
\usepackage{enumitem}
\usepackage[dvipsnames]{xcolor}
\usepackage{nicefrac}
\usepackage{times}
\usepackage{mathtools}

\usepackage{graphicx} 
\usepackage{color,fullpage,float,caption,subcaption}

\usepackage{multirow}

\usepackage{parskip}

\makeatletter
\numberwithin{equation}{section}
\theoremstyle{plain}
\newtheorem{thm}{\protect\theoremname}
\theoremstyle{definition}
\newtheorem{defn}[thm]{\protect\definitionname}
\theoremstyle{remark}
\newtheorem{rem}[thm]{\protect\remarkname}
\theoremstyle{plain}
\newtheorem{lem}[thm]{\protect\lemmaname}
\theoremstyle{plain}
\newtheorem{prop}[thm]{\protect\propositionname}
\theoremstyle{plain}
\newtheorem{cor}[thm]{\protect\corollaryname}
\theoremstyle{plain}

\theoremstyle{plain}
\newtheorem{ex}[thm]{\protect\examplename}
\theoremstyle{plain}

\theoremstyle{plain}

\newtheorem{ass}[thm]{\protect\assumptionname}

\makeatother

\usepackage{babel}
\providecommand{\corollaryname}{Corollary}
\providecommand{\definitionname}{Definition}
\providecommand{\lemmaname}{Lemma}
\providecommand{\propositionname}{Proposition}
\providecommand{\remarkname}{Remark}
\providecommand{\theoremname}{Theorem}
\providecommand{\hypothesisname}{Hypothesis}
\providecommand{\problemname}{Problem}
\providecommand{\examplename}{Example}
\providecommand{\notationname}{Notation}
\providecommand{\remarkname}{Remark}
\providecommand{\assumptionname}{Assumption}

\newcommand{\blue}[1]{#1}
\newcommand{\var}{\,\cdot\,}
\newcommand{\TV}{\text{TV}}

\newcommand{\cA}{\mathcal{A}}

\newcommand{\cF}{\mathcal{F}}

\newcommand{\cH}{\mathcal{H}}

\newcommand{\cN}{\mathcal{N}}

\newcommand{\scL}{\mathscr{L}}

\newcommand{\EE}{\mathbb{E}}

\newcommand{\PP}{\mathbb{P}}

\newcommand{\RR}{\mathbb{R}}

\newcommand{\WW}{\mathbb{W}}

\newcommand{\bB}{\mathbf{B}}

\newcommand{\bc}{\mathbf{c}}

\newcommand{\bn}{\mathbf{n}}

\newcommand{\bw}{\mathbf{w}}

\newcommand{\fB}{\mathfrak{B}}

\newcommand{\dd}{\mathop{}\!\mathrm{d}}

\newcommand{\mcA}{\mathcal{A}}
\newcommand{\mcB}{\mathcal{B}}

\newcommand{\mcF}{\mathcal{F}}

\newcommand{\mcH}{\mathcal{H}}

\newcommand{\mcJ}{\mathcal{J}}

\newcommand{\mcP}{\mathcal{P}}

\newcommand{\mcX}{\mathcal{X}}

\newcommand{\mbE}{\mathbb{E}}

\newcommand{\mbI}{\mathbb{I}}

\newcommand{\mbP}{\mathbb{P}}

\newcommand{\mbR}{\mathbb{R}}

\newcommand{\mbW}{\mathbb{W}}

\newcommand{\mfA}{\mathfrak{A}}
\newcommand{\mfB}{\mathfrak{B}}

\newcommand{\msL}{\mathscr{L}}

\newcommand{\bmeta}{\boldsymbol{\eta}}

\newcommand{\bnu}{\boldsymbol{\nu}}

\newcommand{\eps}{\varepsilon}

\newcommand{\Lip}{\text{Lip}}

\newcommand{\tzero}{|_{t=0}}

\DeclareMathOperator*{\esssup}{ess\,sup}
\DeclareMathOperator*{\argsup}{arg\,sup}

\allowdisplaybreaks
\usepackage[T1]{fontenc}
\usepackage[utf8]{inputenc}

\usepackage{mathtools}
\let\oldtocsection=\tocsection
\let\oldtocsubsection=\tocsubsection
\let\oldtocsubsubsection=\tocsubsubsection
\renewcommand{\tocsection}[2]{\hspace{0em}\oldtocsection{#1}{#2}}
\renewcommand{\tocsubsection}[2]{\hspace{1em}\oldtocsubsection{#1}{#2}}
\renewcommand{\tocsubsubsection}[2]{\hspace{2em}\oldtocsubsubsection{#1}{#2}}
\numberwithin{thm}{section}

\newcommand{\supp}{\text{Supp}\mathop{}\!}

\usepackage{enumitem}
\setlist[description]{leftmargin=\parindent,labelindent=\parindent}

\title{Stochastic Analysis of Overlapping Generations Models Under Incomplete Markets}

\author{Cangxiong Chen \and Sigmund Ellingsrud \and Fabian N. Harang \and Alfonso Irarrazabal \and Avi Mayorcas }

\address{Cangxiong Chen:
 School of Computer Science and Statistics, Trinity College Dublin, D02 N8P5 Dublin, Ireland}
\email{cangxiong.chen@tcd.ie} 

\address{Sigmund Ellingsrud:
Department of Economics, BI Norwegian Business School, 0484 Oslo, Norway}
\email{sigmund.ellingsrud@bi.no} 

\address{Fabian A. Harang:
 Department of Economics, BI Norwegian Business School, 0484 Oslo, Norway}
\email{fabian.a.harang@bi.no}

\address{Alfonso Irarrazabal:
Department of Economics, BI Norwegian Business School, 0484 Oslo, Norway}
\email{alfonso.irarrazabal@bi.no} 

\address{Avi Mayorcas:
 Department of Mathematical Sciences, University of Bath, BA2 7AY, Bath, UK}
\email{am2735@bath.ac.uk}

\begin{document}

\begin{abstract}
  We provide a stochastic analysis of an overlapping‐generations model under incomplete markets. By casting individual optimization with idiosyncratic income risk into a forward–backward stochastic differential‐equation (FBSDE) system, we (i) establish existence and uniqueness of the dynamic general‐equilibrium interest rate and (ii) derive analytical and semi‐explicit formulas for both the equilibrium interest‐rate path and the natural borrowing limit — defined as the discounted expected shortfall of future income. Our FBSDE‐based approach yields tractable policy functions and equilibrium mappings without relying on high‐dimensional PDE methods, offering clear insights into how income dynamics and demographic structure drive interest‐rate fluctuations and credit constraints.\\[1ex]
   
\textbf{MSC2020 Subject Classification:} 91B70, 91B74, 93E20, 65K10 \\[1ex]
\end{abstract}

\thanks{Published in \emph{Mathematics and Financial Economics} \textbf{20} (2026), 405--453. \url{https://doi.org/10.1007/s11579-026-00418-5}.}

\maketitle
\tableofcontents

\section{Introduction}

Recent demographic shifts, most notably declining fertility rates, combined with rapid technological innovation, have far-reaching effects on economic inequality. These trends impact cohorts and income groups unevenly, underscoring the critical need for models that explicitly incorporate agent heterogeneity.
Since \cite{samuelson1958} and \cite{diamond1965}, the overlapping-generations (OLG) framework has been pivotal in macroeconomics, particularly for analyzing the distributional effects of demographic and technological changes. OLG models capture intergenerational dynamics in finite-lived economies, where agents optimally allocate consumption and savings over time and interact through aggregate markets.
\cite{riosrull1996_OLG} and \cite{huggett1996_OLG} further enriched the framework by introducing idiosyncratic income risks and market incompleteness. In such settings, agents cannot fully insure against individual shocks, preventing perfect risk‐pooling. Consequently, higher‐income or older cohorts self‐insure more effectively than lower‐income or younger cohorts, generating persistent cross‐sectional dispersion in wealth and consumption. Transitory labor‐income shocks thus give rise to enduring differences in lifetime welfare.

It is worth noting that overlapping‐generations models can yield theoretical predictions absent in infinitely lived‐agent frameworks. \cite{galorryder1989}, \cite{cass1972}, and \cite{shell1971} highlight equilibrium multiplicity and efficiency concerns in OLG settings. \cite{aiyagari_85_observational} shows that, under specific conditions, aggregate outcomes in OLG models coincide with those in corresponding infinitely lived‐agent models. Moreover, \cite{yaari1965} and \cite{blanchard1985} formulate infinite‐horizon representations of OLG economies with stochastic lifespans in a deterministic setting.  To our knowledge, there is little work analyzing equilibrium in stochastic OLG economies under incomplete markets and heterogeneous agents, as the resulting equilibrium conditions give rise to highly nonlinear systems of PDEs that complicate equilibrium characterization.

In this paper we introduce a novel approach to the rigorous study of continuous‐time OLG models with heterogeneous agents by combining well-established tools from \textit{stochastic control theory}, notably the Pontryagin maximum principle, and \textit{forward–backward stochastic differential equations (FBSDEs)}, \cite{pardoux_peng_90_adapted, antonelli_93_backward,yong_zhou_99_stochastic,delarue_02_existence,zhang_17_backward}. We derive explicit and semi‐explicit expressions for economic quantities of interest in these models, in particular the \emph{optimal consumption policy}, \emph{equilibrium interest rates}, \emph{natural borrowing limits} and functional derivatives of equilibrium interest rates with respect to capital supply. Our approach provides analytical tools that may prove useful for studying the interaction between demographic structure, income risk, and macroeconomic equilibrium outcomes. It offers an alternative route to equilibrium characterization relative to PDE-based techniques commonly used in continuous-time heterogeneous-agent models. By working directly with coupled forward–backward stochastic differential equations, the analysis avoids certain boundary or stationarity assumptions that are often introduced for numerical tractability in PDE formulations. \cite{huggett_93_risk_free, aiyagari_94_uninsured,achdou_han_lasry_lions_moll_21_income, ambrose_21_existence}.

 We make contributions to the economic literature in providing analysis of equilibrium interest rates in continuous time OLG models. In \cite{DemichelisPolemarchakis2007_OLG} and \cite{Edmond2008_OLG} the model is cast with fully deterministic endowments (no idiosyncratic risk). \cite{DemichelisPolemarchakis2007_OLG} construct a finite‑horizon, log‑utility example in which there is (up to time‑shift) a unique non‑stationary equilibrium price path in continuous time. \cite{Edmond2008_OLG} shows that, under log utility, the equilibrium problem can be written as a linear integral equation for the inter-temporal interest rate function, which admits a unique globally stable solution and is straightforward to approximate numerically.  Analogous existence-and-uniqueness results for OLG models with idiosyncratic income shocks and incomplete markets have not been established.  

There exists work showing existence, and in certain instances,  uniqueness of recursive equilibria in heterogeneous‐agent models with incomplete markets. The vast majority of these results pertain to discrete‐time economies (e.g., \cite{AcemogluJensen2015_JPE}). \cite{Proehl2024} contributes to this literature by incorporating aggregate risk: leveraging the monotonicity of equilibrium correspondences and tools from convex analysis, she establishes that the generalized Euler operator is maximally monotone, which yields both a convergent solution algorithm and a unique recursive equilibrium in a Bewley–Aiyagari‐style growth model subject to aggregate shocks. Building on these methods, \cite{Cao2020_JET} demonstrates the existence of extended recursive equilibria, in which policy and price functions also depend on the value function. \cite{BrummKryczkaKubler2017_Eca} further extend the analysis by proving the existence of sunspot‐dependent recursive equilibria within finite‐agent approximations.   Our work adds to this strand of literature by studying a continuous‐time overlapping‐generations model under incomplete markets, enabling us to apply stochastic‐analysis tools, specifically a forward–backward stochastic differential‐equation (FBSDE) system, to establish existence and uniqueness of the dynamic general‐equilibrium interest rate.

A second key result  of our work is the derivation of an \textit{endogenous natural borrowing limit} for individuals. In classical OLG and heterogeneous-agent models, borrowing limits are often required to prevent Ponzi schemes \cite{aiyagari_94_uninsured} and remain identical for all households. In contrast, our framework shows that a heterogeneous, natural borrowing limit arises \textit{endogenously} from agents' optimization behaviour and terminal wealth constraints. We prove that the natural borrowing limit depends on the conditional expectation of future income. This result reveals that the borrowing limit depends on the present value of expected future income, discounted by the equilibrium interest rate path. Notably, this borrowing limit is \textit{heterogeneous} across agents, as it reflects their individual income processes and wealth trajectories. This stands in sharp contrast to exogenous limits assumed in many existing models, making our framework more flexible for economic applications. Deriving model implications such as the heterogeneous natural borrowing limit is a particular advantage of our framework. 
To the best of our knowledge, this type of equilibrium-state-dependent natural borrowing limit has not been characterized analytically in standard incomplete-markets life-cycle models.

From a mathematical perspective, our work contributes to the growing literature on probabilistic methods for dynamic equilibrium analysis. While continuous-time formulations of models such as \cite{aiyagari_94_uninsured} and \cite{huggett_93_risk_free} have been considered using PDE approaches \cite{achdou_han_lasry_lions_moll_21_income,achdou_francisco_lasry_lions_moll_14_pde,ambrose_21_existence}, we adopt a purely stochastic formulation. In particular, we derive optimal consumption policies using solutions to coupled forward-backward stochastic differential equations. This approach avoids the boundary constraints required in PDE methods and provides semi-explicit solutions for consumption and wealth dynamics. Similar approaches have been made to studying mean field games \cite{delarue_02_existence,carmona_delarue_18_book_I,carmona_delarue_18_book_II} resulting in neural network based numerical algorithms which we are also inspired by \cite{carmona_lauriere2021}. 
 By treating the equilibrium interest rate as a functional of aggregate distributions, we establish a novel fixed-point argument for solving the dynamic general equilibrium problem.
We provide analytical expressions for the natural borrowing limit and prove its stability under perturbations in income processes or interest rate paths. This stability is essential for numerical implementations and economic interpretations.

Besides theoretical contributions, we also develop a numerical solution based on deep learning for finding the partial equilibrium in the life-cycle model where we hold the interest rate to be constant throughout the life cycle. Our numerical strategy resembles the direct method of \cite{carmona_lauriere2021} in that we restrict controls to a neural network class and solve a Monte Carlo stochastic optimization problem rather than discretizing the associated HJB equation. However, unlike their McKean--Vlasov control and mean field formulations, our implementation considers a single-agent life-cycle problem in partial equilibrium, without law-dependent dynamics or a fixed-point condition over the population distribution. Besides providing validation to the theories we have developed in this work, we believe that our approach using neural networks might be of independent interest to researchers working on building economic models using dynamical systems.

\textbf{Structure of the paper:} The remainder of the paper is organized as follows. Section~\ref{sec:model_setup} introduces the overlapping generations model in mathematical terms and describes the main objects we will work with in the remainder of the paper. Section~\ref{sec:life_cycle_partial_equilibria} contains the main workhorse result of the paper, focussing on a single life-cycle model,  Theorem~\ref{thm:life_cycle_optimal_rep}, gives an FBSDE representation of the optimal consumption policy in the one generation case. In addition, Section~\ref{sec:life_cycle_partial_equilibria} contains a presentation of the deterministic life-cycle model, an a priori natural borrowing limit and proofs of existence and uniqueness of an optimal solution to the stochastic life-cycle model under suitable assumptions. Section~\ref{sec:life_cycle_gen_eq} demonstrates important stability results on the optimal consumption policy in the life-cycle model and as a result demonstrates existence and uniqueness of a general equilibrium interest rate in the life-cycle case. Finally, Section~\ref{sec:OLG} contains the most economically important results of the paper, combining the work of previous sections to analyse the overlapping generations model. Here we provide proofs of existence and uniqueness of the optimal consumption policies for each generation as well as existence and local uniqueness of a general equilibrium interest rate under certain reasonable assumptions on the coefficients and model parameters. Section~\ref{sec:numerics} presents numerical simulations of partial equilibria in the life-cycle model.

\subsection{Notation}
Given a probability space $(\Omega,\mcF,\mbP)$, a normed vector space $(\mcX,\|\,\cdot\,\|_{\mcX})$ and a Borel measurable map $X:\Omega \to \mcX$ we define the family or norms
    \begin{equation*}
    \|X\|_{\scL^p_\omega}\coloneqq \begin{cases}
        \mbE[\|X\|^p_{\mcX}]^{\nicefrac{1}{p}}, & p\in [1,+\infty),\\
        \esssup_{\omega\in \Omega} \|X(\omega)\|_\mcX, & p=+\infty,
    \end{cases} 
\end{equation*}
and let $\msL^p(\Omega;\mcX)$ denote the space of all Borel measurable maps $X:\Omega\to \mcX$ such that $\|X\|_{\msL^p_\omega}<\infty$. We also set
\begin{equation*}
    \msL^0(\Omega;\mcX) \coloneqq \{X: \Omega \to \mcX\,:\, X \, \text{ is Borel measurable }\}.
\end{equation*}
Similarly, given an interval $I\subseteq \mbR$ and a measurable map $f:I\to \mcX$ we define the family of norms
  \begin{equation*}
    \|f\|_{\scL^p_I }\coloneqq \begin{cases}
       \left( \int_I \|f(t)\|^p_{\mcX}\dd t\right)^{\nicefrac{1}{p}}, & p\in [1,+\infty),\\
        \esssup_{t\in I} \|f(t)\|_\mcX, & p=+\infty,
    \end{cases} 
\end{equation*}
and set $\msL^p(I;\mcX)$ to be the space of all measurable maps $f:I\to \mcX$ such that $\|f\|_{\msL^p_I} <\infty$. For convenience, when $|I|= L$ for some $L>0$ and when it will not cause confusion we simply write $\|f\|_{\msL^p_L}$ for the norm on $\msL^p([0,L];\mcX)$. In the same setting we write $C_b(I;\mcX)$ for the space of bounded continuous maps $f:I\to \mcX$ equipped with the norm $\|f\|_{\msL^\infty_I}$ and $C^1_b(I;\mcX)$ for the space of bounded, continuously differentiable functions $f:I\to \mcX$ equipped with the norm
\begin{equation*}
    \|f\|_{C^1_I}\coloneqq \|f\|_{\msL^\infty_I} + \|f'\|_{\msL^\infty_I}.
\end{equation*}
We write $\text{Lip}(I;\mcX)$ for the space of all bounded, Lipschitz continuous maps $f:I\to \mcX$ equipped with the norm
\begin{equation}\label{eq:Lip_norm_def}
    \|f\|_{\Lip_I} \coloneqq \|f\|_{\msL^\infty_I}+ \esssup_{t\neq s \,\in I} \frac{\|f(t)-f(s)\|_{\mcX}}{|t-s|}. 
\end{equation}

\section{Overlapping Generations Model Setup}\label{sec:model_setup}

In this section we outline a continuous time overlapping generations general equilibrium model where agents live for a given finite interval of time and face stochastic and independent income shocks. Markets are incomplete in the sense that households can only self‐insure against bad shocks by trading a single risk‐free bond.  In equilibrium, aggregate savings must be equal to an exogenous capital supply. In the following Section~\ref{sec:life_cycle} we study a single cohort version of this model referred to as the \textit{life-cycle} model.

We fix a lifespan $L>0$, a probability space $(\Omega,\mcF,\mbP)$ which carries an uncountable family of i.i.d  Brownian motions $\{B^b\}_{b\in \mbR}$ where each $B^b$ is shifted so that $B^b_b =0$. Let $\{\mcF^b_t\}_{t\in [b,b+L]}$ be the augmented natural filtration associated to $B^b|_{[b,b+L]}$. Then, for a given continuous path 
\begin{equation*}
    \mbR \ni t\mapsto r_t \in \mbR_+,
\end{equation*}
for each $b\in \mbR$ we describe a household by its wealth and income which are modelled according to the dynamics
\begin{align}
        w^b_t(r) &= w^b_b + \int_b^t r_s w^{b}_s \dd s + \int_b^t \eta^b_s  \dd s - \int_b^t c_s^b \dd s, \quad t\in [b,b+L]\label{eq:olg_wealth_intro}
        \\
\eta^b_t &= \eta^b_b +\int_{b}^t \mu_s(\eta^b_s) \dd s + \int_{b}^t \sigma_s(\eta^b_s) \dd B^b_s, \quad t\in [b,b+L] \label{eq:olg_income_intro}
\end{align}
where $\eta^b_b \sim \rho^\eta_b$, $w^b_b \sim \rho^w_b$ for families of laws $\{\rho^\eta_b\}_{b\in \mbR} \subset \mcP(\mbR_+),\, \{\rho^w_b\}_{b\in \mbR} \subset \mcP(\mbR)$ and
\begin{equation*}
    \mu:\mbR_+\times \mbR \times \Omega \to \mbR \quad \text{and}\quad \sigma:\mbR_+\times \mbR \times \Omega \to \mbR,
\end{equation*}
  are suitable coefficients such that for each $b\in \mbR$, there exists a unique strong solution $\eta^b$ to \eqref{eq:olg_income_intro}.\footnote{One could in principle allow $\mu,\,\sigma$ to depend on $b\in \mbR$ as well, provided each $\mu^b,\,\sigma^b$ satisfy the necessary requirements and their dependence on $b$ is sufficiently regular. For simplicity of presentation we do not consider this case here but no fundamental change would be required to obtain results analogous to ours.} Each realisation $B^b(\omega)$ for $\omega\in \Omega$ represents the possible random shocks experienced by a household born at $b\in \mbR$ and the economy is made up of infinitely many of these heterogeneous households. We will use the terms household and individual interchangeably throughout the paper.

\begin{rem}[Measurability of $b\mapsto (\eta^b,w^b,c^b)$]\label{rem:fubini_measurability}
 In the model set-up above we have glossed over an important issue concerning measurability. Recall that in general mutual independence and joint immeasurability cannot hold simultaneously for an uncountable collection of random variables with positive variance, \cite[Thm.~2.2]{doob_37_continuous} \cite{judd_85_continuum} and \cite[Prop.~2.1]{sun_06_exact}. The resolution, given by \cite[Thm~2.8]{sun_06_exact} is to work with a Fubini extension of the naive product measurable space $\mbR \times \Omega$ which retains the exchangeability of integration, see also \cite{podczeck_10_fubini}. We silently implement the Fubini extension throughout the paper, only commenting on it when relevant, see also Remark~\ref{rem:fubini_lln}. 
\end{rem}

Each household chooses their consumption $t\mapsto c^b_t$ to maximize an expected running discounted utility over the interval $[b,b + L]$ and their expected utility of wealth at the terminal time $b + L$. Either the desire for a pension fund, as in the Merton model, or a bequest for future generations motivates the expected utility of terminal wealth. To formulate the optimization problem in mathematical terms, we take running and terminal concave utility functions $u_1,\,u_2:\RR\rightarrow \RR\cup\{-\infty\}$ and define the \emph{discounted payoff} of an individual born at time $b\in \mbR$, at time $t\in [b,b+L]$ with wealth $w^b$, income $\eta^b$ and consumption policy $c^b$ by the expression
\begin{equation}\label{eq:olg_payoff_discounted}
   \tilde{ \mcJ}^b_{L}\big(t,w^b,\eta^b|c^b\big)=e^{-\delta (b+L-t)}\EE\left[\int_{t}^{b+L} e^{\delta(b+L- s)}u_1\big(c^b_s\big)\dd s + \lambda \, u_2\big(w^b_{b+L}\big)\right], \quad t\in [b,b+L],
\end{equation}
where $\delta\geq 0$ is a common discount factor and $\lambda\geq 0$
 measures the bequest motive's intensity.\footnote{As with $\mu,\,\sigma$, there is no fundamental mathematical challenge to allowing $\delta,\,\lambda$ to depend on $b\in \mbR$. Provided this dependence is sufficiently regular and bounded then our main results would apply with only minor, cosmetic change.} When it will not cause confusion we set
 \begin{equation*}
 \tilde{ \mcJ}^b_{L}\big(w^b,\eta^b|c^b\big) \coloneqq  \tilde{ \mcJ}^b_{L}\big(b,w^b,\eta^b|c^b\big).
 \end{equation*}
 We define the admissible family of consumption paths as
 \begin{equation*}
     \mcA^b \coloneqq \left\{c^b:[b,b+L]\to \mbR\,:\, \text{progressively } \{\mcF^{b}_t\}_{t\in [b,b+L]}\text{-measurable and continuous} \right\}.
 \end{equation*}
 Noting that a utility maximising consumption policy for $\tilde{ \mcJ}^b_L$ also maximises utility for $C\tilde{ \mcJ}^b_L$ given any $C>0$. For mathematical convenience we will, therefore, actually work with the \emph{inflated payoff},
 \begin{equation}\label{eq:olg_payoff_inflated}
 \mcJ^b_L\big(t,w^b,\eta^b|c^b\big)=e^{\delta(b+L-t)}\tilde{\mcJ}^b_L(t,w^b,\eta^b|c^b) =\EE\left[\int_t^{b+L} e^{\delta(b+L- s)}u_1\big(c^b_s\big)\dd s + \lambda \, u_2\big(w^b_{b+L}\big)\right]. 
 \end{equation}

\begin{rem}[CRRA Utility Functions]
    A common example of utility functions to keep in mind are those known as the \emph{constant-relative-risk-aversion (CRRA)}  utility functions. For $\gamma \in \RR_+ \setminus \{1\}$, these are given by the expression
\begin{equation}\label{eq:crra_utility}
    u(x)= \begin{cases}
             \frac{ x^{1-\gamma}}{1-\gamma}, \quad \mathrm{for} \quad &x \geq 0,
             \\
             -\infty \quad &x<0.
    \end{cases}
\end{equation} 
These functions are concave with unbounded derivative at $x=0$, this models an individual's strong desire to consume something rather than nothing (or hold any positive amount of wealth over no wealth). This unbounded gradient, however, poses inevitable mathematical challenges. In subsequent sections we discuss approximations to \eqref{eq:crra_utility} which retain its central features, see Example~\ref{ex:quad_crra_approx}. Note that in \eqref{eq:olg_payoff_discounted} we allow for differing running and terminal utilities $u_1,\,u_2$. In the case of both being CRRA this would be achieved by choosing two parameters $\gamma_1,\,\gamma_2 >0$.
\end{rem}

When each individual optimizes the discounted payoff function given in \eqref{eq:olg_payoff_inflated} with respect to consumption $c^b$, they find the \emph{discounted value function}
\begin{equation}\label{eq:value_function_1}
    v^b_L\big(t,w^b,\eta^b\big)=\sup_{c^b\in \cA^b} \, \mcJ^b_L\big(t,w^b,\eta^b|c\big),\quad t\in [b,b+L].  
\end{equation}
Furthermore, we define
\begin{equation}\label{eq:optimal consumption}
    c^{b;*}_t = \argsup_{c^b\in \cA^b} \, \mcJ^b_L\big(t,w^b,\eta^b|c^b\big),\quad t\in [b,b+L],
\end{equation}
as the optimal path of consumption for the agent born at $b\in \mbR$. The evolution of optimally controlled wealth, $w^{b;*}$, is therefore given by the equation

\begin{equation}\label{eq:optimal_wealth}
    \dd w_t^{b;*} =\big(r_t w_t^{b;*}-c_t^{b;*}+\eta^*_t\big)\dd t,\quad w^{b;*}_b \sim \rho^w_b.
\end{equation}
\begin{rem}\label{rem:conumption_dependence}
    Note that finding the optimal consumption $c^{b;*}$ for each $b\in \mbR$ as prescribed by \eqref{eq:value_function_1} gives a functional pathwise dependence of $c^{b;*}$ on $w^{b}$ and the interest rate $r$. In principle it also depends on the income process $\eta^b$ but this dependence is encoded by $w^b$. Furthermore, since the wealth dynamic also depends on the interest rate path, we can re-write \eqref{eq:optimal_wealth} parsimoniously as
    \begin{equation}\label{eq:optimal_wealth_dependence}
    \dd w_t^{b;*} =\big(r_t w_t^{b;*}-c_t^{b;*}\big(w^{b;*}(r),r\big)+\eta^*_t\big)\dd t,\quad w^{b;*}_b \sim \rho^w_b.
\end{equation}
We will later give a semi-explicit expression for this functional dependence, which reflects the known explicit solution to the related life-cycle model with deterministic income in Section~\ref{sec:deterministic_life_cycle}.
\end{rem}

Since total capital is fixed in our model, the interest rate must be chosen to keep aggregate wealth constant. In the OLG model this requires us to integrate total wealth over a probability measure describing the chance of being born at each time. To this end we introduce the notion of a \emph{flow of demographic measures}
\begin{equation}
    \mbR \ni t \mapsto \nu_t \in \mcP([t-L,t]),
\end{equation}
where for any Borel set $A\in \mcB([t-L,t])$ the quantity $\nu_t(A) \in [0,1]$ describes the probability that a randomly chosen individual alive in the economy at time $t\in \mbR$ was born in $A$. Hence, the aggregate wealth, income and consumption are written as follows, taking account of the functional dependencies described in Remark~\ref{rem:conumption_dependence},
\begin{align}
    W^L_t(r) &\coloneqq \int_{t-L}^{t} w_t^{b}(r) \dd \nu_t(b)\label{eq:aggregate_wealth_intro}\\
    C^L_t(w,r) &\coloneqq \int_{t-L}^t c^{b}_t(w^b,r) \dd \nu_t(b) , \label{eq:aggregate_consumption_intro} \\
    N^L_t &\coloneqq \int_{t-L}^t \eta^{b}_t \dd \nu_t(b). \label{eq:aggregate_income_intro} 
\end{align}
In this setting, the notion of a general equilibrium is defined as follows.

\begin{defn}\label{def:intro_olg_clearing_rate}
    Given a \emph{lifespan} $L>0$ and a capital supply $K\in C(\RR;\RR)$  we say that the associated overlapping generations model \eqref{eq:olg_wealth_intro}-\eqref{eq:olg_payoff_discounted} and \eqref{eq:aggregate_wealth_intro}-\eqref{eq:aggregate_income_intro} is in \emph{general equilibrium} if the interest rate $r:\mbR\to \mbR$ is such that
    \begin{equation}\label{eq:olg_clearing_intro}
        \mbW^L_t(r) \coloneqq \mbE\left[W^L_t(r)\right] = K_t \quad \text{for all}\,\, t\in \mbR.
        \end{equation}
        For brevity we sometimes say that $r$ is a \emph{general equilibrium interest rate} for the associated overlapping generations model.
\end{defn}

\begin{rem}\label{rem:fubini_lln}
 As described in Remark~\ref{rem:fubini_measurability} we are silently appealing to the Fubini extension of \cite{sun_06_exact}. As a result, when $\nu_t \sim \text{Leb}([t-L,t])$ an extension of the usual law of large numbers holds and we actually have
    \begin{equation}
        W^L_t(r) = \int_{t-L}^t w^b_t(r) \dd \nu_t(b) = \int_{t-L}^t \mbE[w^b_t(r)] \dd \nu_t(b) = \mbW^L_t(r).
    \end{equation}
    Since this equality does not hold in the single life-cycle case or if $\nu_t \not \sim \text{Leb}([t-L,t])$ we retain the distinction between the objects $W^L_t(r)$ and $\mbW^L_t(r)$ throughout.
\end{rem}

\begin{rem}
The path $t\mapsto K_t$ represents the capital supply of the economy at time $t$ and is given exogenously in the current paper for simplicity, and is often set to $0$ in many applications. It provides a relationship with the amount of capital that is borrowed and the amount that must be saved. In economic theory, the capital supply $K$ is often determined endogenously through a production technology, see e.g. \cite{aiyagari_94_uninsured} and \cite{krusell_smith_98_income} for details on this. For simplicity of presentation we keep the production function abstract here, simply positing an exogenous capital flow, but plan to investigate the combined setting in more detail in future works.  
\end{rem}

\begin{rem}
   As noted in Remark~\ref{rem:conumption_dependence}, the optimal consumption is a pathwise functional of the wealth and interest rate processes. Thus, the closure condition \eqref{eq:olg_clearing_intro} exhibits a dependence of the \emph{general equilibrium} interest rate on both the law of the family of processes $\{w^b\}_{b\in \mbR}$ and the \emph{flow of demographic measures}. In this regard, our general equilibrium problem is closely related to the setting of mean field control. Agents do not select their consumption policy $c^{b;*}$ in order to advantageously affect the interest rate $r$, however, the interest rate must be set in order to attain a desired distribution of wealth. However, our model does not encode a notion of optimal wealth distribution, as would be common in a genuine mean field control setup. Instead, the general equilibrium condition is simply a fixing of the scenario and an important economic output of the model. As discussed in \cite{carvalho2016} there is evidence that demographic shift has a measurable impact on the real interest rate. This is something that we aim to investigate computationally.
     
Solving the general equilibrium problem for \eqref{eq:olg_wealth_intro}-\eqref{eq:olg_payoff_inflated} is equivalent to the following procedure:
\begin{enumerate}[label=\roman*)]
    \item \label{it:fix_rate} Fix a deterministic interest rate $r\in C(\mbR;\RR)$ 
    \item \label{it:optimal_solve} For each $b\in \mbR$, solve the optimal control problem 
    \begin{equation*}
        \begin{cases}
            \sup_{c^b \in \mcA^b} \mcJ^b_L(w^b,\eta^b|c^b)
            \\
            \dd w^b_t = (r_tw^b_t-c^b_t+\eta^b_t)\dd t , \quad w^b_b \sim \rho^w_b,
            \\
            \dd \eta^b_t = \mu_t(\eta^b_t)\dd t +\sigma_t(\eta^b_t) \dd \beta_t, \quad \eta^b_b \sim \rho^\eta_b.
        \end{cases}
    \end{equation*}
    \item \label{it:gen_eq_rate} View the optimal solution $(w^{b;*},c^{b;*})$ as a functional $r\mapsto  (w^{b;*}(r),c^{b;*}(r))$ and show that for a given flow of demographic measures $t\mapsto \nu_t$, there exists a unique continuous path $\bar{r}$ such that 
    \begin{equation*}
    	\mbE\left[ \int_{t-L}^t w^{b;*}_t(\bar{r}) \dd \nu_t(b)\right] = K_t,\quad \text{for all }\, t\in \mbR.
    \end{equation*}
\end{enumerate}
We complete these three steps under a number of reasonable regularity and structural assumptions on the coefficients involved as well as a restriction to \emph{local} uniqueness of the general equilibrium interest rate in the space of all continuous paths. As already mentioned, along the way we also derive explicit or semi-explicit expressions for a number of economically relevant aggregates and outputs of the model.
\end{rem}
Since Step~\ref{it:optimal_solve} is identical (up to shifting the domain) for each $b\in \mbR$, our first focus will be on the so-called \emph{life-cycle} model which considers an economy of one generation living in isolation on $[0,L]$ where each household solves the same optimal control problem. 

\subsection{Main Results} We informally summarise the main mathematical results of our paper, with references to precise statements and proof given in subsequent sections. These mathematical results are supplemented with numerical examples and experiments in Section~\ref{sec:numerics}.

 We appeal to two kinds of assumptions throughout the paper. The first, stated formally below, imposes regularity and boundedness assumptions on the utility functions $u_1$ and $u_2$. The second kind, which we do not state formally as they vary in detail but are of the same kind, allow for singularities in the utilities, specifically in $(u'_1)^{-1}(x)$ and $u'_2(x)$ as $x\searrow 0$.
\begin{ass}\label{ass:utility_reg_intro}
     The utility functions $u_1$ and $u_2$ are concave and such that $(u_1')^{-1}$ and $u_2'$ are  both Lipschitz continuous and such that there exists a $\kappa>0$ for which
\begin{equation}
\begin{aligned}
     &\inf_{x\in \mbR} u'_2(x) \geq 0, \quad \inf_{x\in \mbR}  (u'_1)^{-1}(x) \geq 0, \quad  \sup_{y \in [0,+\infty)} |(u'_1)^{-1} (y)|\vee |u'_2(y)| \leq \kappa,\\
    & \sup_{x\neq y \in \mbR} \left(\frac{|(u'_1)^{-1}(x)-(u'_1)^{-1}(y)|}{|x-y|} \vee \frac{|u'_2(x)-u'_2(y)|}{|x-y|} \right)<\kappa.
\end{aligned}
\end{equation}
\end{ass}

Our main result concerns the OLG model and is stated somewhat informally below. A full statement can be found as Theorem~\ref{thm:olg_well_posed}.
\begin{thm}[Partial and General Equilibria of the OLG Model]
Let $\delta,\, \lambda,\, \kappa >0$, $u_1,\, u_2$ satisfy Assumption~\ref{ass:utility_reg_intro} for the same $\kappa$, $\{\rho^w_b\}_{b\in \mbR}\subset \mcP(\mbR),\, \{\rho^\eta_b\}_{b\in \mbR}\subset \mcP(\mbR_+)$ and 
\begin{equation*}
    \mu:\mbR_+\times \mbR \times \Omega \to \mbR \quad \text{and}\quad \sigma:\mbR_+\times \mbR \times \Omega \to \mbR,
\end{equation*}
    be such that for every $b\in \mbR$ there exists a unique strong solution to the income equation \eqref{eq:olg_income_intro}. Then, if $L>0$ is sufficiently small as a function of all other relevant parameters,
    \begin{enumerate}[label=\roman*)]
        \item for every $b\in \mbR$ there exists a unique, solution to the optimal control problem 
        \begin{align}
            \sup_{c^b \in \mcA^b} &\mcJ_L^b(w^b,\eta^b|c^b) = \EE\left[\int_b^{b+L} e^{\delta(b+L- s)}u_1\big(c^b_s\big)\dd s + \lambda \, u_2\big(w^b_{b+L}\big)\right], \label{eq:olg_intro_payoff}\\
\text{Subject to:} \qquad &\begin{aligned}
    w^b_t(r) &= w^b_b + \int_b^t r_s w^{b}_s \dd s + \int_b^t \eta^b_s  \dd s - \int_b^t c_s^b \dd s,  &&t\in [b,b+L],
        \\
\eta^b_t &= \eta^b_b +\int_{b}^t \mu_s\big(\eta^b_s\big) \dd s + \int_{b}^t \sigma_s\big(\eta^b_s\big) \dd B^b_s,  &&t\in [b,b+L],
\end{aligned} \label{eq:olg_intro_dynamics}
        \end{align}
with $w^b_b \sim \rho^w_b$ and $\eta^b_b \sim \rho^\eta_b$. In addition, for each $b\in \mbR$ the optimal consumption policy is given by the expression, for $t\in [b,b+L]$
\begin{equation}\label{eq:olg_intro_consumption}
    c_t^{b;*}(w^b,r) = (u_1')^{-1}\left(\lambda \exp\left(\int_t^{L+b} (r_u-\delta)\dd u\right)\EE\left[ u'_2\left(w_{b+L}^{b}\right)|\cF_t^b \right]\right).
\end{equation}
        \item given $R>0$ and $\{\nu_t\}_{t\in \mbR}$ a \emph{flow of demographic measures} there exists a constant $C>0$ depending on all parameters such that there exists a unique \emph{general equilibrium interest rate} (in the sense of Definition~\ref{def:intro_olg_clearing_rate}) $\bar{r}$, for a constant capital supply $K \in \mbR\setminus\{0\}$, in the set
        \begin{equation}
            \bigg\{ r:\mbR\to \mbR_+\,:\, \sup_{t\in \mbR}|r_t| \leq R+C
            \bigg\}.
        \end{equation}
    \end{enumerate}
\end{thm}
A rigorous statement of this result and its proof are given as Theorem~\ref{thm:olg_well_posed}. In addition, under Assumption~\ref{ass:utility_reg_intro} we obtain, for example
\begin{itemize}
    \item Marshallian and Hicksian stability of optimal wealth and consumption, Proposition~\ref{prop:life_cycle_consumption_stabillity} and Lemma~\ref{lem:LC_rate_stable}.
    \item Local existence and uniqueness of solutions and general equilibria to the life-cycle model, Proposition~\ref{prop:life_cycle_wealth_stability},  Corollary~\ref{cor:life_cycle_well_posed} and Corollary~\ref{cor:life_cycle_equilibrium}.
    \item Functional and differential representations of the general equilibrium interest rate for the life-cycle model, Proposition~\ref{prop:life_cycle_rate_props} and Proposition~\ref{prop:life_cycle_interest_rate_derivative}.
    \item Existence, uniqueness and stability of solutions and general equilibria under locality assumptions for the overlapping generations model, Theorem~\ref{thm:olg_well_posed} and Proposition~\ref{prop:olg_rate_stable}.
\end{itemize}
A number of these results rely on choosing certain parameter models, notably $\kappa,\,\lambda$ or $L$ sufficiently small (see for example Corollary~\ref{cor:life_cycle_well_posed}, Corollary~\ref{cor:life_cycle_equilibrium} and Theorem~\ref{thm:olg_well_posed}). While we do not expect these restrictions to be sharp, they serve as minimal guarantees for stability and well-posedness of the models in question. We leave it to future work to relax these requirements and find optimal thresholds for both abstract solutions and concrete numerical approximations.

In the special case when utility functions $u_1$ and $u_2$ are given by CRRA utility functions of the form $\sim \frac{x^{1-\gamma}}{1-\gamma}$, which do not satisfy Assumption~\ref{ass:utility_reg_intro}, the model simplifies for a priori solutions and analytical expressions are easier to derive, conditioned on existence of solutions in this case. In particular,  without enforcing an exogenous borrowing constraint, we can show that  all agents experience an idiosyncratic natural borrowing limit, as described by our second main result. 
\begin{thm}[Natural Borrowing Limit in the OLG Model]\label{thm:olg_natural_borrowing_limit}
Let $r:\mbR\to \mbR_+$ be fixed, $\gamma_1,\,\gamma_2>0$, 
    \begin{equation}
        u_1(x) \coloneqq \begin{cases}
            \frac{1}{1-\gamma_1}x^{1-\gamma_1}, & x\geq 0,\\
            -\infty, & x<0,
        \end{cases}
        \qquad 
        u_2(x) \coloneqq \begin{cases}
            \frac{1}{1-\gamma_2}x^{1-\gamma_2}, & x\geq 0,\\
            -\infty, & x<0,
        \end{cases}
    \end{equation}
    and for any $b\in \mbR$, $w^{b;*}$ be the optimal wealth path solving \eqref{eq:olg_intro_payoff}-\eqref{eq:olg_intro_dynamics} with optimal consumption policy given by \eqref{eq:olg_intro_consumption}. Then, for any optimal wealth policy it must hold that
   \begin{equation}
       w^{b;*}_t \geq \, - \int_t^{b+L} \exp\left(-\int_t^s r_u \dd u\right)\EE\big[ \eta^b_s|\cF^b_t\big]\dd s, \quad \PP-a.s, \quad \text{for all } t\in [b,b+L].
   \end{equation}
\end{thm}
\begin{proof}
This result is proved in the life-cycle case below, see Proposition~\ref{prop:natural_borrowing_limit}. Extending this result $b$-wise proves Theorem~\ref{thm:olg_natural_borrowing_limit}.
\end{proof}

In addition, we are able to show the following a priori properties of both the life cycle and overlapping generations models under singular assumptions on the utility functions.
\begin{itemize}
    \item Markovian representation of consumption under regularity assumptions on solutions to the associated Hamilton--Jacobi--Bellman equation, Proposition~\ref{prop:consumption_markovian}.
    \item Asymptotic behaviour of wealth as $|r|\to +\infty$ under a singularity assumption on $(u'_1)^{-1}$, Proposition~\ref{prop:range_of_expected_w}.
    \item Existence of constant general equilibrium interest rates in the case of constant demography under the assumptions of Proposition~\ref{prop:range_of_expected_w}.
\end{itemize}

\section{Partial Equilibria in the Continuous Time Life-Cycle Model}\label{sec:life_cycle_partial_equilibria}

 In this section we analyse the dynamics of life-cycle models  in partial equilibrium, i.e. for a given interest rate path $r$, with stochastic income given as an It\^o-diffusion. Our analysis derives optimal consumption rates based on the stochastic maximum principle through solving a system of FBSDEs. As we are proposing an alternative approach to the conventional PDE based setup for consumption/saving problems in macroeconomics, we will provide a detailed and pedagogical description of the model, and how to solve FBSDEs. While our main novelty lies in the analysis of the overlapping generations model, building up a solid micro foundation, with individuals from the life-cycle is crucial, and understanding conditions for existence, uniqueness and stability of this model will be central in our subsequent analysis. 

Throughout this section we fix $b=0$. However, all results apply for any $b\in \mbR$ up to suitably shifting the domain of definition and estimated regions commensurately. The precise problem we consider is as follows. Given a probability space $(\Omega,\mcF,\mbP)$ carrying a standard Brownian motion $B$ with natural filtration $\{\mcF_t\}_{t\geq 0}$ and we fix
\begin{equation}\label{eq:life_cycle_set}
     \mcA \coloneqq \left\{c:[0,L]\to \mbR\,:\, \text{progressively } \{\mcF_t\}_{t\in [0,L]}\text{-measurable and continuous} \right\}.
 \end{equation}
Then, find
\begin{equation}\label{eq:life_cycle_problem}
\begin{aligned}
     &\sup_{c\in \mcA}\mcJ_L(w,\eta|c)=\EE\left[\int_0^L e^{\delta(L- s)}u_1(c_s)\dd s + \lambda \, u_2(w_L)\right],\\
      &\text{subject to}  \qquad \begin{aligned}
        w_t &=w_0 + \int_0^t (r_s w_s-c_s+\eta_s)\dd s,
        \\
         \eta_t &= \eta_0 + \int_0^t \mu_s(\eta_s) \dd s +\int_0^t \sigma_s(\eta_s) \dd \beta_s,
    \end{aligned} \qquad \text{for all }\,\, t\in [0,L].
\end{aligned}
\end{equation}
\subsection{Deterministic Income Model: a Benchmark Special Case}\label{sec:deterministic_life_cycle}

We provide a short overview on solving continuous time life-cycle models in the special case with a deterministic, continuous income process, i.e. $\sigma \equiv 0$. This provides results for direct comparison with the stochastic setting, where uncertain income forces continuous updating of the optimal choice of consumption according to the random behaviour of income. 

\begin{prop}\label{prop:determ}
 Consider the special case of the optimal control problem \ref{eq:life_cycle_problem} with $b=0$, $\sigma(x) \equiv 0$, an interest rate path $r\in C([0,L];\RR)$ and degenerate distributions $\rho^w_0 = \delta_{w_0},\,\rho^\eta_0 = \delta_{\eta_0}$ for some $w_0\in \mbR,\, \eta_0 \in \mbR_+$. Then, the optimal consumption policy $c^* $ is found by evaluating
\begin{equation}\label{det_c_wT}
    c_t^* =  (u_1')^{-1}\left( \lambda e^{ \int_t^L (r_s-\delta)ds}  u_2'(w_L) \right).  
\end{equation}
\end{prop}
\begin{proof}

Since in this case, the problem reduces to maximize the payoff 
   \begin{equation}\label{eq:max deterministic}
   \int_0^L e^{-\delta s} u_1(c_s) + \lambda e^{-\delta L}u_2(w_L)\dd s, 
\end{equation}
subject to the deterministic budget constraint for wealth $w$
\begin{equation}\label{det_bc}
    \dot{w}_t = r_t w_t + \eta_t - c_t,\quad w_0\in \RR,
\end{equation}
the result is classical and found using the deterministic Pontryagin maximum principle, see for example \cite[Sec. 10]{sydsaeter2008further}.  
\end{proof}
To illustrate this proposition, we provide an example where the utility functions are chosen to be CRRA and $\eta_t = \eta_0\exp(\mu t)$ for some parameter $\mu>0$. 
\begin{ex}\label{ex:deterministic_crra}
    Let now $u_1(x)=u_2(x)=\frac{x^{1-\gamma}}{1-\gamma}$. Then
\begin{equation}\label{det_c_forms}
     c_t^* =  \lambda^{-\frac{1}{\gamma}} e^{-\frac{1}{\gamma} \int_t^L (r_s-\delta)ds }  w^*_L. 
\end{equation}
To find the corresponding optimal dynamics for the wealth process $\{w^*_t\}_{t\in [0,L]}$, which due to the constraint  \eqref{det_bc}, and employing the variation of constant technique,  must satisfy
\begin{equation*}
    w_t^* = e^{\int_0^t r_s \dd s }w_0 +\int_0^t e^{\int_s^t r_u \dd u } (\eta_0\exp(\mu s)-w^*_L\lambda^{-\frac{1}{\gamma}} e^{-\frac{1}{\gamma} \int_s^L (r_u-\delta)\dd u })\dd s.  
\end{equation*}
Define now $\Xi_t = \int_0^t e^{\int_s^t r_u \dd u } \eta_0\exp(\mu s)\dd s$, and $\Theta_t =\int_0^t e^{\int_s^t r_u \dd u }\lambda^{-\frac{1}{\gamma}} e^{-\frac{1}{\gamma} \int_s^L (r_u-\delta)\dd u }\dd s$. 
By elementary algebraic computations,
evaluating $t=L$ and rearranging we see that
\[
    w_L^*(1 + \Theta_L) = e^{\int_0^Lr_s \dd s }w_0 +\Xi_L.  
\]
It follows that $w_L^*$ is given by 
\begin{equation*}
    w_L^* = \frac{e^{\int_0^Lr_s \dd s }w_0+\Xi_L}{1+\Theta_L}.  
\end{equation*}
Inserting $w_L^*$ into the dynamics of $c^*$ yields a closed form consumption policy.
It's then possible to use the closed form optimal consumption policy to derive the path of optimal wealth.

For illustration, we compute the closed form solutions numerically and plot them in Figure \ref{fig:analytic-solution}.
\begin{figure}[!htbp]
  \scriptsize
  \centering
  \begin{subfigure}[t]{0.33\textwidth}
    \centering
    \includegraphics[width=\linewidth]{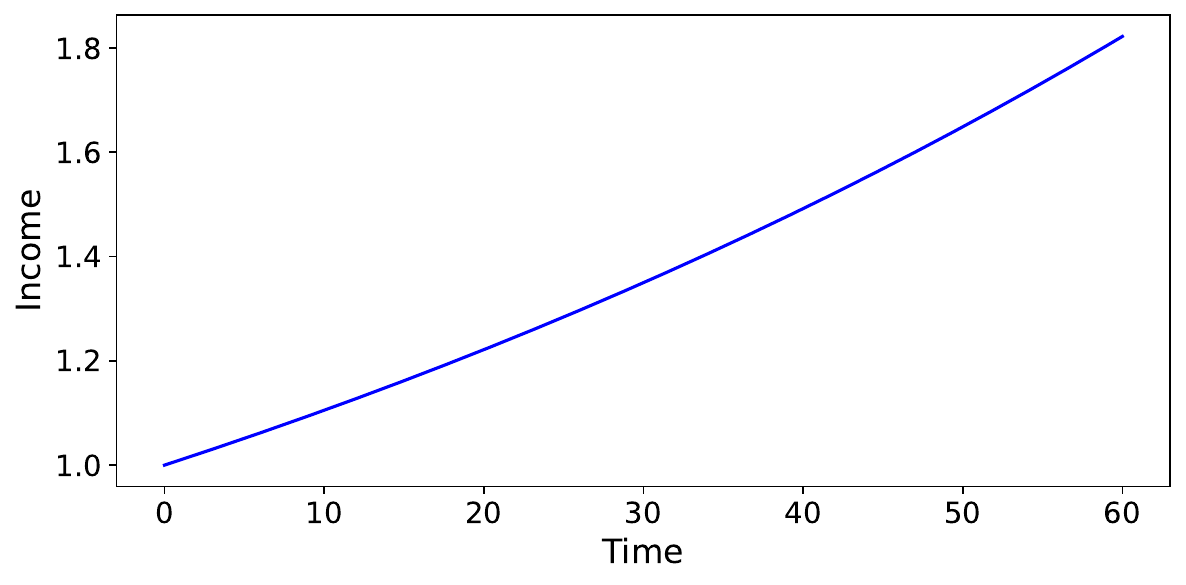}
    \caption{Income trajectory.}
    \label{fig:analytic-solution-income}
  \end{subfigure}
  \hfill
  \begin{subfigure}[t]{0.33\textwidth}
    \centering
    \includegraphics[width=\linewidth]{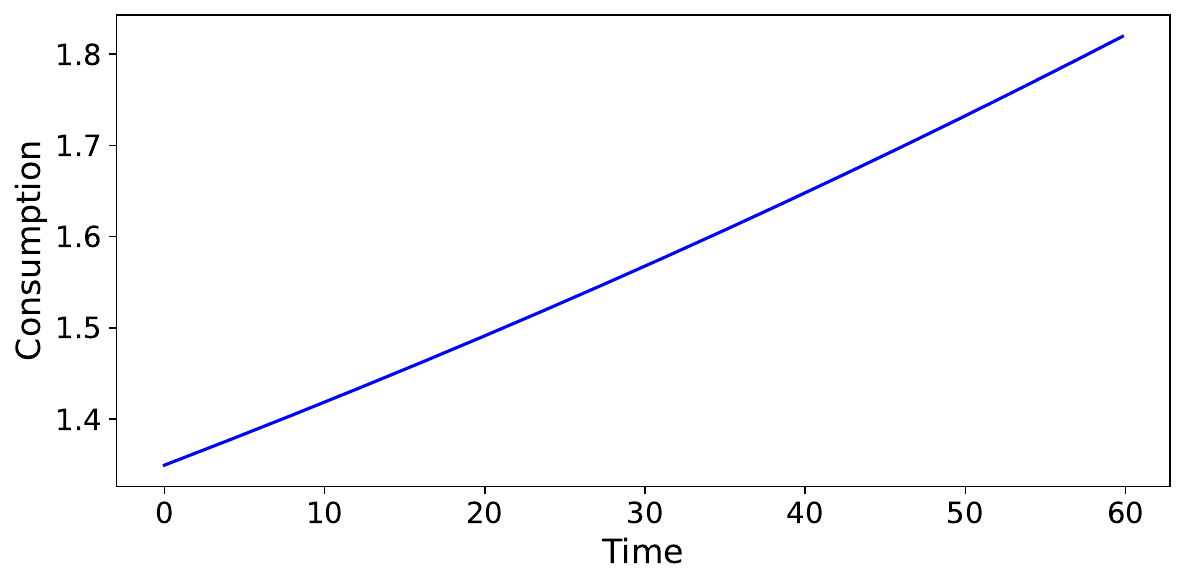}
    \caption{Consumption trajectory.}
    \label{fig:analytic-solution-consumption}
  \end{subfigure}
  \hfill
  \begin{subfigure}[t]{0.33\textwidth}
    \centering
    \includegraphics[width=\linewidth]{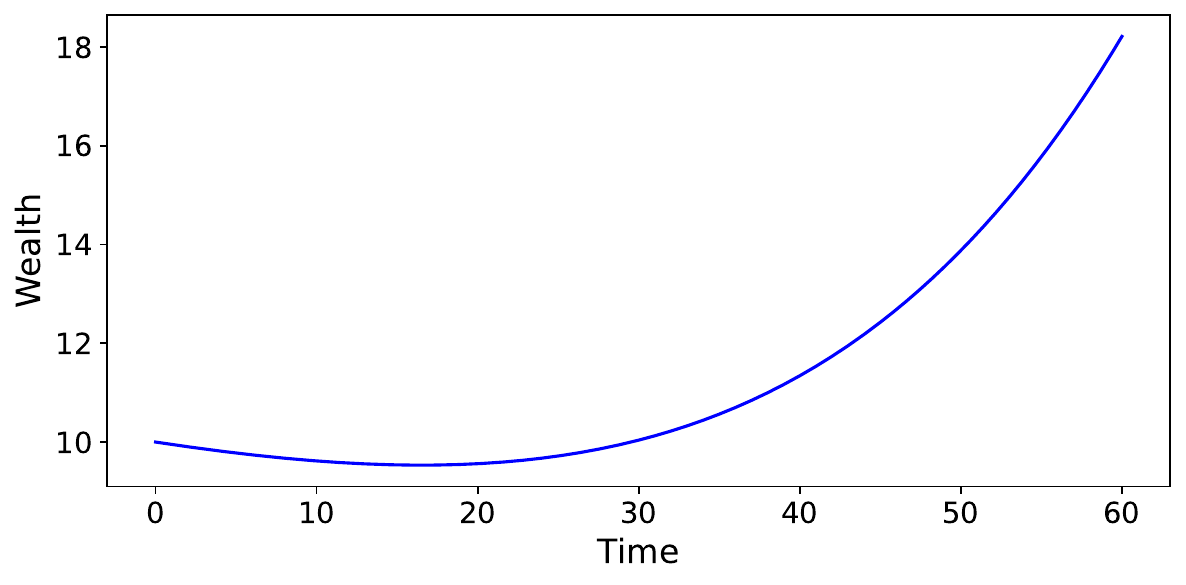}
    \caption{Wealth trajectory}
    \label{fig:analytic-solution-wealth}
  \end{subfigure}
  \caption{Plots of the analytic solution given in Example~\ref{ex:deterministic_crra} with the following parameters: $\delta = 0.02,\, r= 0.03,\, \gamma_1 = 2,\,\gamma_2 = 2, \, \mu = 0.01, \lambda = 100,\,  L = 60,\, \eta_0 = 1,\, w_0 = 10$.} 
  \label{fig:analytic-solution}
\end{figure}
\end{ex}

\begin{rem}
    In the economics literature, the life-cycle model of consumption begins with \cite{modigliani1954}. Subsequently \cite{blinder1975} extended the model to incorporate utility for terminal wealth. Bequest motive generates a nonlinear relationship between current consumption and terminal wealth. The only case where the marginal propensity to consume is linear in total wealth appears when the risk aversion parameter is equal to the elasticity parameter of bequest as in \ref{ex:deterministic_crra}. In the next section we will add stochastic noise to income, and we will observe that the optimal consumption policy is similar in spirit as the deterministic setting, but now involving conditional expectations of future marginal utility. 
\end{rem}

\subsection{The Stochastic Maximum Principle and BSDEs}

We now move to the stochastic life-cycle problem \eqref{eq:life_cycle_problem}. One possible approach to find the optimal control $c^*$ is through the dynamic programming principle, leading to an Hamilton--Jacobi--Bellman (HJB) equation. An alternative is to use the stochastic maximum principle, by analogy with the deterministic case discussed above.  In this article we choose the latter and illustrate how this approach provides insightful representation of the economic system of heterogeneous agents of interest. Furthermore, since the noise in \eqref{eq:life_cycle_problem} occurs in the income variable $\eta$ and not in the controlled wealth variable $w$, the associated HJB equation is hyperelliptic rather than parabolic, making it challenging to analyse. See for example \cite{ambrose_21_existence}. In contrast, the stochastic maximum principle approach allows us to employ fixed point arguments at the price of making some regularity assumptions on the utility functions. We will subsequently show that these regularity assumptions are satisfied by suitable approximations to the commonly used CRRA functions, see Example~\ref{ex:quad_crra_approx}.




The following theorem is the main technical result of our paper, which shows that solving the stochastic optimal control problem \eqref{eq:life_cycle_problem} with $C^1$ utilities is equivalent to solving a particular non-linear, path dependent SDE. The result relies on the stochastic maximum principle and the explicit solution formula for linear BSDE (Proposition~\ref{prop:bsde_linear_solution}).  


\begin{thm}\label{thm:life_cycle_optimal_rep}
    Let $(\Omega,\mcF,\mbP)$ be a probability space, $(w_0,\eta_0) \in \msL^0(\Omega;\mbR \times \mbR_+)$ be a given random variable and $\{\beta_t\}_{t\in [0,L]}$ be a standard $\mbP$-Brownian motion  with natural filtration $\{\mcF_t\}_{t\in[0,L]}$. Without changing notation we augment $\{\mcF_t\}_{t\in[0,L]}$ so that $(w_0,\eta_0)$ is $\mcF_0$-measurable, and recall that $r\in C([0,L];\RR)$ is for now exogenous and deterministic. Then, given measurable coefficients $\mu:\Omega\times [0,L]\times \mbR\to \mbR$, $\sigma :\Omega\times [0,L]\times \mbR\to \mbR$ and concave and continuously differentiable utility functions $u_1,\,u_2$, any solution to the system
    \begin{equation}\label{eq:life_cycle_explicit}
    \begin{aligned}
    &\begin{aligned}
          &\dd w_t = (r_tw_t +\eta_t -c^*_t(w))\dd t,\qquad w\tzero = w_0,\\
&\dd \eta_t = \mu_t(\eta_t)\dd t + \sigma_t(\eta_t)\dd \beta_t,\qquad \eta\tzero = \eta_0,
    \end{aligned}
        \qquad \text{for all }\,t\in [0,L],\\
        &c^*_t(w)= (u'_1)^{-1} \bigg( \lambda \exp\bigg( \int_t^L (r_s-\delta) \dd s \bigg)  \EE \left[ u'_2(w_L) |\cF_t\right]\bigg)
    \end{aligned}
    \end{equation}

     is a solution  to the optimal control problem \eqref{eq:life_cycle_problem}
    %
%
      %
%
    in the sense of the stochastic maximum principle (see e.g. \cite[Thm.~6.4.6 and 6.4.7]{pham_09_continuous}). 
\end{thm}
\begin{proof} To show that the existence of solutions to \eqref{eq:life_cycle_explicit} implies a solution to the optimal control problem in \eqref{eq:life_cycle_problem}, we  follow the stochastic maximum principle, see for example \cite[Thm.~6.4.6]{pham_09_continuous}. 

Let $\{\bar{\beta}_t\}_{t\in [0,L]}$ be a second standard Brownian motion, independent of $\beta$ and adapted to $\{\cF_t\}_{t\in [0,L]}$. 
We then re-write the dynamics of \eqref{eq:life_cycle_explicit} as follows,
\begin{equation}\label{eq:life_cycle_wealth_income_combined}
    \begin{aligned}
         \dd \begin{pmatrix}
        w_t 
        \\
        \eta_t
    \end{pmatrix}&= \begin{pmatrix}
            r_tw_t -(u'_1)^{-1}(y^1_t)+\eta_t,  
            \\
            \mu_t(\eta_t)\
            \end{pmatrix}\dd t + \begin{pmatrix}
                0 &0 
                \\
                0 & \sigma_t(\eta_t)
            \end{pmatrix}  \begin{pmatrix}
        \dd \bar{\beta}_t\\
        \dd \beta_t
    \end{pmatrix}. 
    \end{aligned}
\end{equation}
Note that while \eqref{eq:life_cycle_wealth_income_combined} is a degenerate two dimensional diffusion, the stochastic maximum principle still applies. We define $\mcH:[0,L]\times \RR^2\times \RR^2 \times \RR^2 \times \RR\rightarrow \RR$ the Hamiltonian associated with the optimization \eqref{eq:life_cycle_problem} as 
\begin{equation}\label{eq:hamilton}\mcH_t((w,\eta),y,z;c) = (r_t w -c +\eta) y^1+ \mu_t(\eta)y^2+\sigma(\eta)z^2 + e^{\delta(L-t)} u_1(c).
\end{equation}
It is clear that $c\mapsto \mcH(\,\cdot\,;c)$ is concave due to the assumption that $u_1$ is a concave function.

To apply the stochastic maximum principle, we 
consider the BSDE dual to \eqref{eq:life_cycle_wealth_income_combined} given by 
\begin{equation*}
\begin{aligned}
    - \dd \begin{pmatrix}
        y^1_t \\
        y^2_t
    \end{pmatrix} &= \begin{pmatrix}
        r_t y^1_t \\
        y^1_t + \mu'_t(\eta_t) y^2_t + \sigma'_t(\eta_t) z^2_t
    \end{pmatrix} \dd t -\begin{pmatrix}
        z^1_t & 0\\
        0 &z^2_t
    \end{pmatrix} \begin{pmatrix}
        \dd \bar{\beta}_t\\
        \dd \beta_t
    \end{pmatrix},\quad
    \begin{pmatrix}
        y^1_L\\
        y^2_L
    \end{pmatrix} &= \begin{pmatrix}
        \lambda u_2'(w_L)\\
        0
    \end{pmatrix},
\end{aligned}
\end{equation*}
and write $y=(y^1,y^2)$ and $z=(z^1,z^2)$.
The solution pair $(y,z)$ to this linear BSDE is adapted to the filtration $\{\cF_t\}_{t\in [0,L]}$. Moreover, $c \mapsto \mcH(\,\cdot\,;c)$ reaches its maximum when the first order condition is satisfied,  namely when 
\begin{equation*} 
y^1_t =e^{\delta(L-t)} u_1'(c_t). 
\end{equation*}
It follows that if we define 
$c^*_t = (u'_1)^{-1}(e^{-\delta (L-t)}y_t^1)$, 
then we have that 
\begin{equation*}
    \mcH_t((w_t,\eta_t),y_t,z_t;c^*_t) = \max_{c\in \cA}\mcH_t((w_t,\eta_t),y_t,z_t;c). 
\end{equation*}
where we recall that $\cA$ denotes the admissible class of controls $c $;  square integrable continuous  stochastic process adapted to the filtration $\{\cF_t\}_{t\in [0,L]}$. 
The linear BSDE describing $y^1$ can be solved independently of $y^2$, and the solution is given explicitly as
\begin{equation*}
    y_t^1=\lambda \mbE[e^{\int_t^L r_s\dd s} u'_2(w_L)|\cF_t],  
\end{equation*}
as shown in the Appendix in \eqref{eq:explicit BSDE}. Thus inserting this representation of the solution into the expression for the optimal consumption policy yields the third equation in \eqref{eq:life_cycle_explicit}.
 Since we have assumed that this system of equations  has a unique solution, then 
by  \cite[Thm.~6.4.6]{pham_09_continuous} we conclude that $t\mapsto c^*_t$ is the optimal consumption policy that solves \eqref{eq:life_cycle_problem}, concluding the proof.
\end{proof}

\begin{rem}
   Note that if we set the noise to zero in \eqref{eq:life_cycle_explicit} by imposing $\sigma \equiv 0$, together with assuming initial values of wealth and income to be deterministic, then the optimal consumption in Theorem \ref{thm:life_cycle_optimal_rep} coincides with the deterministic optimal consumption policy in Proposition \ref{prop:determ}. The stochastic system in \eqref{eq:life_cycle_explicit} therefore provides a direct generalization of the deterministic life-cycle model, and gives a similar interpretation of the optimal consumption policy in the stochastic setting.
\end{rem}
\begin{rem}
    It is worth noting that the optimal consumption only requires that the instantaneous income process is an It\^o process, which constitutes a broad class of stochastic processes. One would expect that this could be extended even further to L\'evy processes by using results from \cite{oksendal2007applied}, but we do not consider such an extension here  to rather focus on the subsequent equilibrium problem and overlapping generations case.  
\end{rem}
\begin{ex}
Consider utility functions of a CRRA type, i.e. of the form $u_1(x)=u_2(x)=\frac{x^{1-\gamma}
}{1-\gamma}$ for $\gamma >0$. In this case, using that the CRRA function is homogeneous, the optimal consumption policy can be represented as 
\begin{equation}\label{eq:consumption_CRRA}
  c^*_t = f_{\lambda,\gamma}(t,L)\EE[w_L^{-\gamma}|\cF_t]^{-\frac{1}{\gamma}},\quad where \quad f(t,L)=\lambda^{-\frac{1}{\gamma}}\exp\left(-\frac{1}{\gamma}\int_t^L (r_s-\delta) \dd s\right).  
\end{equation}

Of course, with this choice of utility functions we get a singular function inside the conditional expectation, and one will need to show that the resulting consumption policy is indeed finite and non-negative. 
Note in particular that  this does resemble the optimal consumption derived in the deterministic setting in \ref{prop:determ}, and in particular \eqref{det_c_wT}. However, 
note that even though we chose both utility functions with the same exponent $\gamma>0$ we do not find a linear dependence of current consumption on terminal wealth. This is due to the conditional expectation which does not commute with the non-linear functions $x\mapsto x^{-\gamma}$ and $y\mapsto y^{-\nicefrac{1}{\gamma}}$. The best we can obtain in this case is a linear lower bound on consumption using Jensen's inequality.
\end{ex}

\begin{rem}
\eqref{eq:life_cycle_explicit} and \eqref{eq:consumption_CRRA} both resembles the conventional Euler equation found in Economics, see e.g.  \cite{LjungqvistSargent2004}. 
An Euler equation \footnote{Note that Euler equations exist in other fields. We refer here explicitly to the terms use in Economics, see e.g. \cite{LjungqvistSargent2004}} is the economic term for optimal inter-temporal choice of consumption, and states that the optimal relationship between consumption over time has to equate the marginal rate of substitution with the relative price.
A marginal rate of substitution is the ratio of two marginal utilities adjusted by the rate of time preference.
To see this, note that the following three fundamental relations hold, for $\Delta >0$
\begin{align*}
    u_1'(c_t^*) &= \exp\left(\int_t^L (r_s-\delta)\dd s\right) \left(\lambda \EE\left[ u'_2(w_L)|\cF_t\right]\right)\\
    \EE[u_1'(c_{t+\Delta}^*) | \cF_t] &= \exp\left(\int_{t+\Delta}^L (r_s-\delta)\dd s\right) \left(\lambda \EE\left[ u'_2(w_L)|\cF_t\right]\right)\\
    \frac{u_1'(c_t^*)}{\EE[u_1'(c_{t+\Delta}^*) | \cF_t]} &= \exp\left(\int_{t}^{t+\Delta} (r_s-\delta)\dd s\right)\\
\end{align*}
This recovers the conventional economic logic that an optimal path of consumption is such that the marginal utility of consumption today equals the expected marginal utility of saving for consumption whether it's for tomorrow, over-morrow or at the end of time.
See \cite{acemoglu2009} for a thorough discussion on Euler equations in dynamic economic models.
\end{rem}


\subsection{Markovian Representation of consumption via PDE}
The optimal consumption derived in the previous section is given as a conditional expectation with respect to the filtration generated by the driving Brownian motion of the income process $\eta$, which is found by employing the stochastic maximum principle.  For computational purposes it is therefore important to determine whether this process is Markovian with respect to the wealth process or not. In fact, when trying solve the optimization problem in \eqref{eq:life_cycle_problem} using dynamic programming found through the Hamilton-Jacobi-Bellmann partial differential equation instead of our proposed FBSDE approach, one typically gets such Markovian property directly. We will elaborate a bit on this, only considering now the BSDE part. To this end, the value function $v(t,x)$ in the optimal consumption problem for the individual is found through the PDE 
\begin{align}\label{HJB Equation}
\partial_t v_t(w,\eta) + \frac{1}{2} \sigma^2(\eta) \partial_{\eta} v_t(w,\eta) + \max_{c}\cH(w,\eta,\partial_w v,\partial_\eta v;c )  = 0, \quad v_T(w,\eta) = \lambda u_2(w)
\end{align}
where $H$ is the Hamiltonian defined in \eqref{eq:hamilton}. 
Since PDE based solution methods have so far been more common in the application of continuous time models to economics, \cite{achdou_francisco_lasry_lions_moll_14_pde,gabaix_lasry_lions_moll_16_dynamics, achdou_han_lasry_lions_moll_21_income}, we include here a well-known connection between the two methodologies. The following proposition is of verification style, stating that under sufficient regularity conditions the  optimal consumption is indeed Markovian with respect to wealth. 

\begin{prop}\label{prop:consumption_markovian}
  Suppose there exists a function a unique classical solution $v\in C^{1,2}([0,L)\times \RR^d)\cap C^0([0,L]\times \RR^d)$ to the HJB equation \eqref{HJB Equation}, which enjoys linear growth property and polynomial growth of its gradient $\nabla_{w,\eta} v$.  
  %
%
  Then the optimal consumption policy $\{c_t^*\}_{t\in [0,L]}$ described in \eqref{eq:life_cycle_explicit} is equal to 
\[   c_t^*(w_t,\eta_t)= (u'_1)^{-1}(e^{-\delta (L-t)}\partial_w v(t,w_t,\eta_t)).    \] 
with $\{(w_t,\eta_t)\}_{t\in [0,L]}$ also as in  \eqref{eq:life_cycle_explicit}.
In particular, we see that $c^*_t$ is Markovian with respect to the pair of processes $(w,\eta)$.
\end{prop}
\begin{proof}
  This statement is combination of the two statements  \cite[Prop. 6.3.2]{pham_09_continuous} and \cite[Thm.~6.4.7]{pham_09_continuous}, slightly rewritten to suit the notation of this article. 
\end{proof}

\begin{rem}
 A relatively recent article, \cite{ambrose_21_existence} obtained local in time existence and uniqueness of solutions to the HJB equation describing solution to the optimal control problem in the life cycle problem \eqref{eq:life_cycle_problem} and a relaxed general equilibrium. While the utility functions treated therein are the true CRRA functions, a number of other modifications to the problem are imposed. For example \cite{ambrose_21_existence} considers solutions with truncated densities of income and wealth and relaxes the general equilibrium condition, see \cite[Sec.~3]{ambrose_21_existence}. It seems challenging to obtain sufficient regularity of the PDE solutions studied by \cite{ambrose_21_existence} such that one can guarantee an equivalence to the FBSDE approach.

\end{rem}

\subsection{Natural Borrowing Limits and Wealth Asymptotic}
A finite time horizon with both an incentive to consume and an incentive to bequeath encodes a strong disciplinary effect on the individuals. Since, in the idealized model, individuals receive an infinitely negative payoff for ending up with negative wealth at time $L$, the optimal savings/consumption policy must be such that the the terminal wealth at time $L$ is non-negative. 
This creates a natural borrowing limit for the individuals; a debt limit such that even borrowing up to this limit, individuals expect to  repay their debt before time $L$. A distinct advantage of the FBSDE approach to describing the life-cycle dynamics is that we can express the natural borrowing limit as an analytic function, which we will see is a natural generalisation of the deterministic counterpart.  

In the economic literature \cite{aiyagari_94_uninsured} introduced the concept of natural borrowing limit in an infinite horizon economy with finite income states. In a stationary equilibrium with interest rate $r>0$, the natural borrowing limit is defined as $-\frac{\eta^1}{r}$ where $\eta^1$ is the lowest income. 
In a continuous-time  model with a discrete two-state income process \cite{achdou_francisco_lasry_lions_moll_14_pde}, in which $\eta_t\in \{ \eta^1,\eta^2\}$ for all $t\in [0,L]$, the natural borrowing limit at any point in time  can be proven to be the net present value of all future income - in the {\em low} income state. That is, suppose $\eta^1< \eta^2$, then the borrowing limit (equivalently lower bound of wealth $\underline{w}$) at time $t\in [0,L]$ will be 
\footnote{When $r$ is constant we get $\underline{w} = -{\frac{\eta^1}{r}}\left(e^{-rt}-e^{-rL}\right)$. Now let $t\to 0$ and $L\to \infty$ to obtain Aiyagari's natural borrowing limit.}.
\[\underline{w}=-\eta^1 \int_t^L  e^{-\int_s^L r_u \dd u }\dd s \]
The reason for this borrowing limit, even for people in the high income state,  is that at any point in time there is a greater than $0$ probability that you will jump to, and remain in, the lowest income state for the rest of your life (until time $L$). Therefore, if you borrow more than relative to the lowest possible income, there will be a positive probability that you end up with $-\infty$ in your value function which is sub-optimal.

The same thought experiment can be extended to a continuous time model with discrete state space $\eta_t\in \{ \eta^1,\eta^2,\ldots, \eta^N\}$ for all $t\in [0,L]$, as long as there is a positive probability of jumping from your current income level to the lowest, at any point in time. 
However, if we  assume the stochastic income process to be a random walk, in which the state space is given by $\{\eta^1,\ldots, \eta^N\}$ with $\eta^i\leq \eta^j$ for all $i\leq j$, then at each point in time the income process may jump one level up or one level down. Thus, the natural borrowing limit at time $t$ depends on the state of the income process at time $t$. In particular, if $\eta_t=\eta^i$, then the maximal amount you are willing to borrow will be given in terms of the lowest income state possible in the next period, i.e. given that $\eta_t=\eta^i$
\[
\underline{w}_t = -\eta^{i-1} \int_t^L  e^{-\int_s^L r_u \dd u }\dd s
\]
 The reason for this is that the consumption process may be updated accordingly in the case of a bad income shock. Note in particular that since the income process may jump up and down, the borrowing limit in this model is depending on time as well as the current income level at that time. The natural borrowing limit therefore becomes heterogeneous in income. This gives a much more generous natural borrowing limit for individuals in high income states.

Just as a random walk can approximate a Brownian motion, the next proposition will show that when income is an It\^o diffusion we also recover a heterogeneous  natural borrowing limit in the case of CRRA utility functions, conditioned on having a unique solution to the associated optimal control problem. 

\begin{prop}[Lower bound for the natural borrowing limit]\label{prop:natural_borrowing_limit}
    Let $\gamma_1,\,\gamma_2>0$, 
    \begin{equation}
        u_1(x) \coloneqq \begin{cases}
            \frac{1}{1-\gamma_1}x^{1-\gamma_1}, & x\geq 0,\\
            -\infty, & x<0,
        \end{cases}
        \qquad 
        u_2(x) \coloneqq \begin{cases}
            \frac{1}{1-\gamma_2}x^{1-\gamma_2}, & x\geq 0,\\
            -\infty, & x<0,
        \end{cases}
    \end{equation}
    and $(w,c(w))$ be a solution to the system \eqref{eq:life_cycle_explicit} which is also a solution to the optimal control problem \eqref{eq:life_cycle_problem}. Then for all $t\in [0,L]$ it holds that
    \begin{equation}\label{eq:natural borrowing limit}
           w_t\,  \geq\,  \underline{w}_t\coloneqq -\int_t^L \exp\left(-\int_t^s r_u \dd u\right)\EE[ \eta_s|\cF_t]\dd s, \quad \PP-a.s..
    \end{equation}
    %
%
   %
    %
\end{prop}
\begin{proof}
%
%

Our first observation is to note that since $c(w)$ is assumed to be an optimal control, the lower bound for $w_t$ must correspond to a situation where $c_t(w)=0$, since any positive amount of consumption is strictly preferred.

Secondly, we note that formally, for $i=1,\,2$,
\begin{equation}
    u'_i(x) = \begin{cases}
        x^{-\gamma_i}, & x> 0,\\
        +\infty, & x\leq  0,
    \end{cases}
    \qquad
    (u'_i)^{-1}(x) = \begin{cases}
        x^{-\frac{1}{\gamma_i}}, & x > 0,\\
        +\infty, & x\leq 0.
    \end{cases}
\end{equation}
Hence, $\mbP-a.s.$ we have
\begin{align*}
\infty > c_t(w) =&\, (u'_1)^{-1} \bigg( \lambda \exp\bigg( \int_t^L (r_s-\delta) \dd s \bigg)  \EE \left[ u'_2(w_L) |\cF_t\right]\bigg) \geq 0.
\end{align*}
Furthermore, since $c(w)$ is the optimal consumption policy it cannot be the case that $w_L <0$, since in this case we would have $\mcJ_L(w,\, \eta|c) =-\infty$ which is clearly sub-optimal. Hence, we may use the precise form of $u_1,\,u_2$, their derivatives and inverses to see that
\begin{align}
    c_t(w) = &\, \lambda^{-\frac{1}{\gamma_1}} \exp\bigg( -\frac{1}{\gamma_1}\int_t^L (r_s-\delta) \dd s \bigg)  \EE \left[ w_L^{-\gamma_2}|\cF_t\right]^{-\frac{1}{\gamma_1}} \notag \\
    \leq &\, \lambda^{-\frac{1}{\gamma_1}} \exp\bigg( -\frac{1}{\gamma_1}\int_t^L (r_s-\delta) \dd s\bigg) \,  \EE \left[ w_L|\cF_t\right]^{\frac{\gamma_2}{\gamma_1}}, \label{eq:consumption_jensen_bound}
\end{align}
where in the last line we used the conditional Jensen inequality applied to the convex function $[0,+\infty) \ni x\mapsto x^{-\gamma_2}$. Note that since both $\gamma_1,\,\gamma_2 >0$ inequality \eqref{eq:consumption_jensen_bound} implies that
\begin{equation*}
    \mbE[w_L\,|\mcF_t] =0 \quad \Rightarrow \quad c_t(w)=0.
\end{equation*}

From the dynamics of $w$, by variation of constants and a bit of algebraic manipulation, it is clear that 
\begin{align*}
    w_L = &\, \exp\left(\int_t^L r_u \dd u \right)\bigg(\exp\left(\int_0^t r_u \dd u \right) w_0 
    +\int_0^t \exp\left(\int_s^t r_s \dd s\right) (\eta_s- c_s)\dd s\bigg) 
    \\
    &\, +\int_t^L \exp\left(\int_s^L r_s \dd s\right) (\eta_s- c_s)\dd s.
\end{align*}
We can rewrite this in the following simplified form 
\begin{equation*}
    w_L = \exp\left(\int_t^L r_u \dd u \right) w_t +\int_t^L \exp\left(\int_s^L r_u \dd u\right) (\eta_s- c_s)\dd s.  
\end{equation*}
As argued above, the lower bound for $w_t$ corresponds to the case $c_t(w)$ which is guaranteed by having $\mbE[w_L|\mcF_t] =0$. In this case, since $w_t$ is $\mcF_t$-measurable,
\begin{equation*}
    \EE[w_L|\cF_t]= 0 \, \Longleftrightarrow \, \exp\left(\int_t^L r_u \dd u \right) w_t=-\int_t^L \exp\left(\int_s^L r_u \dd u\right)\EE[ (\eta_s- c_s)|\cF_t]\dd s. 
\end{equation*}
So that rearranging, and again using that we have argued  that $c(w)\geq 0$, 
\begin{align*}
    w_t = &\, -\int_t^L \exp\left(-\int_t^s r_u \dd u\right)\EE[ (\eta_s- c_s)|\cF_t]\dd s, \\
    =&\, -\int_t^L \exp\left(-\int_t^s r_u \dd u\right)\EE[\eta_s|\cF_t]\dd s  + \int_t^L \exp\left(-\int_t^s r_u \dd u\right)\EE[c_s|\cF_t]\dd s\\
    \geq &\, -\int_t^L \exp\left(-\int_t^s r_u \dd u\right)\EE[\eta_s|\cF_t]\dd s,
\end{align*}
which proves the claim.
\end{proof}

In economics the natural borrowing limit (\cite{aiyagari_94_uninsured}, \cite{achdou_han_lasry_lions_moll_21_income} and others) is constant, ensuring that individuals cannot engage in Ponzi schemes. Practically, the borrowing limit has been used computationally to define the lower bound of assets when solving the model numerically. In contrast, $\underline{w}_t$ is stochastic and defined for each individual realization of wealth.

\begin{rem}
    The natural borrowing limit allows us to conclude that terminal wealth is almost surely positive, i.e.  $w_L\geq 0$ $\PP$-a.s.. Indeed, this follows directly from \eqref{eq:natural borrowing limit} since for $t=L$, 
\[
  \underline{w}_L=-\int_L^L \exp\left(-\int_t^s r_u \dd u\right)\EE[ \eta_s|\cF_L]\dd s=0.
\]
\end{rem}

\begin{rem}
    It should be noted that we do not find the natural borrowing limit itself, but rather a lower bound for the natural borrowing limit. The reason is that we use Jensen's inequality to find a bound for the consumption. It could therefore be that the consumption process would be 0, $\PP$-a.s.,  at a wealth level which is higher than $\underline{w}$. 
\end{rem}

\begin{ex}
    Suppose the income rate is given as an exponential Brownian motion, i.e. as $\eta_t = \eta_0\exp\left(\left(\mu-\frac{\sigma^2}{2}\right)t + \sigma \beta_t\right)$ for some parameters $\mu,\,\sigma>0$. This process can be decomposed into a martingale and non-martingale part
    \begin{equation*}
        \eta_t = \eta_0\exp\left(-\frac{\sigma^2}{2}t + \sigma \beta_t\right)\exp\left(\mu t\right)
    \end{equation*}
    so that for any $\mu \geq 0$ we have 
    \begin{equation*}
        \underline{w}_t = -\eta_t\int_t^L\exp\left(-\int_t^sr_u -\mu \,\dd u\right) \dd s. 
    \end{equation*}

    So the natural borrowing limit in this case is a discounting of the current income level of the individual according to future interest rates. In particular, when $r$ is constant the expression becomes 
    $$
        \underline{w}_t = -\frac{\eta_t}{r-\mu} (1-e^{-(r-\mu)(L-t)})
    $$

\end{ex}
\begin{figure}[!htbp]
  \centering
  \begin{subfigure}[t]{0.51\textwidth}
    \centering
\includegraphics[width=\linewidth]{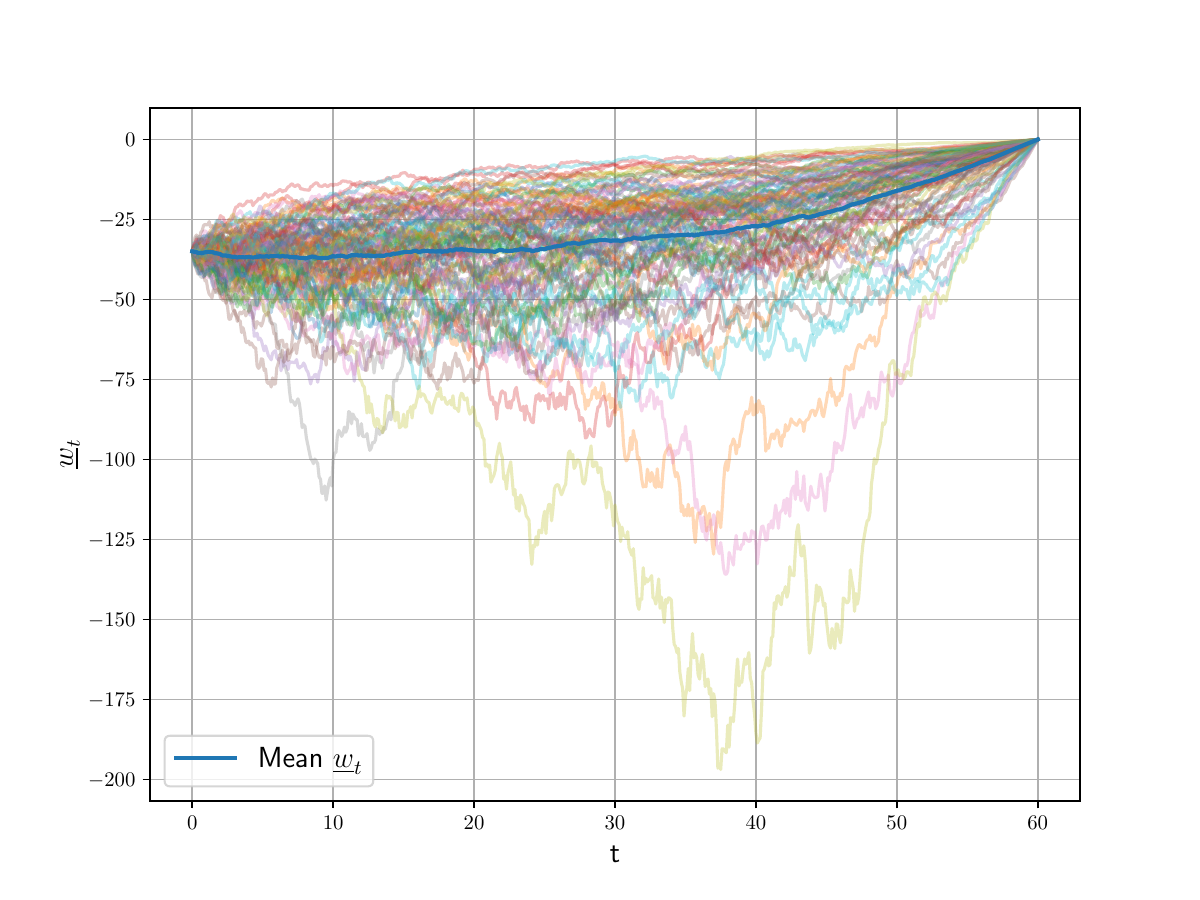}
    \caption{}
    \label{fig:borrowing-limit-dynamic}
  \end{subfigure}
  \hspace{-2em}
  \begin{subfigure}[t]{0.51\textwidth}
    \centering
    \includegraphics[width=\linewidth]{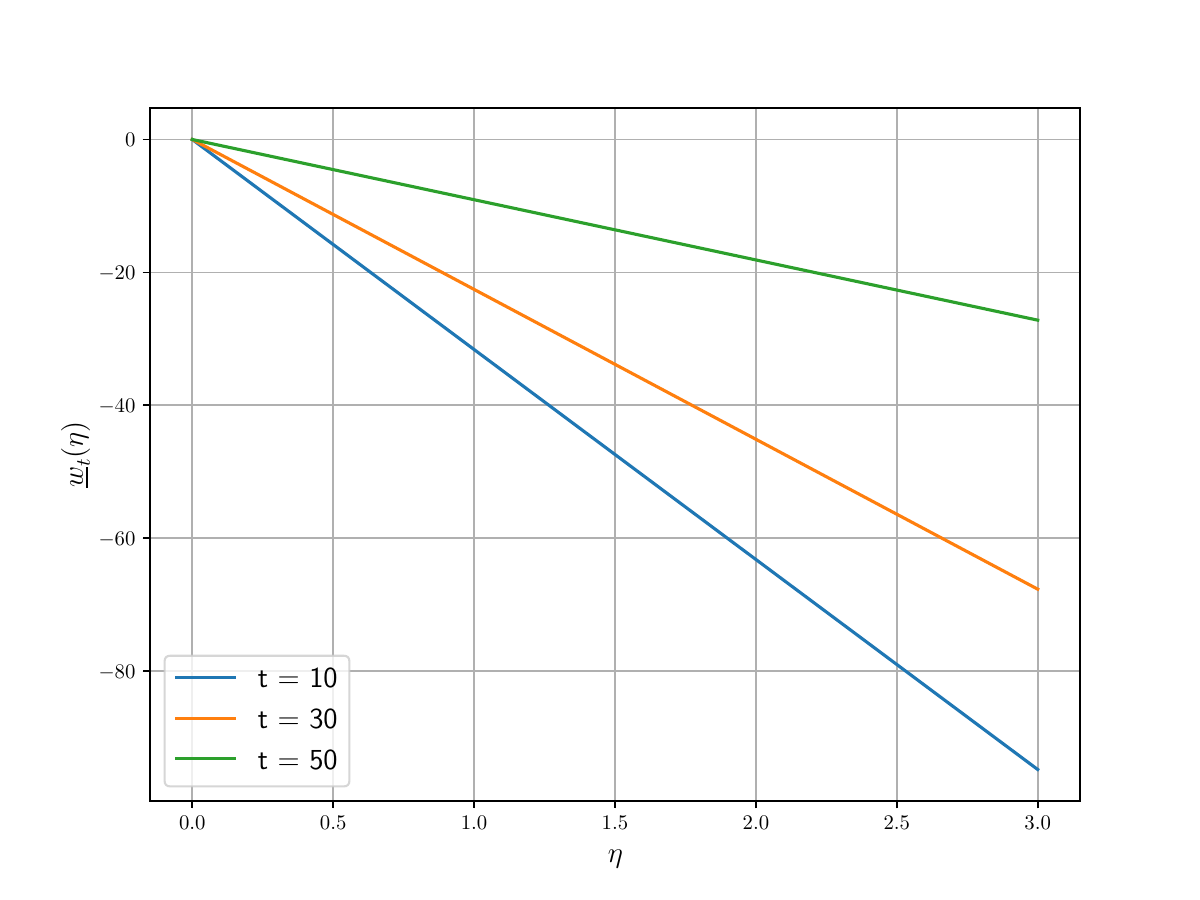}
    \caption{}
    \label{fig:borrowing-limit-static}
  \end{subfigure}
  \hfill
 \caption{Stochastic natural borrowing limit. We use $\mu=0.01$, $\sigma=0.1$, $\eta_0=1.0$, $r=0.03$, $L=60$. In panel (A) the blue line depicts expected borrowing limit $\mbE[\underline{w}_t]$.
     }
    \label{fig_nbl}
\end{figure}
In Figure (\ref{fig_nbl}) we simulate the natural borrowing limit for the case when income follows a geometric Brownian motion and interest rate is constant. In panel (A), we see that most individuals will be able to borrow more early in their life-cycle on average, but their level depends on their idiosyncratic realization of income. In panel (B), $\underline{w}_t = -\eta_t  \frac{1 - e^{-(r-\mu)(L-t)}}{r}$ is depicted (as a plot of $\underbar{w}$ against income at varying times) for a constant interest rate $r$. Intuitively, the borrowing limit is tighter for low-income individuals $\eta_t$. Furthermore, individuals will face a tighter borrowing limit later in life as their human capital falls. Hence, the line corresponding to $t=50$ is higher (i.e. tighter lower bound) than that for $t=10$ at all income levels.

\subsection{Asymptotic Values of \boldmath{$r\mapsto \EE[w_t(r)]$}}\label{sec:life_cycle_asymptotics}

Using the specific structure of the optimal consumption and wealth, we can provide quantitative bounds for the asymptotic limits of expected wealth, $\EE[w_t]$  as a function of $r\in \RR$. As expected, when $r\rightarrow \infty$, we can show that aggregate wealth $\EE[w_t]$ also tends to infinity. 
From the lower bound of wealth found in Proposition \ref{prop:natural_borrowing_limit} it is reasonable to expect that wealth diverges to $-\infty $ as $r\rightarrow -\infty$. However, proving this seems to be challenging. For our purpose, however, it is sufficient to prove that this  limit is non-positive. We prove these facts in the next proposition.

\begin{prop}\label{prop:range_of_expected_w}
Assume $u_1$ and $u_2$ are such that 
    \[ (u'_1)^{-1}(u'_2(x))\leq x \quad \forall x\in \RR,\]
    and that the inverse  $(u'_1)^{-1}$ is homogenous of degree $-\alpha$ for some $\alpha\geq1$.
    Suppose further that there exists a pair $(w,c)$ which solves \eqref{eq:life_cycle_sde_sol}. 
    Then, for all $t\in [0,L]$ it holds that
    \begin{equation*}
            \lim_{r\mapsto \infty} \EE[w_t(r)] =\infty, \quad \mathrm{and} \quad \lim _{r\mapsto -\infty} \EE[w_t(r)]\leq 0 . 
    \end{equation*}
\end{prop}
\begin{proof}
By variation of constants formula, we can write 
\[
\EE[w_t] = e^{rt}\left(\EE[w_0]+\int_0^t e^{-rs} (\EE[\eta_s]-\EE[c_s(r)])\dd s\right). 
\] 
%
To analyse the asymptotic behaviour when $r\rightarrow \infty $ or $-\infty$ we need to investigate the behaviour of the final integral term in the above representation, namely
\[
A(r):=\int_0^t e^{-rs}\left(\EE[\eta_s]-\EE[c_s(r)]\right)\dd s. 
\]
Using that $(u'_1)^{-1}$ is convex and homogenous of degree $-\alpha$, and that $u'_2$ is such that 
\[
(u'_1)^{-1}(u'_2(x))\leq x \quad \forall x\in \RR, 
\]
it follows from Jensen's inequality that 
\[
\EE[c_t(r)] \leq \lambda^{-\alpha} e^{-\alpha r(L-t)} \EE[w_L].
\]
Note that 
\[\EE[w_L]= e^{rL}\bigg(\EE[w_0]+\EE\Big[\int_0^L e^{-rs}\big(\eta_s-c_s(r)\big)\dd s\Big] \bigg)\leq e^{rL}\bigg(\EE[w_0]+\int_0^L e^{-rs}\EE[\eta_s]\dd s\bigg) .
\]
Combining these estimates, we see that
\[
\EE[c_t(r)] \leq \lambda^{-\alpha} e^{(1-\alpha) rL} e^{\alpha r t}\bigg(\EE[w_0]+\int_0^L e^{-rs}\EE[\eta_s]\dd s\bigg) 
\]
Using this, we see that 
\[
A(r)\geq \int_0^t e^{-rs}\EE[\eta_s]-\lambda^{-\alpha} e^{(1-\alpha) rL} e^{(\alpha-1)rs} \bigg(\EE[w_0]+\int_0^L e^{-rl}\EE[\eta_l]\dd l\bigg)\dd s. 
\]
%
 Recall that $t\leq L$ and $\kappa\geq 1$. Then we see that 
\[
\lambda^{-\alpha} e^{(1-\alpha) rL} e^{(\alpha-1)rs}\rightarrow 0 \quad {\mathrm{as}} \quad r\rightarrow \infty,  
\]
and thus it follows that 
\[
A(r)\geq 0 \quad \mathrm{as}\quad r\rightarrow \infty.
\]
Using this, together with the assumption that $\EE[w_0]>0$ and that 
\[
\EE[w_t] \geq e^{rt}(w_0 + A(r)).
\]
we see that for all $t\in [0,L]$
\[
\EE[w_t]\rightarrow \infty \quad \mathrm{as} \quad r\rightarrow \infty. 
\]

We conclude by proving that $\lim_{r\to -\infty}\EE[w_t]\leq0$ for all $t\in [0,L]$. But this follows from an easier estimate; by assumption, $c_t\geq 0$ for all $t$. Then  clearly 
\[
\EE[w_t]\leq e^{rt}\bigg(\EE[w_0]+\int_0^t e^{-rs}\EE[\eta_s]\bigg)
\]
From this it is clear that $\EE[w_t]\leq  0$ as $r\rightarrow -\infty$ for all $t$.
\end{proof}

\begin{rem}\label{rem:natural_borrowing_gamma_leq_1}
    In the context of CRRA utilities with $u_1(x)=u_2(x)=\frac{x^{1-\gamma}}{1-\gamma}$, \blue{assuming existence and uniqueness of the dynamics holds in this case}, then $(u'_1)^{-1}(x)=x^{-\frac{1}{\gamma}}$, and thus in this case $\alpha=-\nicefrac{1}{\gamma}$. So that Proposition~\ref{prop:range_of_expected_w} applies provided $\gamma\leq 1$. 
\end{rem}

\begin{rem}
  In principle, if we know that a function $f:\RR\rightarrow \RR$ is continuous and that $\lim_{r\rightarrow \infty}f(r)=\infty$ and $\lim_{r\rightarrow-\infty}f(r)\leq 0$, 
  then we can conclude that for any $K\geq 0$ there exists an $r\in \RR$ such that $f(r)=K$. That is, in our context we conclude that for each $t\in [0,L]$ there exists an $r \in \RR$ (depending on $t$) such that $\EE[w_t]=K$. Therefore, this does not allow us to conclude that there exists a general equilibrium interest rate path in the life-cycle case. To see this, recall that at any time $t\in [0,L]$, wealth $C([0,L];\mbR) \ni r\mapsto w_t(r)$ is a functional of the entire path of the interest rate, not just the interest rate at time $t\in [0,L]$. Concretely, it is not clear, from Proposition~\ref{prop:range_of_expected_w} that there exists a continuous selection $t\mapsto r(t)$ to give $\mbE[w_{\,\cdot\,}] \equiv K $ on $[0,L]$. However, in the OLG setting analysed in subsequent sections, we can use this technique to prove existence in stationary economies.
  %
  %
\end{rem}

\subsection{Existence and uniqueness of life-cycle SDE dynamics}\label{sec:life_cycle}

Theorem~\ref{thm:life_cycle_optimal_rep} shows that given concave utility functions, an optimal wealth process is found by solving the system 
\begin{equation}\label{eq:triple_equations}
    \begin{aligned}
        w_t =&w_0 + \int_0^t \left(r_s w_s -c_s(w,r) +\eta_s \right)\dd s 
        \\
      \eta_t = &\eta_0 + \int_0^t \mu_s(\eta_s) \dd s +  \int_0^t\sigma_s(\eta_s)\dd   \beta_s 
        \\
        c_t(w,r) = & (u_1')^{-1}\left(\lambda \EE\left[\exp\left(\int_t^L (r_s-\delta)\dd s\right) u'_2(w_L)|\cF_t\right]\right),
    \end{aligned}
\end{equation}
on $[0,L]$, where we see the consumption $c$ as a functional of the continuous paths $w$ and $r$. 
Note that by variation of constants, we can rewrite the equation for wealth as
\begin{equation*}
w_t =  e^{-\int_0^tr_s \dd s} w_0 + \int_0^t \exp\left(\int_{s}^t r_u \dd u\right)\left(c_s(w,r) +\eta_s\right) \dd s,\quad t\in [0,L].
\end{equation*}

The purpose of this section is to show that under our standing Assumption \ref{ass:utility_reg_intro} there exists a unique solution to \eqref{eq:triple_equations}.
%
     %
%
     %
     %
%
%
%
The following example illustrates an approximation to the commonly used CRRA utility functions which satisfy Assumption~\ref{ass:utility_reg_intro}.
\begin{ex}\label{ex:quad_crra_approx} 
    The well-studied CRRA utility functions do not satisfy Assumption~\ref{ass:utility_reg_intro}, since for $\gamma>0$ one has
    \begin{equation*}
        u(x) = \begin{cases}
            \frac{1}{1-\gamma} x^{1-\gamma}, &x > 0,\\
            -\infty, & x \leq 0,
        \end{cases}\quad\Rightarrow \quad  u'(x) = \begin{cases}
            x^{-\gamma}, & x>0,\\
            -\infty, & x\leq 0.
        \end{cases}
    \end{equation*}
    The structure of these utility functions is such that $u(|x|) \gg u(-|x|)$ for all $x\in \mbR$; in other words, positive inputs always return higher outputs than a negative input. In practical terms, this models the idea that agents would always rather consume any positive quantity than any negative quantity. With this in mind, for practical purposes, we can suitably approximate $u$ by a sufficiently regular function, retaining the property of preferential consumption at the level of machine tolerance, along with concavity.
    
    For $0<\eps \ll 1$ and any $p\geq 1$, let us define
       \begin{equation*}
   	u_{\eps,p}(x) \coloneqq \begin{cases}
   		\frac{1}{1-\gamma}x^{1-\gamma}, & x \geq \eps,\\
   	-\frac{1}{2\eps^p}x^2 + \left(\frac{1}{\eps^\gamma}+\frac{1}{\eps^{p-1}}\right)x- \frac{1}{2\eps^{p-2}}+ \frac{\gamma}{1-\gamma}\eps^{1-\gamma},  & x\leq \eps.
   	\end{cases}
   \end{equation*}
    The quadratic polynomial is chosen so that the function $x\mapsto u_{\eps,p}(x)$ remains concave while satisfying Assumption~\ref{ass:utility_reg_intro} and is continuous at $x=\eps$. To wit, taking the first derivative we find
       \begin{equation*}
    	u'_{\eps,p}(x) = \begin{cases}
    		x^{-\gamma}, & x\geq \eps,\\
    	-\frac{1}{\eps^p}x + \left(\frac{1}{\eps^\gamma}+\frac{1}{\eps^{p-1}}\right),& x\leq \eps,
    	\end{cases}
    \end{equation*}
    and so
        \begin{equation*}
 	u''_{\eps,p}(x) = \begin{cases}
 		-x^{-(\gamma+1)}, & x> \eps,\\
 			-\infty, & x = \eps,\\
 		-\frac{1}{\eps^p},& x<\eps.
 	\end{cases}
 \end{equation*}
 Hence, $x\mapsto u_{\eps,p}(x)$ is globally concave. In addition, we have
      \begin{equation*}
    	(u'_\eps)^{-1}(y) = \begin{cases}
    	y^{-\gamma}, & y\geq \eps,\\
    	-\eps^p y - \left(\frac{1}{\eps^\gamma}+\frac{1}{\eps^{p-1}}\right),& y\leq \eps.
    \end{cases}
 \end{equation*}
 So that one readily checks that Assumption~\ref{ass:utility_reg_intro} is satisfied. 

\end{ex}

\subsection{Stability estimates}

A particular strength of the FBSDE formulation of the optimal control problem in \eqref{eq:life_cycle_explicit}, is that the optimal consumption policy function is explicitly stated in terms of input variables, and thus stability of this function is easily derived. We do this in the following propositions, which will play a central role in subsequent proofs of general equilibrium.

\begin{prop}[Stability of Consumption]\label{prop:life_cycle_consumption_stabillity}
    Suppose the optimal consumption policy, $c:[0,L]\times \scL^1(\Omega;C_L)\times C([0,L];\RR)\rightarrow \RR$ is given by that of \eqref{eq:life_cycle_explicit}, where $u_1$ and $u_2$ satisfy Assumption~\ref{ass:utility_reg_intro} for some $\kappa>0$. 
   Then, for any given $w,\, v \in \scL^0(\Omega;C_L)$ and $r,\,h\in C_L$, it holds that
\begin{equation}\label{eq:life_cycle_consump_bdd}
    \|c(w,r)\|_{\msL^\infty_L}\leq \kappa, \quad \mbP-\text{a.s.}
    \end{equation}
   and
    \begin{align}
    \|c(w,r)-c(w,h)\|_{\msL^\infty_L} &\leq \lambda  \kappa^2 \exp\left((\|r\|_{\msL^\infty_L}+\|h\|_{\msL^\infty_L})L\right) \| r-h \|_{\msL^\infty_L},\label{eq:consump_interest_stable}\quad &&\PP-a.s.
    \\
   \|c(w,r)-c(v,r)\|_{\msL^\infty_L} &\leq \lambda \kappa^2 \exp\left(\|r\|_{\msL^\infty_L} L\right) \|w-v\|_{\msL^\infty_L}, \quad &&\PP-a.s. \label{eq:consump_wealth_stable} %
    \end{align}
\end{prop}

\begin{proof}
    We recall the representation of the consumption given in \eqref{eq:life_cycle_explicit} and begin by proving \eqref{eq:life_cycle_consump_bdd}. From Assumption~\ref{ass:utility_reg_intro} it holds that that $(u_1')^{-1}$ is bounded for non-negative arguments and that $u_2'$ is non-negative; hence it follows that 
    \begin{equation}\label{eq:bdd_consumption}
         | c_t(w,r)| = \left|(u_1')^{-1}\left(\lambda\, \EE\left[\exp\left(\int_t^L (r_s-\delta)\dd s\right) u'_2(w_L)|\cF_t\right]\right)\right| \leq \kappa
    \end{equation}

    We continue to prove stability of the map $r\mapsto c(w,r)$ for a given continuous process $w\in \msL^0(\Omega;C_L)$. 
    Under the assumption that $(u_1')^{-1}$ is Lipschitz, with Lipschitz constant $\kappa$, then it is readily seen that, for all $r,h\in C_L$,
    \[
     \|c(w,r)-c(w,h)\|_{\msL^\infty_L} \leq \kappa \lambda \EE[u_2'(w_L)|\cF_t]\exp((\|r\|_{\msL^\infty_L}+\|h\|_{\msL^\infty_L})(L-t))\| r-h \|_{\msL^\infty_L},\quad \PP-a.s..
    \]
Using Assumption~\ref{ass:utility_reg_intro} again, we have $|\EE[u_2'(w_L)|\cF_t]|\leq \kappa$ so that we have
 \[
     \|c(w,r)-c(w,h)\|_{\msL^\infty_L} \leq \lambda  \kappa^2 \exp((\|r\|_{\msL^\infty_L}+\|h\|_{\msL^\infty_L})L) \| r-h \|_{\msL^\infty_L}, 
    \]
    where in the last estimate we also used Jensen's inequality and that $\EE[\|w\|_{\msL^\infty_L}]\lesssim \EE[\|w\|_{\msL^\infty_L}^2]$. 

The proof of stability in the map $w\mapsto c(w,r)$ proceeds in a very similar manner as above, i.e. we again invoke the Lipschitz assumptions on $(u_1')^{-1}$ and $u_2'$ from Assumption~\ref{ass:utility_reg_intro} to obtain the $\mbP$-a.s. bound
\begin{equation*}
     \|c(w,r)-c(v,r)\|_{\msL^\infty_L} \leq \lambda \kappa^2 \exp(\|r\|_{\msL^\infty_L} L) \|w-v\|_{\msL^\infty_L}
\end{equation*}
    which concludes the proof. 
    \end{proof}

Note that it's possible to interpret \eqref{eq:consump_interest_stable} as a bounded substitution effect.
Although in a more general, path-wise form as opposed to the scalar- or vector-valued inputs to demand functions discussed in, among others, \cite{mas-colell.whinston.ea95}.
Fixing a path for wealth while changing the interest rate path is as if we calculate compensated demand functions, or Hicksian demand functions.
The difference is that the compensating transfer takes the form of a path rather than a lump-sum transfer.

In a similar vein, \eqref{eq:consump_wealth_stable} is the path-wise analogy to the wealth effect.
It then follows from \ref{prop:life_cycle_consumption_stabillity} that arithmetic compositions of the two, such as the Slutsky identity, are bounded as well.
Our stability estimates highlights one benefit of the FBSDE formulation.
The explicitly stated optimal consumption policy enables a straight-forward way to derive certain results analogous to conventional economic theory.
Building on the bounds above, \eqref{lem:LC_rate_stable},  shows the stability of the uncompensated, or Marshallian, demand.

With the stability of the consumption policy at hand, we are now ready to prove stability of the solution  map $\Theta_L:\scL^1(\Omega;C_L)\times C_L \rightarrow \scL^1(\Omega;C_L) $ defined 
 for all $t\in [0,L]$ by setting
\begin{equation}\label{eq:Theta map}
\Theta_L(w,r)_t \coloneqq \exp\left(\int_{0}^t r_s \dd s\right)w_0 + \int_0^t \exp\left(\int_{s}^t r_s \dd s\right)\left(\eta_s-c_s(w,r) \right) \dd s.
\end{equation}
It is this map we will use to later provide conditions for existence and uniqueness of the system in \eqref{eq:life_cycle_explicit}, but will also play a role in our proofs of general equilibrium. 

\begin{prop}\label{prop:life_cycle_wealth_stability}
    Suppose the utility functions $u_1,\, u_2$ satisfy Assumption~\ref{ass:utility_reg_intro}, and $\eta\in \scL^0(\Omega;C_L)$  given stochastic process.  Then, for any given $w,\, v \in \scL^0(\Omega;C_L)$ and $r,\,h\in C_L$, it holds that
    \begin{equation}\label{eq:apriori_bound}
     \sup_{w\in \scL^0(\Omega;C_L)} \|\Theta_L(w,r)\|_{\msL^\infty_L}  \leq (1+L)\exp(\|r\|_{\msL^\infty_L}L) \left(|w_0| +  \|\eta\|_{\msL^\infty_L}+ \kappa \right),\quad \mbP\text{-a.s.}
    \end{equation}
    and
    \begin{equation*}
    \begin{aligned}
         \|\Theta_L(w,r)- \Theta_L(v,r)\|_{\msL^\infty_L} &\leq L C \kappa \exp(\|r\|_{\msL^\infty_L} L) \|w-v\|_{\msL^\infty_L}, \quad &&\mbP\text{-a.s.}
         \\
          \|\Theta_L(w,r)- \Theta_L(w,h)\|_{\msL^\infty_L} &\leq L C \kappa \exp(2(\|r\|_{\msL^\infty_L}+\|h\|_{\msL^\infty_L}) L) \|r-h\|_{\msL^\infty_L}, \quad &&\mbP\text{-a.s.} 
    \end{aligned}
    \end{equation*}
\end{prop}
  \begin{proof}
      From the definition of $\Theta_L$ in \eqref{eq:Theta map}, we begin by observing that 
      \begin{equation}\label{eq:sol_map_sup_bdd}
          \|\Theta_L(w,r)\|_{\msL^\infty_L} \leq \exp(L\|r\|_{\msL^\infty_L}) |w_0|+L\exp(L\|r\|_{\msL^\infty_L}) \left( \|\eta\|_{\msL^\infty_L} +\|c(w,r)\|_{\msL^\infty_L}\right) .  
      \end{equation}
     Invoking the bound from \eqref{eq:life_cycle_consump_bdd} we see that
      \[
      \|\Theta_L(w,r)\|_{\msL^\infty_L}  \leq \exp(L\|r\|_{\msL^\infty_L}) |w_0|+ L\exp(L\|r\|_{\msL^\infty_L}) \left(\|\eta\|_{\msL^\infty_L} + \kappa  \right),\quad \mbP\text{-a.s.}
      \]
      From this estimate, we collect terms and conclude that \eqref{eq:apriori_bound} holds.
      
      We continue on to prove 
      stability. To this end, let us first fix two random paths $w,\,v$ in $\scL^2(\Omega;C_L)$, and consider the difference 
      \begin{equation*}
         \Theta_L(w,r)_t- \Theta_L(v,r)_t = \int_0^t \exp\left(\int_s^t r_u \dd u\right)(c_s(w,r)-c_s(v,r))\dd s. 
      \end{equation*}
      Using the stability in wealth obtained for the consumption, see \eqref{eq:consump_wealth_stable}, we conclude that  
      \begin{equation*}
          \|\Theta_L(w,r)- \Theta_L(v,r)\|_{\msL^\infty_L} \leq L \lambda  \kappa^2 \exp(2\|r\|_{\msL^\infty_L} L) \|w-v\|_{\msL^\infty_L},\quad \mbP\text{-a.s.} 
      \end{equation*}
      At last we consider the stability of the mapping $r\mapsto \Theta_T(w,r)$. Again, we observe that given $r,\,h\in C_L$, one has
      \begin{align*}
                  \Theta_L(w,r)_t- \Theta_L(w,h)_t = &\,  \int_0^t \left[\exp\left(\int_s^t r_u \dd u\right)-\exp\left(\int_s^t h_u \dd u\right) \right]w_s
                  \\
                  &\, + \exp\left(\int_s^t h_u \dd u\right)\left[c_s(w,r)-c_s(w,h)\right]\dd s. 
      \end{align*}
      Now applying the stability in interest rate of the consumption policy, \eqref{eq:consump_interest_stable}, it follows that
      \[
      \|\Theta_L(w,r)- \Theta_L(w,h)\|_{\msL^\infty_L} \leq L \lambda  \kappa \exp(2(\|r\|_{\msL^\infty_L}+\|h\|_{\msL^\infty_L}) T) \|r-h\|_{\msL^\infty_L},\quad \mbP\text{-a.s.}
      \]
      which concludes the proof.
  \end{proof}

Combining the a-priori estimates and stability results above we obtain existence and uniqueness of solutions to the SDE system \eqref{eq:life_cycle_explicit}, conditional on well-posedness of the income process. This is simply a consequence of choosing $L^*$ sufficiently small in Proposition \ref{prop:life_cycle_wealth_stability}, such that a standard argument of Banach's fixed point theorem may be applied, and we therefore omit the full proof here. From an economic point of view, this says that there exists a unique solution to the life-cycle model provided the lifespan of all individuals is sufficiently small as a function of the other parameters in the model, namely the regularity constant $\kappa$ from Assumption~\ref{ass:utility_reg_intro}, the bequest motive $\lambda$ and the interest rate $\|r\|_{\msL^\infty_L}$.
%

%
\begin{cor}[Existence and Uniqueness of Optimal Wealth Paths] \label{cor:life_cycle_well_posed}
 Let $L>0$, $u_1$ and $u_2$ satisfy Assumption~\ref{ass:utility_reg_intro} for some $\kappa>0$ and $\eta \in \msL^0(\Omega;C_L)$ be a given income process.
  Then, for any interest rate path $r\in C([0,L];\RR)$, there exists an $L^* \coloneqq L^*(\kappa,\lambda,\|r\|_{\msL^\infty_L})\in (0,L)$ such that there exists a unique strong solution $w \in \scL^0(\Omega;C_{L^*}) $ to the wealth and savings equation
  \begin{equation}\label{eq:life_cycle_sde_sol}
       \begin{aligned}
        \dd w_t& = (r_tw_t +\eta_t -c_t(w))\dd t,\quad w_0\sim \rho_w
        \\
        c_t(w)&= (u'_1)^{-1} \left( \lambda \EE \left[ \exp\left( \int_t^L r_s-\delta \dd s \right) u'_2(w_L) |\cF_t\right]\right).
  \end{aligned}
  \end{equation}
Furthermore, if $\mu,\, \sigma$ are coefficients such that there exists a unique, strong solution to the income SDE
 \begin{equation}\label{eq:life_cycle_sde_income}
     \dd \eta_t = \mu_t(\eta_t)\dd t + \sigma_t(\eta_t)\dd B_t, \quad \eta\tzero = \eta_0 \in \msL^0(\Omega;\mbR)
 \end{equation}
then the solution to \eqref{eq:life_cycle_sde_sol} is the unique solution to the optimal control problem \eqref{eq:life_cycle_problem}. 
\end{cor}
\begin{proof}
    The local existence and uniqueness of solutions to \eqref{eq:life_cycle_sde_sol} follows directly from Banach's fixed point theorem (see e.g. \cite{Ciesielski07}), by applying the stability estimates from Proposition \ref{prop:life_cycle_wealth_stability}.

    The fact that this unique solution $w$ is the unique solution to the optimal control problem \eqref{eq:life_cycle_problem} when the income process is described by the SDE \eqref{eq:life_cycle_sde_income} 
 follows directly from Theorem~\ref{thm:life_cycle_optimal_rep}.
\end{proof}
%
%

%

\section{General Equilibria of the Life-Cycle Model}\label{sec:life_cycle_gen_eq}

We show that a unique general equilibrium interest rate exists for the life-cycle model. 
While a life-cycle model with finite lifespans in itself is not particularly economically relevant, the mathematical results obtained on the way allow us to obtain existence and uniqueness of a general equilibrium in the overlapping generations model.

\subsection{Dynamics of the equilibrium interest rates}

We define the notion of general equilibria in the life-cycle model.

\begin{defn}[General Equilibrium in the life-cycle Model] \label{def:life_cycle_gen_eq}
	Let $\{w_t^*\}_{t\in [0,L]}$ be an optimal wealth process solving \eqref{eq:life_cycle_problem}. Given a continuous path  $K\in C([0,L];\RR) $ we say that an interest rate path $r\in C([0,L];\RR)$ is a \emph{general equilibrium interest rate} for the associated life-cycle model if,
\begin{equation}\label{eq:capital_supply}
		\EE[w_t^*]=K_t,\quad \forall t\in [0,L]. 
	\end{equation}
\end{defn}

Our first step is to formulate a fixed point equation that must be satisfied by any general equilibrium interest rate for the life-cycle model. To this end, we view the consumption and resulting wealth processes as functionals of the interest rate path $C([0,L];\mbR)\ni r \mapsto (w(r),c(w(r),r))$, see Remark~\ref{rem:conumption_dependence}. Corollary~\ref{cor:life_cycle_well_posed} shows that this map is well defined for sufficiently small lifespans $L\in (0,L^*)$ with $L^*$ obtained therein.  The following proposition obtains an a priori functional identity which is in one-to-one correspondence with the property of being a general equilibrium rate. 

%

\begin{prop}\label{prop:life_cycle_rate_props}
 Let $L>0$, $\eta \in \msL^0(\Omega;C_L)$ be a given income process, and consider a function $K\in C^1([0,L];\RR)$. Suppose there exists a pair $(w,c)$ which solve  \eqref{eq:life_cycle_sde_sol} on $[0,L]$ such that $\mbE[w_0]=K_0 \in \mbR$.  
Then, there exists a general equilibrium according to Definition \ref{def:life_cycle_gen_eq} for capital allocation $t\mapsto K_t$, if and only if there exists a continuous path $r\in C([0,L];\RR)$ which satisfies  the following functional equation
    \begin{equation}\label{eq:interest_rate}
        K_t r_t = \dot{K}_t+  \EE[c_t(r)]- \EE[\eta_t] ,\quad \forall \,t \in [0,L].
    \end{equation}
   In particular, if $K_t\equiv 0$ for all $t$, then $r$ must solve the following functional equation
    \[
    \mbE[c_t(r)] = \mbE[\eta_t].  
    \]
\end{prop}
%

%
%
\begin{proof}
Assume first that there exists an $r\in C([0,L];\RR)$ such that the economy is in general equilibrium according to Definition \ref{def:life_cycle_gen_eq} for a path $t\mapsto K_t$, i.e. that  $\EE[w_t(r)]=K_t$ for all $t\in [0,L]$.  It follows that the expected wealth must satisfy the following equation 
\begin{equation*}
     \mbE[w_t] =K_t = K_0+\int_0^t K_sr_s \dd s -\int_0^t \EE[c_s(r)]\dd s +\int_0^t \EE[\eta_s] \dd s,\quad \text{for all }\, t\in [0,L] 
\end{equation*}
Differentiating this equation with respect to $t$ and then rearranging, it follows that in equilibrium the interest rate path $r:[0,L]\rightarrow \RR$ must satisfy the functional equation 
\begin{equation*}
    K_t r_t = \dot{K}_t+ \left( \EE[c_t(r)]- \EE[\eta_t] \right),\quad \forall \,t \in [0,L]. 
\end{equation*}
On the other hand, if $r\in C([0,L];\RR)$ satisfies \eqref{eq:interest_rate}, we know that the wealth dynamics  satisfy: 
\begin{equation*}
    w_t = w_0 + \int_0^t r_s w_s \dd s -\int_0^t c_s(r)\dd s +\int_0^t \eta_t. 
\end{equation*}
Taking expectations on both sides, using the assumption that $\EE[w_0]=K$, and inserting the relation that $ \EE[c_t(r)]- \EE[\eta_t] = r_tK_t-\dot{K}_t$
\begin{equation*}
       \EE [w_t] = K_0 + \int_0^t r_s\EE[w_s] \dd s - \int_0^t 
       \left(r_sK_s-\dot{K}_s  \right)\dd s. 
\end{equation*}
By rearranging the terms, and defining $Z_t=\EE[w_t]-K_t$ we see that this equation can be formulated in terms of $Z$ as the following equation
\begin{equation}
    Z_t=\int_0^t r_s Z_s \dd s.
\end{equation}
It is clear that the only solution to this linear equation is $Z_t\equiv 0$ for all $t$, and thus $\EE[w_t]=K_t$ if the interest rate $r$ satisfies \eqref{eq:interest_rate}. This concludes the proof. 
\end{proof}

The identity \eqref{eq:interest_rate} has a natural economic interpretation. If $K_t =0$ for all $t\in [0,L]$ then thinking of $r$ as the price of consumption, aggregate consumption must equal aggregate income at all times to keep zero net wealth in the economy. 
If $K_t \neq 0$ then aggregate consumption must exceed aggregate income by a factor proportional to the interest rate in order to meet the capital requirements. Note that, we implicitly set the depreciation rate to zero in our case.

With the above dynamics of the equilibrium interest rate, we can also analyse the impact of a change in the capital supply $K$ on the interest rate. The following proposition gives this expression in terms of the Fr\'echet derivative of consumption with respect to the interest rate. 
\begin{prop}\label{prop:life_cycle_interest_rate_derivative}
    Let $(w,c)$ be a solution to \eqref{eq:life_cycle_sde_sol} $K_{\,\cdot\,}\equiv K \in \mbR$ and $r(K)\coloneqq \in C([0,L];\mbR)$ be an interest rate path satisfying \eqref{eq:interest_rate}. Then, for any $t\in [0,L]$, it holds that
\begin{equation}\label{eq:r sens K}
\frac{\dd}{\dd K} r_t(K) = \begin{cases}
    (\EE[D_r c_t (r)]-K\mbI)^{-1}r_t, & K\neq 0,\\
    \EE[D_r c_t (r)]^{-1}0, & K=0.
\end{cases},
\end{equation}
where $\mbI$ is the identity operator, $D_r c_t(r)$ denotes the Fr\'echet derivative of $c_t(r)$ with respect to the path $r$ and the formula holds whenever this exists and the required operator is invertible. 
\end{prop}
\begin{proof}
Starting from \eqref{eq:interest_rate}, we first assume that $K\neq 0$ and compute the derivative of $r_t$ with respect to $K$, using the chain rule to see that we must have
\begin{equation*}
    r_t  + K \frac{\dd}{\dd K} r_t = \EE[D_rc_t(r)]\frac{\dd}{\dd K} r_t\quad \iff \quad \left(\EE[D_rc_t(r)] - K \mbI\right)\frac{\dd}{\dd K} r_t = r_t.
\end{equation*}
%
 %
 So that rearranging we obtain the claimed expression when $K\neq 0$.
 
 When $K=0$, we start again from \eqref{eq:interest_rate} to see that one must have
 \begin{equation*}
     \EE[D_rc_t(r)]\frac{\dd}{\dd K} r_t = 0,
 \end{equation*}
 which again gives the claimed expression after taking the inverse of $ \EE[D_rc_t(r)]$ on both sides.
\end{proof}
\subsection{Marshallian stability in the life-cycle}

The general equilibrium interest rate is found at the balance between aggregate consumption and aggregate income, and it confirms that in order for the interest rate to stay positive  aggregate consumption must be greater than aggregate income.
Furthermore, under the assumption that the  income process $t\mapsto \eta_t$ is a 
 (pathwise) continuous process, the wealth process must be a (pathwise) continuous process, and thus the interest rate is a continuous path $t\mapsto r_t$.

\begin{lem}\label{lem:LC_rate_stable} Let $L>0$, $r,\,\tilde{r} \in C([0,L];\mbR)$ be two continuous deterministic interest rate processes, $w_0,\,\tilde{w}_0 \in \msL^0(\Omega;\mbR)$ and $\eta \in \msL^0(\Omega;C_L)$ be a given income process. Then, letting $(w(r),c(w(r),r))$ and $(w(\tilde{r}),c(w(\tilde{r}),\tilde{r}))$ be associated solutions to the life-cycle equation \eqref{eq:life_cycle_sde_sol} on $[0,L]$, with initial data $w(r)\tzero=w_0$ and $w(\tilde{r})\tzero =\tilde{w}_0$, one has
\begin{equation}\label{eq:wealth_apriori_bound}
\begin{aligned}
     \|w(r)\|_{\msL^\infty_L}  &\leq  (1+L)e^{\|r\|_{\msL^\infty_L}} \left(|w_0| +  \|\eta\|_{\msL^\infty_L}+ \kappa \right),
\end{aligned}
     \quad \mbP\text{-a.s.}
    \end{equation}
Furthermore, the following bounds hold
\begin{equation}\label{eq:wealth_rate_stability}
\begin{aligned}
    \|w(r)-w(\tilde{r})\|_{\msL^\infty_L} \lesssim \,A \bigg( \, &|w_0-\tilde{w}_0|\\
    &+ L\bigg( (1+L)\left(|w_0| +  \|\eta\|_{\msL^\infty_L}+ \kappa \right)  +  \lambda  \kappa \, \bigg)\|r-\tilde{r}\|_{\msL^\infty_L}\bigg),
    \end{aligned}
\end{equation}
and
\begin{multline}\label{eq:consump_rate_stability}
     \|c(w(r),r)-c(w(\tilde{r}),\tilde{r})\|_{\msL^\infty_L} \lesssim \, \lambda \kappa (1+\kappa) Ae^{L\|\tilde{r}\|_{\msL^\infty_L}}   \bigg( |w_0-\tilde{w}_0|\\
     	+ L\bigg( (1+L)\left(|w_0| +  \|\eta\|_{\msL^\infty_L} + \kappa \right)  +  \lambda  \kappa \, \bigg)\|r-\tilde{r}\|_{\msL^\infty_L} \bigg).
\end{multline}
Here we have used the shorthand notation
\begin{equation*}
    A\coloneqq A(r,L,\lambda,\kappa) \coloneqq \exp\left(L\Big(\|r\|_{\msL^\infty_L} + \lambda \kappa^2 e^{(1+L)\|r\|_{\msL^\infty_L}  }\Big) \right).
\end{equation*}

   \end{lem}
    \begin{proof}
The first bound, \eqref{eq:wealth_apriori_bound} follows directly by applying \eqref{eq:apriori_bound} to the fixed point $\Theta(w,r)=w(r)$, where $\Theta$ was defined in \eqref{eq:Theta map}. To show \eqref{eq:wealth_rate_stability}, first note that for all $0\leq s <t \leq L$, we have
         \begin{multline*}
        w_{s,t}(r) -w_{s,t}(\tilde{r})
        = w_s(r)-w_{s}(\tilde{r})
        \\+\int_s^t (r_u-\tilde{r}_u) w_u(r) + \tilde{r}_u \left(w_u(r)-w_u(\tilde{r})\right) + \left(c_u(w(r),r)-c_u(w(\tilde{r}),\tilde{r}) \right) \dd u.
    \end{multline*}
    %
%
    %
    So that using \eqref{eq:consump_interest_stable}, \eqref{eq:consump_wealth_stable} and \eqref{eq:wealth_apriori_bound} yields 
    \begin{align*}
        \|w(r)-w(\tilde{r})\|_{\msL^\infty_{[s,t]}} \leq & |w_s(r)-w_s(\tilde{r})|+ |t-s|\|w(r)\|_{\msL^\infty_L}\|r-\tilde{r}\|_{\msL^\infty_L} + |t-s|\|\tilde{r}\|_{\msL^\infty_L}   \|w(r)-w(\tilde{r})\|_{\msL^\infty_{[s,t]}}\\
        &+ |t-s|\|c(w(r),r)-c(w(r),\tilde{r}))\|_{\msL^\infty_{[s,t]}}\\
        &+ |t-s|\|c(w(r),\tilde{r}) -c(w(\tilde{r}),\tilde{r})\|_{\msL^\infty_{[s,t]}}\\
        \leq \, & |w_s(r)-w_s(\tilde{r})|+ |t-s| (1+L)e^{\|r\|_{\msL^\infty_L} } \left(|w_0| +  \|\eta\|_{\msL^\infty_L} + \kappa \right)\|r-\tilde{r}\|_{\msL^\infty_L}  \\
        &+ |t-s|\|\tilde{r}\|_{\msL^\infty_L}  \|w(r)-w(\tilde{r})\|_{\msL^\infty_{[s,t]}}\\
        &+ |t-s| \lambda  \kappa e^{L\left(\|r\|_{\msL^\infty_L} +\|\tilde{r}\|_{\msL^\infty_{L}} \right)} \| r- \tilde{r}\|_{\msL^\infty_{L}} \\
        &+ |t-s| \lambda \kappa^2 e^{L\|r\|_{\msL^\infty_{L}} } \|w(r)-w(\tilde{r})\|_{\msL^\infty_{[s,t]}}.
    \end{align*}
Defining the quantities,
\begin{equation*}
     \mfA(w_0,\eta,\kappa) \coloneqq |w_0| +  \|\eta\|_{\msL^\infty_L}+ \kappa \quad \text{and}\quad \mfB(L,r,\tilde{r}) \coloneqq e^{L\left(\|r\|_{\msL^\infty_L}+\|\tilde{r}\|_{\msL^\infty_L}\right)}
\end{equation*}
and rewriting the above estimate, gives the following bound
\begin{align*}
    \|w(r)-w(\tilde{r})\|_{\msL^\infty_{[s,t]}} \leq & \, |w_s(r)-w_s(\tilde{r})|\\
    &+|t-s|\left( (1+L)e^{\|r\|_{\msL^\infty_L}} \mfA(w_0,\eta,\kappa) +  \lambda  \kappa \mfB(L,r,\tilde{r})\, \right)\|r-\tilde{r}\|_{\msL^\infty_L}\\
   & + |t-s|\bigg(\|\tilde{r}\|_L + \lambda \kappa^2 e^{L\|r\|_{\msL^\infty_L} }\bigg)\|w(r)-w(\tilde{r})\|_{\msL^\infty_{[s,t]}}.
\end{align*}
So that defining
\begin{equation*}
    L_1 \coloneqq \frac{1}{2\bigg(\|\tilde{r}\|_{\msL^\infty_L} + \lambda \kappa^2 \exp\left(L\|r\|_{\msL^\infty_L}  \right)\bigg)}
\end{equation*}
and rearranging the proceeding inequality for any $L\in (0,L_1]$ and $0\leq s<t\leq L$, we obtain
\begin{align*}
     \|w(r)-w(\tilde{r})\|_{\msL^\infty_{[s,t]}} \leq& \, 2|w_s(r)-w_s(\tilde{r})|\\
     &+2 |t-s|\bigg( (1+L)e^{\|r\|_{\msL^\infty_L}} \mfA(w_0,\eta,\kappa)  +  \lambda  \kappa \mfB(L,r,\tilde{r})\, \bigg)\|r-\tilde{r}\|_{\msL^\infty_L}.
     \end{align*}
In particular, taking $s=0$ and $t= L_1$,
\begin{align*}
     \|w(r)-w(\tilde{r})\|_{\msL^\infty_{[0,L_1]}} \leq& \, 2|w_0-\tilde{w}_0|\\
     &+2 L_1\bigg( (1+L)e^{\|r\|_{\msL^\infty_L}} \mfA(w_0,\eta,\kappa)  +  \lambda  \kappa \mfB(L,r,\tilde{r})\, \bigg)\|r-\tilde{r}\|_{\msL^\infty_L}.
\end{align*}
and furthermore, since $  |w_{L_1}(r)-w_{L_1}(\tilde{r})|  \leq  \|w(r)-w(\tilde{r})\|_{\msL^\infty_{[0,L_1]}}$,
\begin{align*}
    |w_{L_1}(r)-w_{L_1}(\tilde{r})| \leq&  \, 2|w_0-\tilde{w}_0|\\
     &+2 L\bigg( (1+L)e^{\|r\|_{\msL^\infty_L}} \mfA(w_0,\eta,\kappa)  +  \lambda  \kappa \mfB(L,r,\tilde{r})\, \bigg)\|r-\tilde{r}\|_{\msL^\infty_L}.
\end{align*}
So that iterating the bound, we find that for any $L>0$,
\begin{align*}
    \|w(r)-w(\tilde{r})\|_{\msL^\infty_L} \leq &\, 2^{N_L} |w_0-\tilde{w}_0|\\
    &+ 2^{N_L}L_1\bigg( (1+L)e^{\|r\|_{\msL^\infty_L}} \mfA(w_0,\eta,\kappa)  +  \lambda  \kappa \mfB(L,r,\tilde{r})\, \bigg)\|r-\tilde{r}\|_{\msL^\infty_L},
\end{align*}
where
\begin{equation*}
    N_L \coloneqq \left\lceil 2L\bigg(\|\tilde{r}\|_{\msL^\infty_L} + \lambda \kappa^2 \exp\left(L\|r\|_{\msL^\infty_L} \right)\bigg)\right\rceil
\end{equation*}
is the number of intervals of length $L_1$ required to cover $[0,L]$. This concludes the proof of \eqref{eq:wealth_rate_stability}.

To see that \eqref{eq:consump_rate_stability} also holds, we apply \eqref{eq:consump_interest_stable} and \eqref{eq:consump_wealth_stable} along with the triangle inequality, 
\begin{align*}
     \|c(w(r),r)-c(w(\tilde{r}),\tilde{r})\|_{\msL^\infty_L} \leq \, &
     \|c(w(r),r)-c(w(r),\tilde{r})\|_{\msL^\infty_L} +  \|c(w(r),\tilde{r})-c(w(\tilde{r}),\tilde{r})\|_{\msL^\infty_L}\\
     \leq \, & \lambda  \kappa e^{L(\|r\|_{\msL^\infty_L}+\|\tilde{r}\|_{\msL^\infty_L})} \| r-\tilde{r}\|_{\msL^\infty_L} + \lambda \kappa^2 e^{L\|r\|_{\msL^\infty_L}} \|w(r)-w(\tilde{r})\|_{\msL^\infty_L}\\
     \lesssim\,& \lambda \kappa  e^{L(\|r\|_{\msL^\infty_L}+\|\tilde{r}\|_{\msL^\infty_L})}\left(\| r-\tilde{r}\|_{\msL^\infty_L} + \kappa \|w(r)-w(\tilde{r})\|_{\msL^\infty_L}\right).
\end{align*}
Applying \eqref{eq:wealth_rate_stability} and consolidating terms completes the proof of \eqref{eq:consump_rate_stability}.
    \end{proof}
As in Section~\ref{sec:life_cycle}, the stability result of Lemma~\ref{lem:LC_rate_stable} allows us to obtain existence and uniqueness of a market clearing interest rate in the life-cycle model. However, since our main focus is general equilibria in the overlapping generations model we do not give a detailed proof. 

\begin{cor}[General Equilibria in the life-cycle Model]\label{cor:life_cycle_equilibrium}
    Let $L>0$, $K>0$, $\eta \in \msL^0(\Omega;C_L)$ be a given income process and suppose there exists a pair $(w,c)$ which solve  \eqref{eq:life_cycle_sde_sol} on $[0,L]$ such that $\mbE[w_0]=K \in \mbR\setminus\{0\}$. Then, for a bequest motive $\lambda>0$ and/or parameter $\kappa>0$ sufficiently small, depending on $L$ and $\eta$, there exists a unique market clearing interest rate $\bar{r} \in C([0,L];\mbR)$ for the life-cycle equations \ref{eq:life_cycle_sde_sol}. 
\end{cor}
  \begin{proof}[Sketch of Proof]
     The result follows from Banach's fixed point theorem applied to the map
     \begin{equation*}
         r \mapsto \frac{1}{K} \left(\mbE[c(r)] -\mbE[\eta]\right).
     \end{equation*}
   Recall that this expression comes from choosing a constant capital allocation $K_t=K$ in Proposition~\ref{prop:life_cycle_rate_props}. Using the stability estimate \eqref{eq:consump_rate_stability}, we can choose $\lambda$ and/or $\kappa$ sufficiently small to obtain a contraction.
  \end{proof}
\begin{rem}
    In our analysis we have assumed that the capital supply is constant, i.e. $K_t=K\in \RR$ for all $t\in [0,L]$. This assumption simplifies the fixed point formulation for the equilibrium interest rate, making it easier to convey the techniques and result. Nevertheless, the approach can be straightforwardly generalized to the case of a time-varying capital by introducing some more notational complexity, provided that $K_t\neq 0$ for all $t$. 
\end{rem}
\section{Overlapping generations}\label{sec:OLG}
So far we have focused on the life-cycle of a heterogeneous population of individuals all \emph{born} at time $0$ and \emph{ceasing economic activity} at time $L>0$.  In this section we extend this model to an economy of overlapping generations. At each moment in time, a new generation with infinitely many individuals are born, with the same \emph{lifespan} $L>0$ as the previous generation. A  flow of probability distributions will describe the age distribution in the total population, which we refer to  as a \emph{flow of demographic measures}. 

The wealth, consumption and income at time $t \in \mbR$ of an individual born at date $b \in \mbR$  under interest rate path $r$ and income modelled as an It\^o process can be written, for $b\in \mbR$ and $t\in [b,b+L]$, 
\begin{align}
        w^b_t(r) &= w^b_b + \int_b^t r_u w^{b}_u(r) \dd u + \int_b^t \eta^b_u  \dd u - \int_b^t c_u^b(w^b(r),r) \dd u, \label{eq:olg_wealth}
        \\
          c_t^{b}(w^b,r) &= (u_1')^{-1}\left(\lambda \exp\left(\int_t^{L+b} (r_u-\delta)\dd u\right)\EE\left[ u'_2\left(w_{L+b}^{b}\right)|\cF_t^b \right]\right),  \label{eq:olg_consumption}\\
\eta^b_u &= \eta^b_b +\int_{b}^u \mu_r(\eta^b_r) \dd r + \int_{b}^u \sigma_r(\eta^b_r) \dd B^b_r, \label{eq:olg_income}
\end{align}
where $\{B^b\}_{b\in \mbR}$ is a collection of i.i.d. Brownian motions, each running on the interval $[b,b+L]$ and such that $B^b_{b}=0$.\blue{ Recall  also Remarks~\ref{rem:fubini_measurability} and \ref{rem:fubini_lln} concerning the Fubini extension and resulting measurability properties of the map $b\mapsto (w^b,c^b,\eta^b)$.}

Here, $w_{b}^{b}$ denotes the initial wealth of an individual born at time $b$ and $\eta^b_b$ their initial income. Later we will specify how these random variables are distributed. Note that Theorem~\ref{thm:life_cycle_optimal_rep} applied for each $b\in \mbR$ shows that solving the system \eqref{eq:olg_wealth}-\eqref{eq:olg_income} is equivalent for finding a maximising consumption policy for each generation.

It will be convenient to introduce notation grouping these objects into collections of processes indexed by $b\in \mbR$. Let us define
\begin{equation*}
    \bw \coloneqq \{w^b\}_{b\in \mbR},\quad \bc(\bw,r) \coloneqq \{c^b(w^b,r)\}_{b\in \mbR}, \quad \bmeta\coloneqq \{\eta^b\}_{b\in \mbR}\quad \text{and}\quad \bB \coloneqq \{B^b\}_{b\in \mbR},
\end{equation*}
where each process is defined on the interval $[b,b+L]$, as well as the sets of initial data
\begin{equation*}
    \bw_0 \coloneqq \{w^b_b\}_{b\in \mbR}\quad \text{and}\quad \bmeta_0 \coloneqq \{\eta^b_b\}_{b\in \mbR}.
\end{equation*}
We associate suitable norms to the wealth initial data and income processes, defining
\begin{align}
    \|\bw_0\|_{\msL^1_\omega\msL^\infty_{\mbR}} &\, \coloneqq  \mbE\Big[\sup_{b\in \mbR}|w^b_b|\Big], \label{eq:w_0_norm}\\
    \|\bw\|_{\msL^1_\omega\msL^\infty_{\mbR} \msL^\infty_{L}} &\, \coloneqq  \mbE\Big[\sup_{b\in \mbR}\|w^b\|_{\msL^\infty_{[b,b+L]}}\Big], \label{eq:wealth_norm}\\
     \|\bc(\bw,r)\|_{\msL^1_\omega\msL^\infty_{\mbR} \msL^\infty_{L}} &\, \coloneqq  \mbE\Big[\sup_{b\in \mbR}\|c^b(w^b,r)\|_{\msL^\infty_{[b,b+L]}}\Big], \label{eq:consumption_norm}\\
    \|\bmeta\|_{\msL^1_\omega\msL^\infty_{\mbR} \msL^\infty_{L}} &\, \coloneqq  \mbE\Big[\sup_{b\in \mbR}\|\eta^b\|_{\msL^\infty_{[b,b+L]}}\Big]. \label{eq:income_norm}
\end{align}
%
%

To fully describe the overlapping generations model we require a fourth actor, the flow of demographic measures $t\mapsto \nu_t \in \mcP(\mbR)$.
\begin{defn}\label{def:dem_flow}

    We say that a family of probability measures $\bnu = \{\nu_t\}_{t\in \mbR}$ is a \emph{flow of demographic measures} if the following hold:
    \begin{enumerate}[label=\roman*)]
        \item The map $\mbR \ni t\mapsto \nu_t \in \mcP(\mbR)$ is continuous with respect to the weak topology of measures.
        \item For each $t\in \mbR$ one has \footnote{Recall that the support of a measure is defined as the closure of all sets $A$ in $\sigma$-algebra such that $\nu(A)\neq 0$.
        }
        \begin{equation*}
            \supp(\nu_t) = [t-L,t].
        \end{equation*}
    \end{enumerate}
    We say that a family of probability measures $\bnu \coloneqq \{\nu_t\}_{t\in \mbR}$ is a flow of \emph{regular demographic measures} if the above conditions hold and
    \begin{enumerate}[label=\roman*)]
    \setcounter{enumi}{2}
        \item for each $t\in \mbR$ there exists a continuous function $b\mapsto n(t,b)$ such that for all measurable $A\subseteq [t-L,t]\subset \mbR$
        \begin{equation*}
            \int_A \nu_t(\dd b) = \int_A  n(t,b)\dd b.
        \end{equation*}
\item for all $b\in \mbR$ the map $t\mapsto n(t,b)$ is differentiable and it holds that
\begin{equation*}
    \sup_{t\in \mbR}\sup_{b\in [t-L,t]} |\partial_t n(t,b)|  <\infty.
\end{equation*}
    \end{enumerate}
    We equip the vector space of all regular demographic measures (viewed as a subset of the signed measures rather than probability measures) with the structure of a Banach space by defining the norm 
    \begin{equation}\label{eq:demography_norm}
        \|\bnu\|_{W^{1,\infty}_{\mbR} \TV_L } \coloneqq \sup_{t\in \mbR} \sup_{b\in [t-L,t]} |n(t,b)| + \sup_{t\in \mbR} \sup_{b\in [t-L,t]} |\partial_t n(t,b)|
    \end{equation}
\end{defn}

Letting $Y_t$ denote the birthday of randomly chosen individual alive at time $t$ we can interpret the demographic measure $\nu_t$ by seeing that for any $a\in [0,L]$,
\begin{equation*}
    \mbP(Y_t \in [t-a,t]) = \int_{t-a}^t  \nu_t(\dd b).
\end{equation*}
Furthermore, if we let $X_t$ be the age of a randomly chosen individual alive at time $t$ we observe the simple relation
\begin{equation*}
    X_t = t-Y_t.
\end{equation*}
Hence, we can delineate the following simple cases:
\begin{enumerate}
    \item \label{it:glob_stat}\textbf{Globally Stationary Populations} are described by the condition that the chance of being born at any time prior to the current time is constant; i.e. a population is globally stationary if
 \begin{equation*}
        \mbP(Y_t \in [t-a,t]) = \mbP(Y_s \in [s-a,s]) \quad \text{for all}\,t,\, s \in \mbR\,\,\text{and}\,\,  a\in [0,L].
    \end{equation*}
 This can be reformulated as a condition on the age of a random member of the population,
    \begin{equation*}
        \mbP(X_t \in [0,a]) = \mbP(X_s \in [0,a]) \quad \text{for all}\,\, t,\,s \in \mbR\,\,\text{and}\,\,  a\in [0,L].
    \end{equation*}
    These conditions can be written at the level of a \emph{flow of regular demographic measures} by requiring %
    \begin{equation}\label{eq:reg_stat_pop}
        n(t,b) = n(s,b-(t-s))\quad \quad \text{for all}\,\, t,\,s \in \mbR\,\,\&\,\,  b\in [t-L,t].
    \end{equation}
    So that in particular
    \begin{equation*}
    	n(t,b) = n(0,b-t)\quad \text{for all}\,\, t\in \mbR\,\,\&\,\, b\in [t-L,t].
    \end{equation*}
\item \label{it:loc_stat}\textbf{Locally Stationary Populations} are described by the condition that being born today or at any infinitesimally prior time are equal; i.e. we say that a population is locally stationary at time $t\in \mbR$ if
\begin{equation*}
    \frac{\dd}{\dd a} \mbP(Y_t \in [t-a,t])=0.
\end{equation*}
    \item \label{it:loc_grow}\textbf{Locally Increasing Populations} are described by the condition that being born today is more likely than any, infinitesimally close previous time; i.e. we say that a population is increasing at time $t \in \mbR$ if
    \begin{equation*}
       \frac{\dd}{\dd a} \mbP(Y_t \in [t-a,t]) >0. 
    \end{equation*}
   \item \label{it:loc_shrink} \textbf{Locally Decreasing Populations} are described by the condition that being born today is less likely that at any infinitesimal previous time; i.e. we say that a population is locally decreasing at $t\in \mbR$ if  
   \begin{equation*}
       \frac{\dd}{\dd a} \mbP(Y_t \in [t-a,t]) <0.
   \end{equation*}
\end{enumerate}
%
%
%
%
%



Given any \emph{flow of demographic measures} $t\mapsto \nu_t$ we can describe the age-weighted, aggregate wealth, consumption and income in the population as follows: 
\begin{align}
    W^L_t(r) &\coloneqq \int_{t-L}^{t} w_t^{b}(r) \dd \nu_t(b)\label{eq:aggregate_wealth}\\
    C^L_t(r,w) &\coloneqq \int_{t-L}^t c^{b}_t(w^b,r) \dd \nu_t(b) , \label{eq:aggregate_consumption} \\
    N^L_t &\coloneqq \int_{t-L}^t \eta^{b}_t \dd \nu_t(b). \label{eq:aggregate_income} 
\end{align}
In this setting, we redefine the market clearing condition in terms of the aggregate expected wealth.

\begin{defn}\label{def:olg_clearing_rate}
    Given a \emph{lifespan} $L>0$ and a capital supply $K\in C(\RR;\RR)$  we say that the associated overlapping generations model \eqref{eq:olg_wealth}-\eqref{eq:olg_income} and \eqref{eq:aggregate_wealth}-\eqref{eq:aggregate_income} is in \emph{general equilibrium} if the interest rate $r:\mbR\to \mbR$ is such that
    \begin{equation}\label{eq:olg_clearing}
        \mbW^L_t(r) \coloneqq \mbE\left[W^L_t(r)\right] = K_t \quad \text{for all}\,\, t\in \mbR.
        \end{equation}
        For brevity we sometimes say that $r$ is a \emph{general equilibrium interest rate} for the associated overlapping generations model.
\end{defn}
\begin{rem}
    We note that in the case $K_t\equiv 0$, the general equilibrium implies that $t\mapsto W_t(r)$ is first order stationary, i.e. $ \mbW^L_t(r)=\mbW^L_s(r)$  
\end{rem}


\subsection{Existence of General Equilibria in Globally Stationary Populations}\label{sec:olg_stationary_populations}
Theorem~\ref{thm:olg_well_posed} and Proposition~\ref{prop:olg_rate_stable} give general well-posedness and stability of the OLG model. In the remainder of the section we show that under suitable stationarity assumptions on the income trajectories, demographic processes and capital supply, there exists a constant general equilibrium interest rate.

\begin{defn}\label{def:stat_demographic}
     We say that a flow of demographic measures $\bnu \coloneqq \{\nu_t\}_{t\in \mbR}$ is globally stationary if there exists a measure $\nu \in \mcP([-L,0])$ such that for all $t\in \mbR$ and $A \in \mcB([t-L,t])$,
    \begin{equation*}
       \nu_t (A)= \nu(A-t) \quad \text{for all}\,\, t\in \mbR, \quad \text{where}\,\, A-t \coloneqq \{a \in A\,:\, a-t\}.
    \end{equation*}
    %
    %
\end{defn}
\begin{defn}\label{def:stat_income}
    We say that a family of integrable income processes $\bmeta \coloneqq \{\eta^b\}_{b\in \mbR}$ is stationary if for every $b,\, t\in \mbR$ and $h>0$ one has
    \begin{equation*}
        \mcF^b_t = \mcF^{b+h}_{t+h}\quad \text{and} \quad \EE[\eta^b_t] = \EE[\eta^{b+h}_{t+h}], 
    \end{equation*}
    where $\{\mcF^b_t\}_{t\in \mbR}$ is the natural filtration of the process $t\mapsto \eta^b_t$ for each $b\in \mbR$. %
\end{defn}
%
%
%
	%
	%
%
	%
%
%
It will be useful to introduce the following shift operator. For any $h\in \mbR$, we set
\begin{align*}
	\tau_h :\msL^1(\mbR^2;\mbR) &\to\msL^1(\mbR^2;\mbR)\\
	f^{\var}_{\var} &\mapsto f^{\var+h}_{\var+h}.
\end{align*}
Using this notation we have the following direct result on shifts in time of the optimal consumption in the stationary case.
\begin{lem}\label{lem:consumption_time_shift}
  Let $r:\mbR\to \mbR$ be any interest rate path, $\bnu$ be any flow of demographic measures, $\bmeta$ be a stationary family of income processes as prescribed by Definition~\ref{def:stat_income}, $L>0$ be any lifespan and suppose that a solution $(\bw,\bc) \coloneqq (\{w^b_t\}_{b,\,t\in \mbR},\{c^b_t\}_{b,t\, \in \mbR})$ exists to \eqref{eq:olg_wealth}-\eqref{eq:olg_income}. Then, for any $t,\, h\in \mbR$ it holds that
    \begin{equation*}
    \tau_hc_{t}^b(w,r)=c_t^b(\tau_h w,\tau_h r), \quad \PP-a.s.. 
    \end{equation*}
   
\end{lem}

\begin{proof}
    We see that 
    \begin{equation*}
        c_{t+h}^{b+h}(w,r) = (u_1')^{-1}\left(\lambda \EE\left[\exp\left(\int_{t+h}^{b+h+L} (r_u-\delta)\dd u\right) u'_2(w_{b+h+L}^{b})|\cF_{t+h}^{b+h}\right]\right)
    \end{equation*}
    By a simple change of variables inside the integral, together with the fact that $\cF_t^b = \cF_{t+h}^{b+h}$ we obtain
    \begin{equation*}
        c_{t+h}^{b+h}(w,r) = (u_1')^{-1}\left(\lambda \EE\left[\exp\left(\int_{t}^{b+T} (\tau_hr_{u}-\delta)\dd u\right) u'_2(\tau_h w_{b+L}^{b})|\cF_{t}^b\right]\right). 
    \end{equation*}
    Thus we find 
    \begin{equation*}
\tau_hc_{t}^b(w,r)=c_t^b(\tau_h w,\tau_h r), \quad \PP-a.s.. 
    \end{equation*}
\end{proof}

Using the above result, we are now ready to prove  a stationarity result for the birth-dependent wealth processes. 

\begin{prop}\label{prop:wealth_time_shift}
    Let $r:\mbR\to \mbR$ be any interest rate path, $\bnu$ be any flow of demographic measures, $\bmeta$ be a stationary family of income processes as prescribed by Definition~\ref{def:stat_income}, $L>0$ be any lifespan and suppose that a solution $(\bw,\bc) \coloneqq (\{w^b_t\}_{b,\,t\in \mbR},\{c^b_t\}_{b,t\, \in \mbR})$ exists to \eqref{eq:olg_wealth}-\eqref{eq:olg_income} with stationary initial wealth, i.e. $w^b_b \equiv w \in\msL^1(\Omega;\mbR)$. Then, for any $t,\, h\in \mbR$ it holds that
    \begin{equation*}
    	 \EE[\tau_h w_t^b(r)] = \EE[w_t^b( \tau_h r)].
    \end{equation*}
\end{prop}
\begin{proof}
We observe that 
\begin{equation*}
\tau_h w_t^b(r)= w_{t+h}^{b+h}(r)=w^{b+h}_{b+h}+\int_{b+h}^{t+h}r_u w_u^{b+h}(r)+\eta_u^{b+h}-c_u^{b+h}(w^b(r),r)\dd u.
\end{equation*}
So that by a simple change of variable in the integral
\begin{equation*}
\tau_h w_t^b(r)=w^{b+h}_{b+h}+\int_{b}^{t}r_{u+h} w_{u+h}^{b+h}(r)+\eta_{u+h}^{b+h}-c_{u+h}^{b+h}(w^b(r),r)\dd u. 
\end{equation*}
Now applying Lemma~\ref{lem:consumption_time_shift}  gives
\begin{equation*}
\EE[\tau_h w_t^b(r)]=\EE[w^{b+h}_{b+h}]+\int_{b}^{t}\tau_h r_{u} \,\EE[\tau_h w_{u}^b(r)]+\EE[\tau_h \eta_{u}^b]-\EE[c_{u}^b\big(\tau_h w^b(r),\tau_h r\big)]\dd u. 
\end{equation*}

So now using that we have assumed $w^{b+h}_{b+h}=w^b_b = w$ for all $b,\, h\in \mbR$ and that $\EE[\tau_h \eta_{u}^b]=\EE[ \eta_{u}^b]$, we see that $\EE[\tau_h w_t^b(r)]$ solves the expected value of the equation \eqref{eq:olg_wealth} with interest rate  $t\mapsto r_{t+h}$  which concludes the proof. 
\end{proof}

\begin{prop}\label{prop:firstorderstationary}
     Let $r:\mbR\to \mbR$ be stationary interest rate path (i.e. $r_t\equiv r\in \mbR$ for all $t\in \mbR$), $\bnu$ be stationary flow of demographic measures as described by Definition~\ref{def:stat_demographic}, $\bmeta$ be a stationary family of income processes as prescribed by Definition~\ref{def:stat_income}, $L>0$ be any lifespan and suppose that a solution $(\bw,\bc) \coloneqq (\{w^b_t\}_{b,\,t\in \mbR},\{c^b_t\}_{b,t\, \in \mbR})$ exists to \eqref{eq:olg_wealth}-\eqref{eq:olg_income} with stationary initial wealth, i.e. $w^b_b \equiv w \sim \rho^w \in \mcP(\mbR)$. Then the aggregate wealth process $W$ is first-order stationary, i.e.
   \begin{equation*}
      \mbW^L_t =  \EE[W^L_t]=\EE[W^L_{t+h}] = \mbW^L_{t+h}, \quad \forall t,\, h\in \RR. 
   \end{equation*}
\end{prop}
\begin{proof}
Applying a simple change of variables followed by  Proposition~\ref{prop:wealth_time_shift}, we observe that 
\begin{align*}
     \mbE\left[W^L_{t+h}(r)\right] = \mbE\left[\int_{t+h-L}^{t+h} w_{t+h}^{b}(r) \nu(\dd b)\right]  
     = \mbE\left[\int_{t-L}^{t} w_{t+h}^{b+h}(r) \nu(\dd b)\right]  =  \mbE\left[\int_{t-L}^{t} \tau_h w_{t}^{b}(r) \nu(\dd b)\right] 
     \\
     = \mbE\left[\int_{t-L}^{t}  w_{t}^{b}(r_{\,\cdot\,+h}) \nu(\dd b)\right] 
       = \mbE\left[\int_{t-L}^{t}  w_{t}^{b}(r) \nu(\dd b)\right]
     = \mbE\left[W^L_t(r)\right].
\end{align*}
where, in the penultimate line we used that since we took $r$ to be constant we have $r_{\,\cdot\,+h}=r$.
\end{proof}

The above proposition allows us to conclude that if the interest rate is constant then the expectation of aggregate wealth is constant.

\begin{thm}\label{th:olg_stationary_existence_crra}
     \blue{Assume that the conditions of Prop.\ref{prop:firstorderstationary} holds, and also that the utility functions $u_1$ and $u_2$ satisfies the conditions of Proposition \ref{prop:range_of_expected_w}}.  Suppose further that $\nu$ is a regular flow of demographic measures which describe a globally stationary population. Then for any fixed $K\geq 0$ there exists at least one constant equilibrium interest rate  $r\in \RR$, i.e. which is such that 
      \[ 
      \WW^L_t(r) = K,\quad \forall \,t. 
      \]
\end{thm}
\begin{proof}
 It is clear from the  assumed  existence of the dynamics of $w$ described in \eqref{eq:olg_wealth} that restricting to $r\in \RR$, then $r\mapsto w(r)$ is continuous, and thus $r\mapsto \WW_t^L(r)$ is continuous. 
Applying the results from Proposition \ref{prop:range_of_expected_w}, and using the fact that $\nu$ is time independent,  we see that for each $t$
\[
 \lim_{r\rightarrow \infty} \WW^L_t(r) = \infty, \quad \lim_{r\rightarrow -\infty} \WW^L_t(r)\leq 0. 
\]
 and its range spans at least $[0,\infty)$. Furthermore, from Proposition \ref{prop:firstorderstationary}, we know that for a constant interest rate $r\in \RR$ then $\WW^L_t(r)$ is first order stationary. By the the intermediate value theorem, we can therefore conclude that for any given $K\geq 0$ there exists at least one constant interest rate $r\in \RR$ providing general equilibrium according to definition \ref{def:olg_clearing_rate}. 
\end{proof}

\blue{
\begin{thm}\label{th:olg_stationary_existence}
    Assume that the conditions of Prop.~\ref{prop:firstorderstationary} hold, and that the utility functions $u_1$ and $u_2$ satisfy Assumption \ref{ass:utility_reg_intro} (In particular, boundedness of the inverse and their derivatives leads to  bounded consumption $0\le c_t \le \kappa$ as seen from \eqref{eq:olg_intro_consumption}). Suppose further that the expected initial wealth is strictly positive, i.e., $\mbE[w_b^b] = \mbE[w_0] > 0$, and that $\nu$ is a regular flow of demographic measures which describe a globally stationary population. Then for any fixed $K> 0$ there exists at least one constant equilibrium interest rate $r\in \RR$, i.e., which is such that 
    \[ 
    \WW^L_t(r) = K,\quad \forall \,t. 
    \]
\end{thm}
\begin{proof}
    It is clear from the assumed existence of the dynamics of $w$ described in \eqref{eq:olg_wealth} that restricting to $r\in \RR$, the map $r\mapsto w(r)$ is continuous, and thus $r\mapsto \WW_t^L(r)$ is continuous. With the aim of showing existence of at least one constant equilibrium interest rate, we will employ the intermediate value theorem. To this end, we determine the asymptotic behaviour of $\WW_t^L(r)$ as $|r| \to \infty$. By the variation of constants formula, the expected wealth at time $t$ of a cohort born at $b \le t$ is given by
\[
\mbE[w_t^b(r)] = e^{r(t-b)}\mbE[w_b^b] + \int_b^t e^{r(t-s)}(\mbE[\eta_s^b] - \mbE[c_s^b(r)])\dd s.
\]

Now, since continuity is established, consider the limit as $r \rightarrow -\infty$. By Assumption \ref{ass:utility_reg_intro}, it follows from the representation of the optimal consumption that $c_s^b(r) \ge 0$, and thus we have the upper bound:
\[
\mbE[w_t^b(r)] \le e^{r(t-b)}\mbE[w_b^b] + \int_b^t e^{r(t-s)}\mbE[\eta_s^b]\dd s.
\]
For any $s \in [b, t)$, we have $t-s > 0$. Therefore, as $r \to -\infty$, $e^{r(t-s)} \to 0$. Due to integrability of the income process (finite absolute moment) the Dominated Convergence Theorem implies that the right-hand side vanishes. Consequently, $\limsup_{r\to -\infty} \mbE[w_t^b(r)] \le 0$, which immediately yields $\lim_{r\to -\infty} \WW^L_t(r)\leq 0$.

Next, we consider the limit as $r \rightarrow \infty$. Under Assumption \ref{ass:utility_reg_intro}, the optimal consumption is bounded above $\mbP$-a.s. by $\kappa$. Thus, $\mbE[c_s^b(r)] \le \kappa$. Factoring out $e^{r(t-b)}$, we establish the lower bound:
\[
\mbE[w_t^b(r)] \ge e^{r(t-b)}\left( \mbE[w_b^b] + \int_b^t e^{-r(s-b)}(\mbE[\eta_s^b] - \kappa)\dd s \right).
\]
For $s \in (b, t]$, the exponent $-(s-b) < 0$, meaning $e^{-r(s-b)} \to 0$ as $r \to \infty$. The integral term thus vanishes in the limit. Because we assumed initial wealth is strictly positive ($\mbE[w_b^b] > 0$) and $e^{r(t-b)} \to \infty$ for any $t > b$, it follows that $\mbE[w_t^b(r)] \to \infty$ for all $b < t$. Integrating over the demographic measure $\nu_t$, we conclude $\lim_{r\to \infty} \WW^L_t(r) = \infty$.

Since the range of $r\mapsto \WW_t^L(r)$ spans at least $[0,\infty)$ and $r \mapsto \WW^L_t(r)$ is continuous, the intermediate value theorem guarantees that for any given $K> 0$, there exists at least one constant interest rate $r\in \RR$ providing a general equilibrium according to Definition \ref{def:olg_clearing_rate}.
\end{proof}

\begin{rem}
    Instead of invoking Assumption \ref{ass:utility_reg_intro}, which in particular excludes CRRA utilities, one could invoke the results from Proposition \ref{prop:range_of_expected_w} use the same strategy in the proof as above to also include CRRA utilities (at least for $\gamma\leq 1)$. However, to this end one must assume the existence of a solution to the optimal consumption dynamics at individual level, which as discussed earlier is not proven or known. 
\end{rem}
}

\subsection{Existence and Uniqueness of General Equilibria in Non-Stationary Populations}

In Section~\ref{sec:olg_stationary_populations} we showed that general equilibria exist in a wide class of overlapping generation models for globally stationary populations. In this section we study the case of non-stationary populations. In contrast to the stationary case we apply  fixed point arguments, similar to those employed in Sections~\ref{sec:life_cycle} and \ref{sec:life_cycle_gen_eq}. As a result, we obtain existence and uniqueness of general equilibria after imposing conditions on the life-span depending on the flow of demographic measures and other parameters of the model.

We first have the following result which mimics the growth and stability estimates of the life-cycle model recast in the overlapping generation setting.

\begin{lem}\label{lem:olg_growth_stable}
	 Let $L>0$, $u_1,\, u_2$ satisfy Assumption~\ref{ass:utility_reg_intro} for some $\kappa >0$, $\bnu = \{\nu_t\}_{t\in \mbR}$ be a flow of \emph{regular demographic measures}, $\bw_0  = \{w^b_b\}_{b\in \mbR}\subset\msL^0(\Omega;\mbR)$, $\bmeta \coloneqq \{\eta^b\}_{b\in \mbR}$ where each $\eta^b \in \msL^0(\Omega;C([b,b+L;\mbR]))$ is a stochastic process and $r \in C(\mbR;\mbR)$ be a given interest rate process. 
     Then for any family of processes $\{(w^b(r),c^b(w^b(r),r))\}_{b\in \mbR}$ solving  \eqref{eq:olg_wealth}-\eqref{eq:olg_consumption}, one has the growth estimates
	 \begin{align}
	 	 \|\bw\|_{\msL^1_\omega \msL^\infty_{\mbR} \msL^\infty_{L}} &\leq  (1+L)e^{\|r\|_{\msL^\infty_{\mbR}}} \left(\|\bw_0\|_{\msL^1_\omega \msL^\infty_\mbR } +    \|\bmeta\|_{ \msL^1_\omega \msL^\infty_{\mbR} \msL^\infty_{L}} + \kappa \right), \label{eq:olg_wealth_apriori_bound} \\
          \|\bc(\bw,r)\|_{\msL^1_\omega  \msL^\infty_{\mbR}\msL^\infty_{L}}  &\leq \kappa \label{eq:olg_consumption_bdd}
	 \end{align}
	 and setting 
	 \begin{align*}
	 	\mfA\coloneqq & \mfA(r,L,\lambda,\kappa) \coloneqq \exp\left(L\big(\|r\|_{\msL^\infty_\mbR} + \lambda \kappa^2 e^{(1+L)\|r\|_{\msL^\infty_{\mbR}} } \big)\right),\\
       \mfB\coloneqq &\mfB(\lambda,\kappa,L,\bw_0,\bmeta) \coloneqq  \|\bw_0\|_{\msL^1_\omega \msL^\infty_\mbR } +    \|\bmeta\|_{\msL^1_\omega \msL^\infty_{\mbR} \msL^\infty_{L}}+ (1+\lambda)\kappa 
	 \end{align*}
	 given a pair of interest rates $r,\,\tilde{r} \in C(\mbR;\mbR)$ and a pair of initial wealth vectors $\bw_0 \coloneqq \{w^b_b\}_{b\in \mbR},\, \tilde{\bw}_0 = \{\tilde{w}^b_b\}_{b\in \mbR}$ with associated wealth families $\bw(r)$ and $\bw(\tilde{r})$, it holds that
\begin{equation}\label{eq:olg_wealth_rate_stability}
	 \begin{aligned}
	 		\|\bw(r)-\bw(\tilde{r})\|_{\msL^1_\omega \msL^\infty_\mbR  \msL^\infty_L} \lesssim \,\mfA &\bigg( \|\bw_0-\tilde{\bw}_0\|_{\msL^1_\omega \msL^\infty_\mbR } + L (1+L)\mfB   \,\|r-\tilde{r}\|_{\msL^\infty_\mbR}\bigg).
	 	\end{aligned}
	 \end{equation}
	 and
	 \begin{equation}\label{eq:olg_consump_rate_stability}	 
\begin{aligned}
	 	\|\bc(\bw(r),r)-\bc(\bw(\tilde{r}),\tilde{r})\|_{\msL^1_\omega \msL^\infty_{\mbR} \msL^\infty_{L}} \lesssim \,\lambda \kappa (1+\kappa) \mfA e^{L\|\tilde{r}\|_{\msL^\infty_\mbR}}  \bigg( &\|\bw_0-\tilde{\bw}_0\|_{\msL^1_\omega\msL^\infty_\mbR }\\
	 	&+ L (1+L)\mfB  \,\|r-\tilde{r}\|_{\msL^\infty_\mbR}\bigg).
	 	\end{aligned}
	 \end{equation}
\end{lem}
\begin{proof}
%
We prove \eqref{eq:olg_wealth_apriori_bound} in detail, and give only a sketch proof of \eqref{eq:olg_wealth_rate_stability} and \eqref{eq:olg_consump_rate_stability} as they follow mutatis mutandis.

Firstly, taking \eqref{eq:wealth_apriori_bound} from Lemma~\ref{lem:LC_rate_stable} in hand, up to suitably shifting the domain $[0,L]$ to $[b-L,b]$, we see that for all  $b\in \mbR$,
\begin{equation*}
	\begin{aligned}
		\|w^b(r)\|_{\msL^\infty_{[b-L,b]}} &\leq  (1+L)\exp\Big(\|r\|_{\msL^\infty_{[b-L,b]}}\Big) \left(|w^b_b| +  \|\eta^b\|_{\msL^\infty_{[b-L,b]}}+ \kappa \right),
	\end{aligned}
	\quad \mbP\text{-a.s.}
\end{equation*}
Taking suprema over $b\in \mbR$ followed by the expectation and using that $r$ is deterministic we obtain
\begin{equation*}
	\begin{aligned}
		\mbE\Big[\sup_{b\in \mbR}\|w^b(r)\|_{\msL^\infty_{[b-L,b]}} \Big]&\leq  (1+L)\exp\big(\|r\|_{\msL^\infty_\mbR}\big) \left(\mbE\Big[\sup_{b\in \mbR}|w^b_b| \Big] +  \mbE\Big[\sup_{b\in \mbR}\|\eta^b\|_{\msL^\infty_{[b-L,b]}}\Big]+ \kappa \right).
	\end{aligned}
\end{equation*}
This gives the estimate \eqref{eq:olg_wealth_apriori_bound} written in the more compact notation of \eqref{eq:w_0_norm} - \eqref{eq:income_norm}.

The proofs of \eqref{eq:olg_wealth_rate_stability} and \eqref{eq:olg_consump_rate_stability} follow similarly.
\end{proof}
Under the same Lipschitz continuity assumptions as employed in Section~\ref{sec:life_cycle} we have the following general, existence and uniqueness result for a general equilibrium interest rate. Recall, that throughout we equip the space of continuous functions $C(\mbR;\mbR)$ with the supremum norm $\|f\|_{\msL^\infty_{\mbR}}\coloneqq \sup_{t\in \mbR}|f_t|$, under which it is a Banach space and implicitly if we write $f\in C(\mbR;\mbR)$ we mean that it is both continuous and has finite supremum norm.

\begin{thm}\label{thm:olg_well_posed}
 Let $L>0$, $u_1, u_2$ satisfy Assumption~\ref{ass:utility_reg_intro} for some $\kappa >0$, $\bnu = \{\nu_t\}_{t\in \mbR}$ be a flow of \emph{regular demographic measures} and $\bw_0  = \{w^b_b\}_{b\in \mbR}\subset\msL^1(\Omega;\mbR)$ be a family of integrable real-valued random variables and $\mu,\,\sigma$ be sufficiently regular coefficients such that for each $b\in \mbR$ there exists a unique strong solution $\eta^b$ to \eqref{eq:olg_income}. Then, 
 \begin{enumerate}[label=\roman*)]
     \item  For any $r\in C(\mbR;\mbR)$  there exists an $\bar{L}_0\coloneqq \bar{L}(\bnu,\bw_0, \bmeta, \kappa,\lambda,\|r\|_{\msL^\infty_\mbR}) \in (0,L]$ such that  there exists a unique family of processes $(\bw(r),\bc(r)) = (\{w^b(r)\}_{b\in \mbR},\{c^b\}_{b\in \mbR})$ with each $(w^b(r),c^b(r))\in \msL^1(\Omega;C_{[b,b+\bar{L}_0]})^{\otimes 2}$ describing a strong solution to \eqref{eq:olg_wealth}-\eqref{eq:olg_consumption}. In addition this pair are a solution to the associated optimal consumption problem \eqref{eq:olg_wealth_intro}-\eqref{eq:olg_payoff_discounted}. \label{it:olg_well_posed_specific}
     \item For any capital profile $K\in C^1(\mbR;(0,+\infty)$, $R>0$ there exists an $\bar{L}_1 \coloneqq \bar{L}_1 (\bnu,\bw_0, \bmeta, \kappa,\lambda,R)\in (0,\bar{L}_0]$ and a unique interest rate $\bar{r}$ such that
     \begin{equation}\label{eq:olg_gen_eq_ball}
         \sup_{t\in \mbR} |\bar{r}_t| \leq R + 2\|\bnu\|_{W^{1,\infty}_{\mbR}\TV_{\bar{L}_1}} \|\bw_0\|_{\msL^\infty_\mbR \msL^1_\omega}
     \end{equation}
     and $\bar{r}$ is a general equilibrium interest rate for the associated overlapping generations model with lifespan $\bar{L}_1$, in the sense of Definition~\ref{def:olg_clearing_rate}. \label{it:olg_well_posed_general}
 \end{enumerate}
 %
 %
\end{thm}
\begin{proof}
We first show \ref{it:olg_well_posed_specific}. Observe that for each $b\in \mbR$, the equation satisfied by $w^b$ is exactly \eqref{eq:triple_equations}, with $c,\eta$ replaced by $c^b,\, \eta^b$ respectively. Hence, it follows from exactly the arguments as the proof of Corollary~\ref{cor:life_cycle_well_posed} that for any interest rate path $r\in C(\mbR;\mbR)$, one has existence and uniqueness of $w^b$ for any $b\in \mbR$ given that we choose some $\bar{L}_0\coloneqq \bar{L}_0(\kappa,\lambda,w^b_b,\beta^b,\|r\|_{\msL^\infty_\mbR}) \in (0,L]$ sufficiently small. Since $\bar{L}_0$ does not depend on $b\in \mbR$, by an abuse of notation we obtain existence and uniqueness of the family $\{w^b(r)\})_{b\in \mbR}$ for $\bar{L}_0 \coloneqq \bar{L}_0(\kappa,\lambda,\bnu, \bw_0,\bmeta,\|r\|_{\msL^\infty_\mbR}) \in (0,L]$. From now on any lifespan used in the proof is implicitly taken with a minimum against $\bar{L}_0$. Recall that in our model all generations have the same life-span $L \wedge \bar{L}_0$ and the income processes are the same (up to shifting in time) for each generation. 

    To show existence and uniqueness of a general equilibrium interest rate we follow a similar strategy as laid out in Section~\ref{sec:life_cycle_gen_eq}. To this end we first find a functional expression which exactly characterises general equilibrium interest rates, see \eqref{eq:olg_rate_equation}. Then we use a fixed point argument appealing to the Marshallian stability  obtained by Lemma~\ref{lem:LC_rate_stable} to show existence and uniqueness of a general equilibrium rate in the set \eqref{eq:olg_gen_eq_ball} for $\bar{L}_1 \in (0,\bar{L}_0]$ sufficiently small. We give a detailed proof in the case that $K$ is constant, i.e. $t\mapsto K_t\equiv K\in (0,+\infty)$.  The proof in the time dependent case being similar. 

We first obtain a functional identity which must be satisfied by any general equilibrium interest rate. Applying Lemma~\ref{lem:diff_two_param_function_integral} to the function $(t,b)\mapsto f(t,b) = \mbE[w^b_t]$, using the equation \eqref{eq:olg_wealth} to evaluate $w^{t-L}_t$ and that $w^t_t \equiv w^b_b$, one finds
   \begin{align*}
       \frac{\dd}{\dd t} \mbW_t(r) =&\,  \mbE[w^{b}_b]\left(n(t,t)-n(t,t-L)\right)- \int_{t-L}^t \mbE\left[r_u w^{t-L}_u + \eta^{t-L}_u - c^{t-L}_u \right] \dd u \,n(t,t-L) \\
       &\, + \int_{t-L}^t \left(\mbE\big[w^b_b\big]+\int_b^t \mbE\big[r_u w^b_u + \eta^b_u - c^b_u \big] \dd u \,\right)  \partial_t n(t,b) \dd b\\
       &\, + \int_{t-L}^t \mbE\big[r_t w^b_t + \eta^b_t - c^b_t \big] n(t,b)\dd b,
   \end{align*}
   where for simplicity we write $w^b_s \coloneqq  w^b_s(r)$ and $c^b_s \coloneqq c(w^b(r),r)$ for all $b\in \mbR$ and $s\neq b$. We recognize the first integrand of the fourth term as $\int_{t-L}^t \mbE[r_t w^b_t] n(t,b) \dd b = r_t \mbW_t$. Hence, we can rearrange the above after setting the right hand side to zero, and arguing as in the proof of Proposition~\ref{prop:life_cycle_rate_props}, any general equilibrium rate must be such that, for all $b \in \mbR$,
   \begin{equation}\label{eq:olg_rate_equation}
       \begin{aligned}
       r_t = \frac{1}{K}\,&\int_{t-L}^t \mbE\left[r_u w^{t-L}_u  + \eta^{t-L}_u - c^{t-L}_u(w^{t-L}(r),r) \right] \dd u \,n(t,t-L)  \\
      &\, - \frac{1}{K}\mbE\big[w^{b}_{b}\big]\left(n(t,t)-n(t,t-L) \right)-\int_{t-L}^t \mbE\big[\eta^b_t - c^b_t(w^{b}(r),r) \big]  n(t,b) \dd b\\
      &\,-  \frac{1}{K} \int_{t-L}^t \left(\mbE\big[w^b_b\big] +\int_b^t \mbE\big[r_u w^b_u + \eta^b_u - c^b_u(w^b(r),r) \big] \dd u \,\right)   \partial_t n(t,b) \dd b,
            \end{aligned}
   \end{equation}
   where we have reinstated the explicit dependence of $c^b$ and $w^b$ on $r$ so that it is clear that \eqref{eq:olg_rate_equation} is a fixed point problem for the market clearing rate. To show that such a fixed point exists, we define the map $\Phi_L :C(\mbR;\mbR) \to C(\mbR;\mbR)$ such that 
   \begin{equation}\label{eq:olg_phi_def}
       \begin{aligned}
           K \Phi_L(r)_t \coloneqq  &\mapsto \,\int_{t-L}^{t} \mbE\left[r_u w^{t-L}_u  + \eta^{t-L}_u - c^{t-L}_u  \right] \dd u \,n(t,t-L)  \\
      &\,-\mbE\big[w^b_b\big]\left(n(t,t) -n(t,t-L)\right) -\int_{t-L}^{t} \mbE\big[\eta^b_{t} - c^b_{t}  \big]  n(t,b) \dd b\\
      &\,-  \int_{t-L}^{t}\left(\mbE\big[w^b_b\big] +\int_b^{t} \mbE\big[r_u w^b_u + \eta^b_u - c^b_u \big] \dd u \,\right)   \partial_t n(t,b) \dd b,
       \end{aligned}
   \end{equation}
and for any $R>0$, the closed and bounded set, %
\begin{equation}\label{eq:olg_rate_closed_set}
    \fB^{\bnu,\bw_0}_R(L) \coloneqq \left\{ r:\mbR\to \mbR \,:\, \sup_{t\in \mbR}|r_t| \leq R+2\|\bnu\|_{W^{1,\infty}_{\mbR}\TV_L} \|\bw_0\|_{\msL^\infty_\mbR \msL^1_\omega} \right\},\quad \text{for }\, L>0.
\end{equation}
Since $ \fB^{\bnu,\bw_0}_R(L)$ is a closed and bounded subset of the complete metric space $(C(\mbR;\mbR);\|\,\cdot\,\|_{\msL^\infty_{\mbR}})$ it is itself a complete metric space.
   
Since a market clearing rate must satisfy $r = \Phi_L(r)$ and vice versa we seek to apply the Banach fixed point theorem. Firstly, taking absolute values on both sides of \eqref{eq:olg_phi_def} we find
  \begin{equation*}
	\begin{aligned}
	|K||\Phi_L(r)_t| \leq &\, L|n(t,t-L)|\sup_{u\in [t-L,t]}\left( |r_u| \mbE\left[|w^{t-L}_u |\right] +  \mbE\left[|\eta^{t-L}_u|\right]  + \mbE\big[|c^{t-L}_u |\big]\right) \\
		&+ \sup_{b\in \mbR}\mbE\big[|w^b_b|\big] \Big(|n(t,t)|+n(t,t-L)|\Big) +  L \sup_{b\in [t-L,t]} \Big(|n(t,b)|\big(\mbE\big[|\eta^b_t|\big] +  \mbE\big[|c^b_t |\big]\big)\Big) \\
		&+ \sup_{b\in [t-L,t]}\left(|\partial_t n (t,b)| \Big(\mbE\big[|w^b_b|\big] + L \sup_{u\in [b,t]}\big( |r_u| \mbE\big[|w^b_u|\big] + \mbE\big[|\eta^b_u|\big] + \mbE\big[|c^b_u|\big]\big)\Big)\right).
	\end{aligned}
	\end{equation*}
Then, taking suprema on both sides over $t\in \mbR$, estimating the suprema of expectations by the expectation of the suprema and consolidating terms and defining the quantity
\begin{equation*}
    \mfA(r,\bw,\bmeta,\bc) \coloneqq \|r\|_{\msL^\infty_\mbR}   \|\bw\|_{\msL^1_\omega \msL^\infty_{\mbR}\msL^\infty_L } +  \|\bmeta\|_{ \msL^1_\omega \msL^\infty_{\mbR}\msL^\infty_L}  + \|\bc\|_{\msL^1_\omega \msL^\infty_{\mbR}\msL^\infty_L }
\end{equation*}
we obtain the bound
 \begin{equation*}
	\begin{aligned}
		|K|\sup_{t\in \mbR}|\Phi_L(r)_t| \leq &\, L\|\bnu\|_{W^{1,\infty}_{\mbR}\TV_L} \mfA(r,\bw,\bmeta,\bc) \\
		&+ 2\|\bw_0\|_{\msL^1_\omega \msL^\infty_\mbR}\|\bnu\|_{W^{1,\infty}_{\mbR}\TV_L} +  L \|\bnu\|_{W^{1,\infty}_{\mbR} \TV_L}\Big(\|\bmeta\|_{ \msL^1_\omega \msL^\infty_{\mbR}\msL^\infty_L}  + \|\bc\|_{\msL^1_\omega \msL^\infty_{\mbR}\msL^\infty_L }\Big) \\
		&+ L \|\bnu\|_{W^{1,\infty}_{\mbR}\TV_L} \Big(\|\bw_0\|_{\msL^1_\omega \msL^\infty_\mbR}+  L \, \mfA(r,\bw,\bmeta,\bc)\Big)\\
		\lesssim \, &+ 2\|\bw_0\|_{\msL^1_\omega \msL^\infty_\mbR}\|\bnu\|_{W^{1,\infty}_{\mbR}\TV_L} \\
		& + L \|\bnu\|_{W^{1,\infty}_{\mbR}\TV_L} \Big(\|\bw_0\|_{\msL^1_\omega \msL^\infty_\mbR}+  (1+L)\,  \mfA(r,\bw,\bmeta,\bc)\Big)
	\end{aligned}
\end{equation*}

So that, applying the estimates \eqref{eq:olg_wealth_apriori_bound} and \eqref{eq:olg_consumption_bdd} from Lemma~\ref{lem:olg_growth_stable}, consolidating terms gives us the bound
\begin{equation*}
	\begin{aligned}
		\sup_{t\in \mbR}|\Phi_L(r)_t| \lesssim \, & \frac{2}{|K|}\|\bw_0\|_{\msL^1_\omega \msL^\infty_\mbR}\|\bnu\|_{W^{1,\infty}_{\mbR}\TV_L} \\
		& + \frac{L}{|K|} \|\bnu\|_{W^{1,\infty}_{\mbR}\TV_L}  (1+L)^2\Big( \big(1+\|r\|_{\msL^\infty_\mbR}\big) e^{\|r\|_{\msL^\infty_{\mbR}}} \big(\|\bw_0\|_{\msL^1_\omega \msL^\infty_\mbR } +    \|\bmeta\|_{ \msL^1_\omega \msL^\infty_{\mbR} \msL^\infty_{L}} + \kappa \Big).
	\end{aligned}
\end{equation*}
Finally, using the assumption that $r\in \mfB^{\bnu,\bw_0}_R(L)$, we obtain the estimate
\begin{equation*}
	\begin{aligned}
		\sup_{r\in \mfB^{\bnu,\bw_0}_R(L) }\|\Phi_L(r)\|_{\msL^\infty_{\mbR}} \lesssim \, & \frac{2}{K}\|\bw_0\|_{\msL^1_\omega \msL^\infty_\mbR}\|\bnu\|_{W^{1,\infty}_{\mbR}\TV_L} \\
		& + \frac{L}{K} \|\bnu\|_{W^{1,\infty}_{\mbR}\TV_L}  (1+L)^2\Big( \big(1+R+2\|\bnu\|_{W^{1,\infty}_{\mbR}\TV_L} \|\bw_0\|_{\msL^\infty_\mbR \msL^1_\omega} \big) \\
        &\qquad \times e^{R+2\|\bnu\|_{W^{1,\infty}_{\mbR}\TV_L} \|\bw_0\|_{\msL^\infty_\mbR \msL^1_\omega} } \big(\|\bw_0\|_{\msL^1_\omega \msL^\infty_\mbR } +    \|\bmeta\|_{ \msL^1_\omega \msL^\infty_{\mbR} \msL^\infty_{L}} + \kappa \Big).
	\end{aligned}
\end{equation*}
 As a result, there exists some $L_1\coloneqq L_1(\kappa,\lambda,\bnu, \bw_0,\bmeta,\lambda,R) \in (0,\bar{L}_0)$ such that for all $L \in (0,L_1]$, $\Phi_L$  maps the closed and bounded set $ \fB^{\bnu,\bw_0}_R(L_1)$ to itself.

To show that $\Phi_L$ is a contraction on $  \fB^{\bnu,\bw_0}_R(L)$ for some $L$ possibly strictly smaller than $L_1$ we obtain a natural stability estimate. First, since we keep the initial wealth of each generation identical, independently of the interest rate, we see directly that for $r,\,\tilde{r} \in C(\mbR;\mbR)$,
 \begin{align*}
      K\left(\Phi_L(r)_t-\Phi_L(\tilde{r})_t\right) =& \int_{t-L}^{t} (r_u-\tilde{r}_u) \mbE\big[w^{t-L}_u\big] + \tilde{r}_u\mbE\big[w^{t-L}_u(r) - w^{t-L}_u(\tilde{r})\big]\dd u \,n(t,t-L) \\
       &- \int_{t-L}^{t}\mbE\big[c^{t-L}_u(w(r),r)-c^{t-L}_u(w(\tilde{r}),\tilde{r})\big] \dd u \,n(t,t-L)\\
       &+ \int_{t-L}^{t} \mbE\big[ c^b_{t}(w(r),r)- c^b_t(w(\tilde{r}),\tilde{r})\big] \ n(t,b)\dd b\\
       &- \int_{t-L}^{t} \int_b^{t} (r_u -\tilde{r}_u)\mbE\big[w^b_u(r)\big] +\tilde{r}_u \mbE\big[w^b_u(r)-w^b_u(\tilde{r})\big] \dd u \, \partial_t n(t,b)\dd b\\
       &+\int_{t-L}^t \int_b^t \mbE\big[c^b_u(w(r),r)-c^b_u(w(\tilde{r}),\tilde{r})\big] \dd u \, \partial_t n(t,b)\dd b.
   \end{align*}
Taking absolute values, suprema over $t\in \mbR$ on both sides and consolidating norms on the right hand side we find the bound
 \begin{align*}
       \|\Phi_L(r)-\Phi_L(\tilde{r})\|_{\msL^\infty_{\mbR}} \leq &\,  \frac{L}{K}(1+L)\|\bnu\|_{W^{1,\infty}_{\mbR}\TV_L}\bigg(\|r-\tilde{r}\|_{\msL^\infty_\mbR} \|\bw\|_{\msL^1_\omega \msL^\infty_{\mbR}\msL^\infty_L} \\
       &\qquad \qquad \qquad \qquad \qquad  + \|\tilde{r}\|_{\msL^\infty_{\mbR}}\|\bw(r) - \bw(\tilde{r})\|_{\msL^1_\omega \msL^\infty_{\mbR}\msL^\infty_L} \bigg)\\
       &+\frac{L}{K}(2+L)\|\bnu\|_{W^{1,\infty}_{\mbR}\TV_L}  \|\bc(\bw(r),r)-\bc(\bw(\tilde{r}),\tilde{r})\|_{\msL^1_\omega \msL^\infty_{\mbR}\msL^\infty_L}.
   \end{align*}
Therefore, invoking the growth estimate \eqref{eq:olg_wealth_apriori_bound} and the stability bounds \eqref{eq:olg_wealth_rate_stability}-\eqref{eq:olg_consump_rate_stability} with $\bw_0 = \tilde{\bw_0}$, we find, for 
\begin{align*}
    \mfA(r,L)\coloneqq &\, \mfA(r,L,\lambda,\kappa) \coloneqq \exp\left(L\big(\|r\|_{\msL^\infty_\mbR} + \lambda \kappa^2 e^{(1+L)\|r\|_{\msL^\infty_{\mbR}} } \big)\right),\\
    \mfB \coloneqq  &\, \|\bw_0\|_{\msL^1_\omega \msL^\infty_\mbR } +    \|\bmeta\|_{ \msL^1_\omega \msL^\infty_{\mbR} \msL^\infty_{L}} + 1 
\end{align*}
that one has
   \begin{align*}
       \|\Phi_L(r)-\Phi_L(\tilde{r})\|_{\msL^\infty_{\mbR}} \lesssim_{\lambda,\kappa} \,  \frac{L}{K}(1+L)^2e^{\|r\|_{\msL^\infty_{\mbR}}+L\|\tilde{r}\|_{\msL^\infty_\mbR}}\bigg(&\, \|\bnu\|_{W^{1,\infty}_{\mbR}\TV_L} \\
       &\, + \left(\|\tilde{r}\|_{\msL^\infty_{\mbR}} + \|\bnu\|_{W^{1,\infty}_{\mbR}\TV_L}  \right)  \mfA(r,L)\bigg)\mfB \|r-\tilde{r}\|_{\msL^\infty_\mbR} 
   \end{align*}
So that taking $r\neq \tilde{r} \in \fB^{\bnu,\bw_0}_R(L)$ for any $L\in (0,L_1]$, we see that there exists a constant $C \coloneqq C(\kappa,\lambda,R,\|\bnu\|_{W^{1,\infty}_{\mbR}\TV_L},\|\bw_0\|_{\msL^\infty_\mbR \msL^1_\omega}, L) >0$ which is monotone non-decreasing in all its arguments (and double-exponentially growing in its final four arguments), such that %
\begin{align*}
       \|\Phi_L(r)-\Phi_L(\tilde{r})\|_{\msL^\infty_{\mbR}} \leq&\,  \frac{L}{K} C \left( \|\bw_0\|_{\msL^1_\omega \msL^\infty_\mbR }+ \|\bmeta\|_{ \msL^1_\omega \msL^\infty_{\mbR} \msL^\infty_{L}} +1\right) \|r-\tilde{r}\|_{\msL^\infty_\mbR}.
   \end{align*}
Therefore, there exists an $L_0(\kappa,\lambda,R,\|\bnu\|_{W^{1,\infty}_{\mbR}\TV_L},\|\bw_0\|_{\msL^\infty_\mbR \msL^1_\omega}, L,K) \in (0,L_1]$ such that for all $L\in (0,L_0]$ 
%
%
%
%
the map $\Phi_{L}$ is a contraction from the closed and bounded set $ \fB^{\bnu,\bw_0}_R(L) \subset C(\mbR;\mbR)$ to itself. Hence, applying the Banach fixed point theorem, there exists a unique fixed point $\bar{r}\in  \fB^{\bnu,\bw_0}_R$, concluding the proof. 
\end{proof}
%
 %
%
The following proposition shows that the interest rate is stable with respect to the relevant inputs of wealth at birth $(w_0(b))_{b\in \mbR}$, income $(\eta^b)_{b\in \mbR}$ and demographic distribution $(\rho_t)_{t\in \mbR}$.

\begin{prop}\label{prop:olg_rate_stable}
 Let $L>0$, $u_1,$ and $u_2$ satisfy Assumption~\ref{ass:utility_reg_intro} for some $\kappa>0$, $\bmeta \coloneqq \{\eta^b\}_{b\in \mbR},\, \tilde{\bmeta} \coloneqq \{\tilde{\eta}^\}_{b\in \mbR}$ be two collections of measurable stochastic processes with each $\eta^b,\, \tilde{\eta}^b\in \msL^0(\Omega;C_{[b,L+b]})$, $\bnu = \{\nu_t\}_{t\in \mbR},\, \tilde{\bnu} = \{\tilde{\nu}_t\}_{t\in \mbR}$ be two flows of \emph{regular demographic measures}, $\bw_0  = \{w^b_b\}_{b\in \mbR},\, \tilde{\bw}_0  = \{\tilde{w}^b_b\}_{b\in \mbR}$ be two families of integrable real valued random variables and $R>0$. Then, there exists a lifespan 
 \begin{equation*}
     \bar{L}_0 \in \left(0,\bar{L}_1(\bmeta,\bnu,\bw_0,\kappa,\lambda,R)\wedge \bar{L}_1(\tilde{\bmeta}, \tilde{\bnu}, \tilde{\bw}_0,\kappa,\lambda,R) \right)
 \end{equation*}
 such that for any $L\in (0,\bar{L}_0]$ there exists unique general equilibrium rates $\bar{r}\coloneqq \bar{r}(\bmeta,\bnu,\bw_0)$ and $\tilde{r}\coloneqq \tilde{r}(\tilde{\bmeta},\tilde{\bnu},\tilde{\bw}_0)$ in the set
 \begin{equation*}
     \mfB_R(L)\coloneq \left\{ r: \mbR\to \mbR\,:\, \sup_{t\in \mbR}|r_t| \leq R+ \left(\|\bnu\|_{W^{1,\infty}_{\mbR} \text{TV}_L} \|\bw_0\|_{\msL^\infty_{\mbR}\msL^1_{\omega}} \vee \|\tilde{\bnu}\|_{W^{1,\infty}_{\mbR} \text{TV}_L} \|\tilde{\bw}_0\|_{\msL^\infty_{\mbR}\msL^1_{\omega}} \right)\right\}.
 \end{equation*}
 In addition, there exists a constant $C\coloneqq C(\kappa,\lambda,\bw_0,\tilde{\bw}_0,\bmeta,\tilde{\bmeta},\bnu,\tilde{\bnu},R)>0$ such that
\begin{equation}\label{eq:olg_rate_stability}
\begin{aligned}
    \|r-\tilde{r}\|_{\msL^\infty_\mbR} \leq C \bigg(& \|\bmeta-\tilde{\bmeta}\|_{\msL^\infty_{\mbR} \msL^1_\omega \msL^\infty_L} +\|\bnu-\tilde{\bnu}\|_{W^{1,\infty}_{\mbR}\TV_L} + \|\bw_0-\tilde{\bw}_0\|_{\msL^\infty_{\mbR}\msL^1_\omega} \bigg).
    \end{aligned}
\end{equation}
\end{prop}
\begin{proof}
    The proof leverages the stability of consumption and wealth as a function of income and initial wealth (Lemma~\ref{lem:LC_rate_stable}) and the explicit expression for the market clearing rate \eqref{eq:olg_rate_equation}. To avoid repetitive detail, we focus on the stability in initial wealth, the proofs for stability in income and demographic measure being similar. Taking the difference between $r\coloneqq r(\bw_0)$ and $\tilde{r}\coloneqq r(\tilde{\bw}_0)$, with both defined by \eqref{eq:olg_rate_equation} with existence and uniqueness of each in $\mfB_R$ guaranteed by Theorem~\ref{thm:olg_well_posed}, we find, for $t,\, b\in \mbR$ 
    \begin{align*}
        r_t-\tilde{r}_t = & \mbE\big[w^{b}_{b}-\tilde{w}^{b}_{b}\big] \left(n(t,t)-n(t,t-L)\right)\\
        &-\left(\int_{t-L}^{t} (r_u-\tilde{r}_u) \mbE\left[w^{t-L}_u(r)\right] + \tilde{r}_u\mbE\left[w^{t-L}_u(r) - w^{t-L}_u(\tilde{r})\right]\dd u \,\right)n(t,t-L) \\
       &- \int_{t-L}^{t}\mbE\big[c^{t-L}_u(r,w(r))-c^{t-L}_u(\tilde{r},w(\tilde{r}))\big] \dd u \,n(t,t-L)\\
       &+ \int_{t-L}^{t} \mbE\big[ c^b_{t}(r,w(r))- c^b_t(\tilde{r},w(\tilde{r}))\big] \ n(t,b)\dd b\\
&- \int_{t-L}^{t} \mbE\big[w^b_b(r) -w^b_b(\tilde{r})\big] \, \partial_t n(t,b)\dd b\\
       &- \int_{t-L}^{t} \left(\int_b^{t} (r_u -\tilde{r_u})\mbE[w^b_u(r)] +\tilde{r}_u \mbE[w^b_u(r)-w^b_u(\tilde{r})] \dd u \,\right) \partial_t n(t,b)\dd b\\
       &+\int_{t-L}^t \int_b^t \mbE\big[c^b_u(r,w)-c^b_u(\tilde{r},w(\tilde{r}))\big] \dd u \, \partial_t n(t,b)\dd b.
    \end{align*}
    Therefore, appealing to  Lemma~\ref{lem:LC_rate_stable} in particular \eqref{eq:wealth_apriori_bound}, \eqref{eq:wealth_rate_stability} and \eqref{eq:consump_rate_stability}, up to suitably shifting the interval of definition from $[0,L]$ to $[b,b+L]$,  we obtain the estimate 
    \begin{align*}
    	\sup_{t \in \mbR} |r_t-\tilde{r}_t| \,&\leq  \,  \sup_{b\in \mbR}\mbE\Big[\, \big|w^{b} - \tilde{w}^{b}\big|\,\Big] + L\|\bnu\|_{W^{1,\infty}_{\mbR}\TV_L} \sup_{t\in \mbR}\sup_{u\in [t-L,t]} |r_u-\tilde{r}_u| \sup_{t \in \mbR}\mbE\Big[\left\|w^{t-L}(r)\right\|_{\msL^\infty_{[t-L,t]}}\Big]\\
    	&\qquad + L\|\bnu\|_{W^{1,\infty}_{\mbR}\TV_L} \sup_{t\in \mbR}\sup_{u\in [t-L,t]} |\tilde{r}_u|  \sup_{t\in \mbR} \mbE\Big[\left\|w^{t-L}(r) - w^{t-L}(\tilde{r})\right\|_{\msL^\infty_{[t-L,t]}}\Big]\\
    	&\qquad + L \|\bnu\|_{W^{1,\infty}_{\mbR}\TV_L} \sup_{t\in \mbR} \mbE\Big[\left\|c^{t-L}(r,w(r))-c^{t-L}(\tilde{r},w(\tilde{r}))\right\|_{\msL^\infty_{[t-L,t]}}\Big]\\
    	&\qquad + L \|\bnu\|_{W^{1,\infty}_{\mbR}\TV_L}\sup_{t\in \mbR}\sup_{b\in [t-L,t]} \mbE\Big[ \big|c^b_{t}(r,w(r))- c^b_t(\tilde{r},w(\tilde{r}))\big|\Big]\\
    	&\qquad + L^2 \|\bnu\|_{W^{1,\infty}_{\mbR}\TV_L} \sup_{t\in \mbR} \sup_{b\in [t-L,t]}\sup_{u\in [b,t]}\mbE\big[|w^b_u(r)|\big]  \sup_{t\in \mbR} \sup_{b\in [t-L,t]} \|r-\tilde{r}\|_{\msL^\infty_{[b,t]}}  \\
    	&\qquad +L^2\|\bnu\|_{W^{1,\infty}_{\mbR}\TV_L}\sup_{t\in \mbR} \sup_{b \in [t-L,t]}\|\tilde{r}\|_{C_{[b,t]}} \sup_{t\in \mbR} \sup_{b \in [t-L,t]}\mbE\Big[\big\|w^b(r)-w^b(\tilde{r})\big\|_{\msL^\infty_{[t-L,t]}}\Big] \\
    	&\qquad +L^2 \|\bnu\|_{W^{1,\infty}_{\mbR}\TV_L}\sup_{t\in \mbR}\sup_{b\in [t-L,t]} \mbE\Big[\big\|c^b(r,w)-c^b(\tilde{r},w(\tilde{r}))\big\|_{\msL^\infty_{[b,t]}}\Big] \\
    	&\lesssim_{\lambda,\kappa}  \, \left\|\bw_{0} - \tilde{\bw}_{0}\right\|_{\msL^\infty_{\mbR}\msL^1_\omega} \\
    	&\quad + L \|\bnu\|_{W^{1,\infty}_{\mbR}\TV_L}  (1+L)e^{\|r\|_{\msL^\infty_\mbR}} \left(\|\bw_{0}\|_{\msL^\infty_{\mbR}\msL^1_\omega}+ \|\bmeta\|_{\msL^\infty_{\mbR} \msL^1_\omega L^\infty_{L}} +\kappa\right)\|r-\tilde{r}\|_{\msL^\infty_\mbR} \\
    	&\quad   +L\|\bnu\|_{W^{1,\infty}_{\mbR}\TV_L} \|\tilde{r}\|_{\msL^\infty_\mbR}   e^{\|r\|_{\msL^\infty_\mbR}}\left\|\bw_{0}-\tilde{\bw}_{0}\right\|_{\msL^\infty_\mbR \msL^1_\omega} \\
    	&\quad  +L^2\|\bnu\|_{W^{1,\infty}_{\mbR}\TV_L} \|\tilde{r}\|_{\msL^\infty_\mbR}  e^{ \|r\|_{\msL^\infty_\mbR}} \left(1+\|\bw_{0}\|_{\msL^\infty_{\mbR}\msL^1_\omega}+ \|\bmeta\|_{\msL^\infty_{\mbR} \msL^1_\omega L^\infty_{L}}\right) \|r-\tilde{r}\|_{\msL^\infty_\mbR}\\
    	&\quad +2L \|\bnu\|_{W^{1,\infty}_{\mbR}\TV_L} e^{(1+L)\left(\|r\|_{\msL^\infty_\mbR}+\|\tilde{r}\|_{\msL^\infty_\mbR}\right) }\left\|\bw_{0}-\tilde{\bw}_{0}\right\|_{\msL^\infty_\mbR \msL^1_\omega}\\
    	&\quad +  2L \|\bnu\|_{W^{1,\infty}_{\mbR}\TV_L}e^{(1+L)\left(\|r\|_{\msL^\infty_\mbR}+\|\tilde{r}\|_{\msL^\infty_\mbR}\right) }\left(1+ L\left(1+\|\bw_{0}\|_{\msL^\infty_{\mbR}\msL^1_\omega}+ \|\bmeta\|_{\msL^\infty_{\mbR} \msL^1_\omega L^\infty_{L}}\right) \right)\|r-\tilde{r}\|_{\msL^\infty_\mbR} \\
    	& \quad +  L^2 \|\bnu\|_{W^{1,\infty}_{\mbR}\TV_L}(1+L)e^{\|r\|_{\msL^\infty_\mbR}} \left(\|\bw_0\|_{\msL^\infty_\mbR \msL^1_\omega} +  \|\bmeta\|_{\msL^\infty_{\mbR} \msL^1_\omega L^\infty_{L}}+ \kappa \right)\|r-\tilde{r}\|_{\msL^\infty_\mbR}\\
    	&\quad +L^2\|\bnu\|_{W^{1,\infty}_{\mbR}\TV_L}\|\tilde{r}\|_{\msL^\infty_\mbR}  e^{\|r\|_{\msL^\infty_\mbR}} \left\|\bw_{0}-\tilde{\bw}_{0}\right\|_{\msL^\infty_\mbR \msL^1_\omega}\\
    	&\quad + L^3\|\bnu\|_{W^{1,\infty}_{\mbR}\TV_L}\|\tilde{r}\|_{\msL^\infty_\mbR}e^{\|r\|_{\msL^\infty_\mbR}}\left(1+\|\bw_0\|_{\msL^\infty_\mbR \msL^1_\omega} +  \|\bmeta\|_{\msL^\infty_{\mbR} \msL^1_\omega L^\infty_{L}}\right) \|r-\tilde{r}\|_{\msL^\infty_\mbR}\\
    	&\quad +L^2 \|\bnu\|_{W^{1,\infty}_{\mbR}\TV_L} e^{(1+L)\left(\|r\|_{\msL^\infty_\mbR}+\|\tilde{r}\|_{\msL^\infty_\mbR}\right) } \sup_{t\in \mbR}\left\|\bw_{0}-\tilde{\bw}_{0}\right\|_{\msL^\infty_\mbR \msL^1_\omega}\\
    	&\quad +  L^2 \|\bnu\|_{W^{1,\infty}_{\mbR}\TV_L}e^{(1+L)\left(\|r\|_{\msL^\infty_\mbR}+\|\tilde{r}\|_{\msL^\infty_\mbR}\right) }\left(1+ L\left(1+\|\bw_0\|_{\msL^\infty_\mbR \msL^1_\omega} +  \|\bmeta\|_{\msL^\infty_{\mbR} \msL^1_\omega L^\infty_{L}}\right)\right)\|r-\tilde{r}\|_{\msL^\infty_\mbR} 
    \end{align*}
    Since $r,\,\tilde{r} \in \mfB_R(L)$ it holds that
    \begin{equation*}
    	\|r\|_{\msL^\infty_{\mbR}} \vee \|\tilde{r}\|_{\msL^\infty_{\mbR}} \leq R + \left(\|\bnu\|_{W^{1,\infty}_{\mbR} \text{TV}_L} \|\bw_0\|_{\msL^\infty_{\mbR}\msL^1_{\omega}} \vee \|\tilde{\bnu}\|_{W^{1,\infty}_{\mbR} \text{TV}_L} \|\tilde{\bw}_0\|_{\msL^\infty_{\mbR}\msL^1_{\omega}} \right).
    \end{equation*}
    Hence, observing that all terms on the right hand side of the long inequality bounding $\sup_{t\in \mbR} |r_t-\tilde{r}_t| = \|r-\tilde{r}\|_{\msL^\infty_{\mbR}}$ are multiplied by either $\|r-\tilde{r}\|_{\msL^\infty_{\mbR}}$ or $\|\bw_0-\tilde{\bw}_0\|_{\msL^\infty_\mbR \msL^1_\omega}$, there exists a 
    \begin{equation*}
    	\bar{L}_0 \coloneqq \bar{L}_0(\bnu, \bw_0,\tilde{\bw}_0,\bmeta, R,\lambda,\kappa) \in \left(0,\bar{L}_1(\bmeta,\bnu,\bw_0,\kappa,\lambda,R)\wedge \bar{L}_1(\tilde{\bmeta}, \tilde{\bnu}, \tilde{\bw}_0,\kappa,\lambda,R) \right)
    \end{equation*}
    
     sufficiently, small such that for some constant $C\coloneqq C(\bnu, \bw_0,\tilde{\bw}_0,\bmeta, R,\lambda,\kappa)>0$ one has
    \begin{align*}
    	\|r-\tilde{r}\|_{\msL^\infty_\mbR} \leq C e^{(1+L)\big(R + \big(\|\bnu\|_{W^{1,\infty}_{\mbR} \text{TV}_L} \|\bw_0\|_{\msL^\infty_{\mbR}\msL^1_{\omega}} \vee \|\tilde{\bnu}\|_{W^{1,\infty}_{\mbR} \text{TV}_L} \|\tilde{\bw}_0\|_{\msL^\infty_{\mbR}\msL^1_{\omega}} \big)\big) } \|\bw_0 -\tilde{\bw}_0\|_{\msL^\infty_\mbR \msL^1_\omega} .
    \end{align*}
Proofs of the other stability estimates follow in a similar manner.
\end{proof}
%
%
%
%
\section{Numerical Experiments}\label{sec:numerics}
We provide numerical results on the partial equilibrium case in the life cycle model \eqref{thm:life_cycle_optimal_rep} where we hold the interest rate to be a constant through the life cycle. We train a neural network to approximate the optimal consumption policy given in 
\eqref{eq:life_cycle_problem}. In contrast to solving the HJB equation using PDE solver based on finite-elements method, our approach using deep learning method does not require imposing boundary conditions and is generative by nature which makes it easy to interpolate the consumption policy. It is worth mentioning that our goal is to showcase how we can numerically solve a stochastic control problem using neural networks and we do not focus on finding the best neural network architecture or training procedure that achieves the lowest loss possible. Our numerical strategy resembles the direct method of \cite{carmona_lauriere2021} in that we restrict controls to a neural network class and solve a Monte Carlo stochastic optimization problem rather than discretizing the associated HJB equation. However, unlike their McKean-Vlasov control and mean field formulations, our implementation considers a single-agent life-cycle problem in partial equilibrium, without law-dependent dynamics or a fixed-point condition over the population distribution.

We will leave the more general case of finding equilibrium interest rates and simulating the OLG models using neural networks for future work. Besides providing validation to the theories we have developed in this work, we believe that our approach using neural networks might be of independent interest to researchers working on building economic models using dynamical systems.  

We now describe our loss function, training and evaluation of the consumption policy network, following the notation used in Section \ref{sec:life_cycle}. Our loss function comes from approximating the total utility given in \eqref{eq:life_cycle_problem} by choosing a number of time steps so that the integral inside the expectation can be approximated by a finite sum. More precisely, we take the discretised stochastic income with geometric Brownian motion in the following recursive form:
\begin{equation*}
    \eta_{t+1} = \eta_{t} + \mu 
    \eta_t \Delta t + \sigma \eta_t z, \ \ z \sim \cN(0, (\Delta t)^2 ),
\end{equation*}
where $\Delta t : = L /M$ is the length of the time step given the life span $T$ and the number of time steps $M$.

We define the utility function $u$ using the CRRA function:
\begin{equation*}
u(x, \gamma) = 
    \begin{cases}
    \frac{x^{1-\gamma_1}}{1-\gamma_1}, & \text{if } \gamma_1 \neq 1, \, x > 0 \\
    \log(x), & \text{if } \gamma_1 = 1,  \, x > 0,\\
    \end{cases}
\end{equation*} 
and the terminal utility:
\begin{equation*}
    e^{-\rho L} \lambda u(w_L, \gamma_2).
\end{equation*} 
To ensure numerical stability when we differentiate $u$ when $x$ is small, we replace $x$ by a preset threshold which is a small positive number when $x$ is smaller than that threshold. We also add a quadratic penalty for terminal utility when terminal wealth $w_L$ is negative. This helps the training to stay away from solutions that give rise to negative terminal wealth.   

We define our consumption network $\alpha$ to take three inputs: time point $t$, wealth level $w_t$ and income level $\eta_t$. We compute the wealth at time $t$ recursively by the following expression using our consumption policy network $\alpha$:
\begin{equation}\label{eq:wealth process}
    w_t = w_{t-1} + r w_{t-1} \Delta t + \eta_t \Delta t - \alpha(t-1, w_{t-1}, \eta_{t-1}) \Delta t.
\end{equation}

The loss function is defined as an approximation to the expectation of the accumulated utilities over the income process:
\begin{equation}\label{loss function: stochastic income, fixed interest rate}
    J(\alpha) := \frac{1}{N} \sum_{i=0} ^N J^{(i)}(\alpha),
\end{equation}
where 
\begin{equation*}
    J^{(i)}(\alpha) := - \sum_{t=0} ^L e^{-\rho t}u(\alpha(t,w^{(i)} _t, \eta_t ^{(i)}),\gamma_1) \Delta t + e^{-\rho L} \lambda u(w^{(i)} _L, \gamma_2)
\end{equation*}
that captures the utility based on income process for agent $i$. In the implementation, we simulate different income processes by varying the random seed for sampling $z$ from Gaussian $\cN(0, (\Delta t)^2 )$ for each agent $i, \, 1 \leq i \leq N$. We allow the initial income $y^{(i)} _0$ to be sampled from either a uniform or a log-normal distribution, and the initial wealth $w^{(i)} _0$ from a uniform or a Pareto distribution. 

We adopt a simplified ResNet architecture \cite{he2015deep} for our neural network which consists only of convolutional layers and shortcut connections with one projection layer at the top and a projection head at the bottom. The number of residual blocks is set to be a model parameter that can be chosen according to the problem. We have chosen the ResNet architecture due to its simpleness and its capability to mitigate gradient vanishing during training. To ensure that the output of the consumption network is non-negative, we choose the Softplus activation function in the final layer. In doing so, we avoid having to impose a hard constraint in the optimisation. 

We find it helpful for convergence of training to pretrain the consumption network for a certain of epochs to minimise the Mean Squared Error (MSE) between the consumption trajectory of the stochastic model and that given by the deterministic closed-form solution. We choose the AdamW optimiser \cite{loshchilov2018decoupled} and apply learning rate decay. We train the consumption policy network on $4000$ agents and then apply the trained policy to datasets obtained from sampling initial income and wealth from different distributions. We set the pretraining epochs to $400$ and the main training epochs to $500$. The learning rate decay step is taken to be $150$ with a decay factor of $0.6$. We set the number of residual blocks in our neural network to $2$ and the number of hidden units in each fully connected layer to $25$. We apply gradient clipping with norm $1.0$.

First, we initialise wealth and income with given values: $\eta_0 = 1,\, w_0 = 10$. In Figure \ref{fig:stochastic-solution-deterministic-initial-wealth-income}, we plot the income, consumption and wealth trajectories of a population of $4000$ agents through the life cycle while highlighting $50$ of them.
\begin{figure}[!htbp]
  \scriptsize
  \centering
  \begin{subfigure}[t]{0.47\textwidth}
    \centering
    \includegraphics[width=\linewidth]{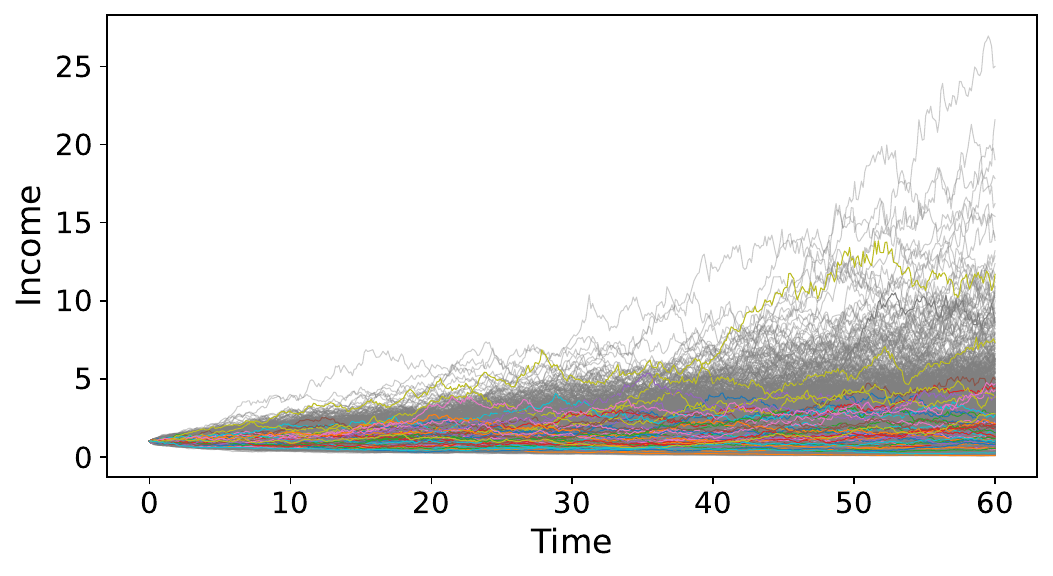}
    \caption{Income trajectories.}
    \label{fig:stochastic-income-tranjectories}
  \end{subfigure}
  \hfill
  \begin{subfigure}[t]{0.47\textwidth}
    \centering
    \includegraphics[width=\linewidth]{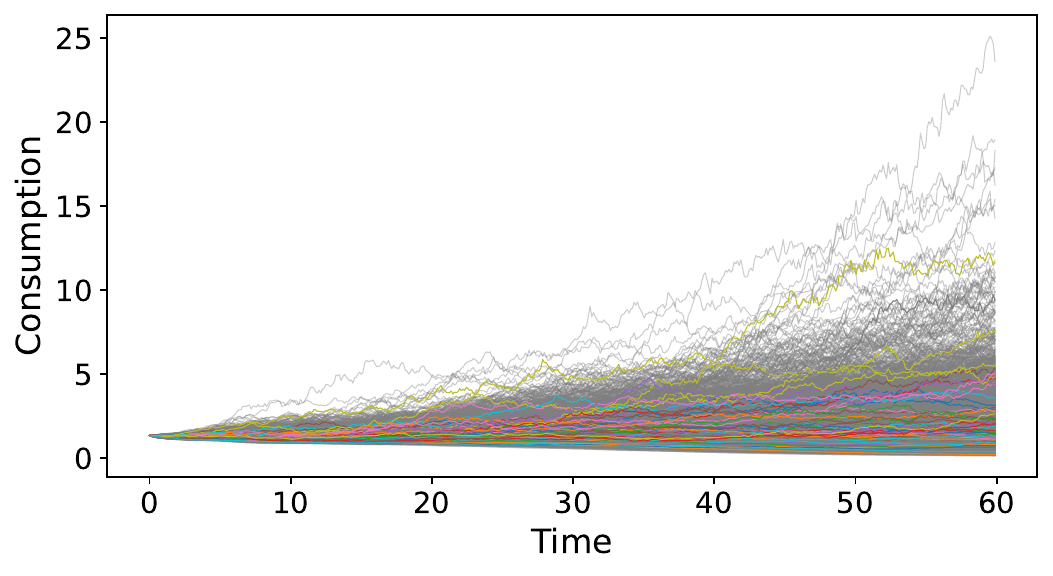}
    \caption{Consumption trajectories.}
    \label{fig:stochastic-consumption-trajectories}
  \end{subfigure}
  \hfill
  \begin{subfigure}[t]{0.47\textwidth}
    \centering
    \includegraphics[width=\linewidth]{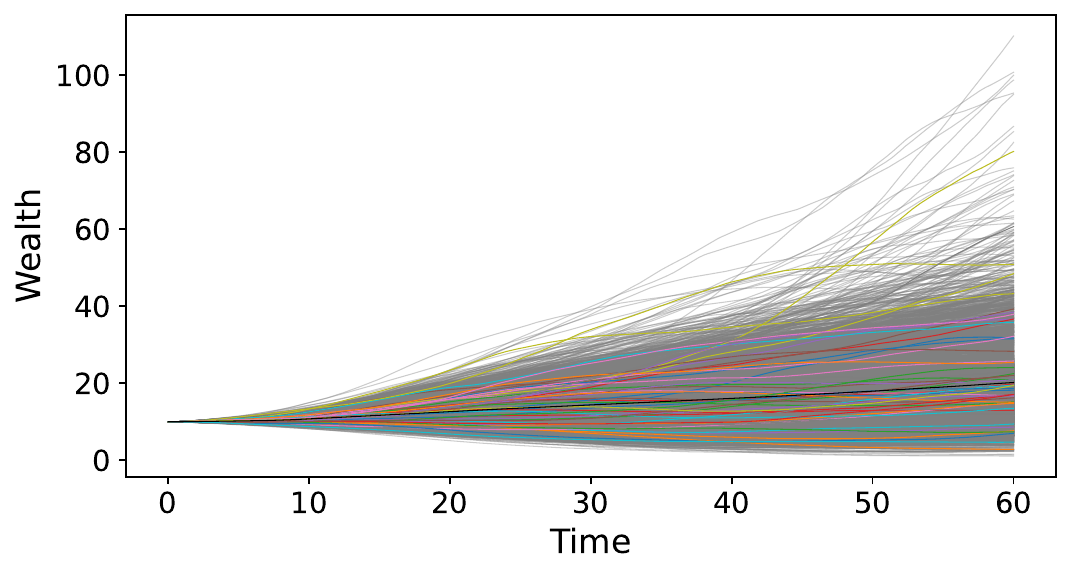}
    \caption{Wealth trajectories}
    \label{fig:stochastic-wealth-trajectories}
  \end{subfigure}
  \caption{Plots of the income, consumption and wealth trajectories using our trained consumption policy on deterministic initial wealth $w_0 = 10.0$ and income $\eta_0 = 1.0$ for $4000$ agents (in grey lines) while highlighting $50$ of them using coloured lines. Other parameters we set are: $\delta = 0.02,\, r= 0.03,\, \gamma_1 = 2,\, \gamma_2 = 2, \, \mu = 0.01, \, \lambda = 100, \, L = 60, \, \sigma = 0.1$.}
  \label{fig:stochastic-solution-deterministic-initial-wealth-income}
\end{figure}
From Figure \ref{fig:stochastic-solution-deterministic-initial-wealth-income} and the data associated to the plots, we can observe that the consumption levels keep up with the income level and no agents have negative wealth at the end of the life cycle. Agents that have higher incomes tend to consume more throughout life time and are still able to accumulate larger wealth.   

We then apply the trained policy to a dataset with initial income and initial wealth sampled from log-normal and Pareto distributions respectively. We plot the distribution of wealth of the population of the same size as before at three age groups in Figure \ref{fig:wealth-distributions}.
\begin{figure}[!htbp]
  \scriptsize
  \centering
  \begin{subfigure}[t]{0.33\textwidth}
    \centering
    \includegraphics[width=\linewidth]{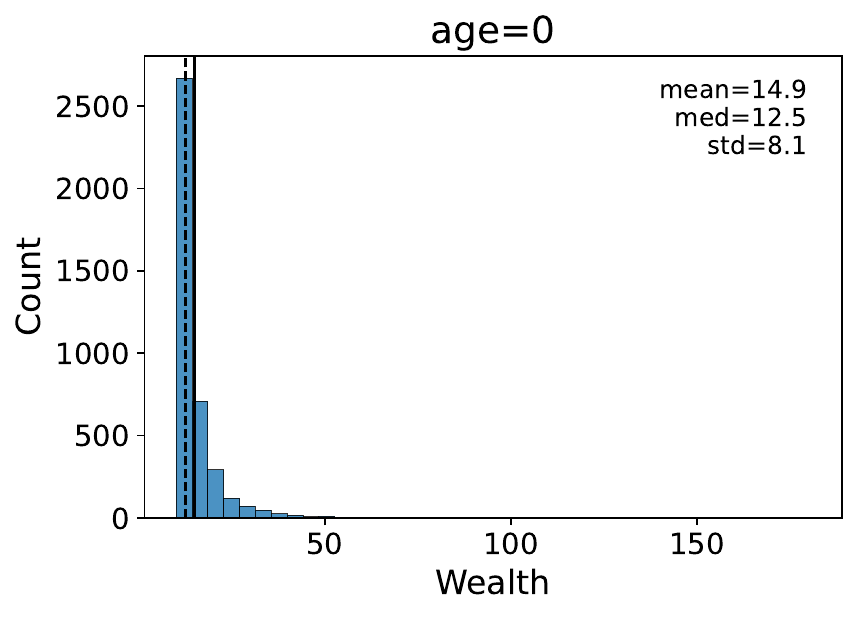}
    \caption{}
    \label{fig:wealth-distributions-age-0}
  \end{subfigure}
  \hfill
  \begin{subfigure}[t]{0.33\textwidth}
    \centering
    \includegraphics[width=\linewidth]{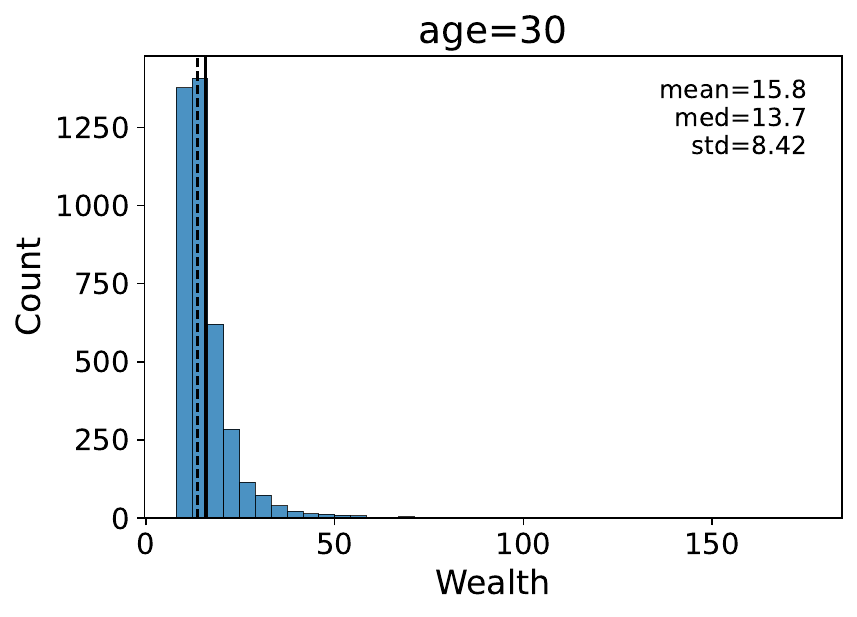}
    \caption{}
    \label{fig:wealth-distributions-age-30}
  \end{subfigure}
  \hfill
  \begin{subfigure}[t]{0.33\textwidth}
    \centering
    \includegraphics[width=\linewidth]{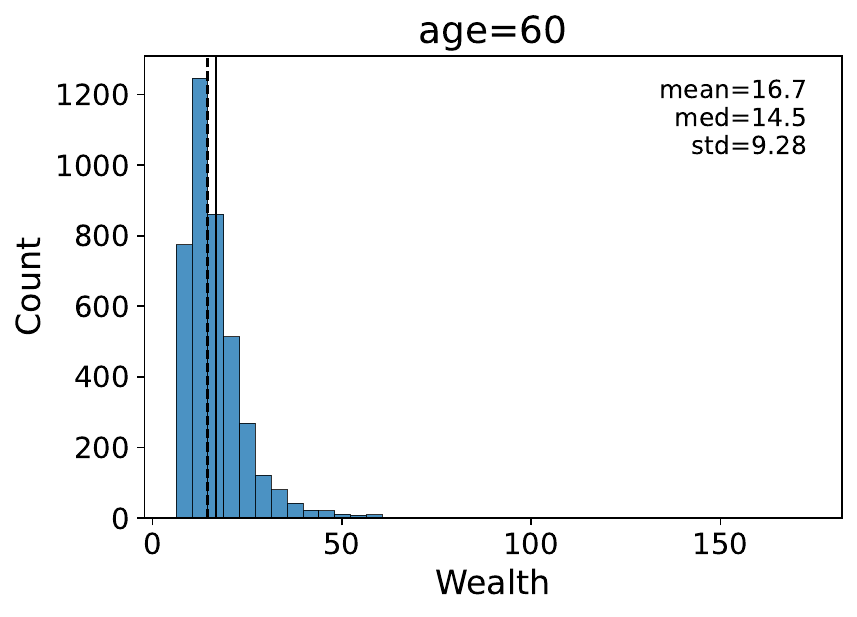}
    \caption{}
    \label{fig:wealth-distributions-age-60}
  \end{subfigure}
  \caption{Distributions of wealth at age $0, 30, 60$ in the life cycle when initial income are sampled from the log-normal distribution with mean $0$ and standard deviation $1$, and initial wealth sampled from a Pareto distribution with scale $10$ and shape $3$.}
  \label{fig:wealth-distributions}
\end{figure}
From Figure \ref{fig:wealth-distributions}, we can observe that the distribution of wealth among the population appears to shift from the initial Pareto distribution at age $0$ to a more Gaussian-like distribution with heavier tails at the age of $60$.

We also plot the value function and the consumption policy at given income levels for three age groups, as shown in Figure \ref{fig:policy-and-value-plots}. 
\begin{figure}[!htbp]
  \scriptsize
  \centering
  \begin{subfigure}[t]{0.47\textwidth}
    \centering
    \includegraphics[width=\linewidth]{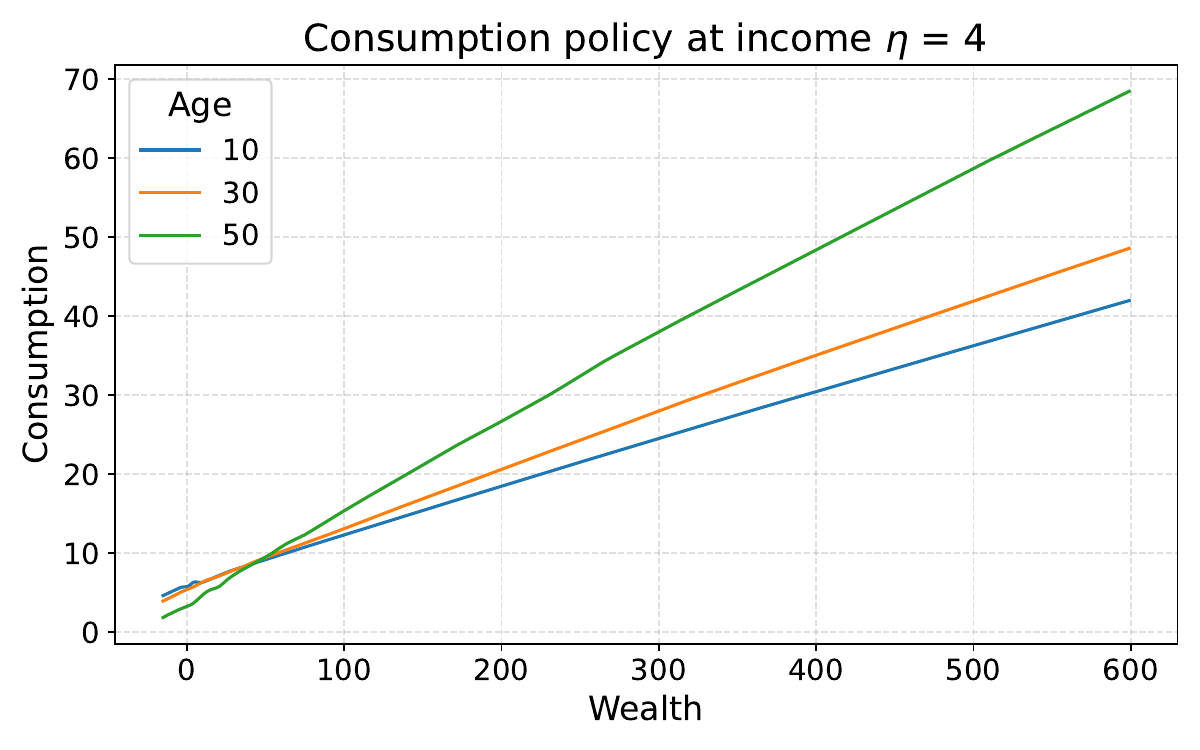}
    \caption{}
    \label{fig:consumption-policy-income-4}
  \end{subfigure}
  \hfill
  \begin{subfigure}[t]{0.47\textwidth}
    \centering
    \includegraphics[width=\linewidth]{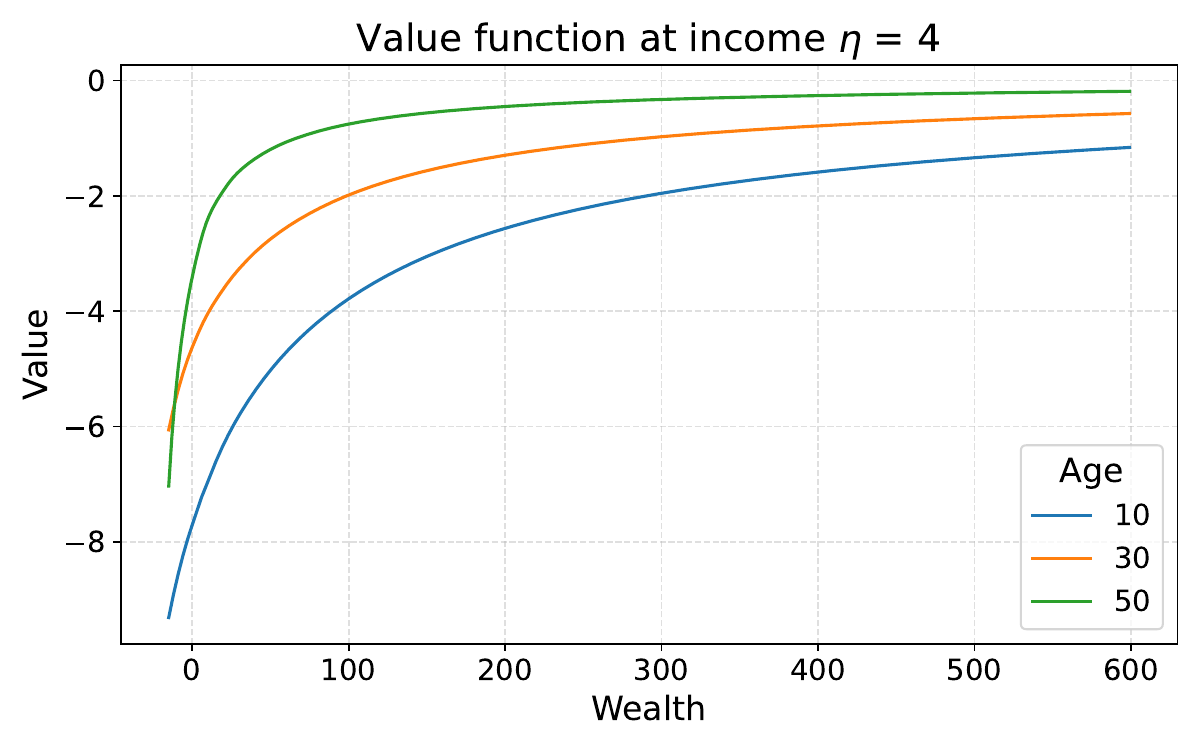}
    \caption{}
    \label{fig:value-function-income-4}
  \end{subfigure}
  \hfill
  \begin{subfigure}[t]{0.47\textwidth}
    \centering
    \includegraphics[width=\linewidth]{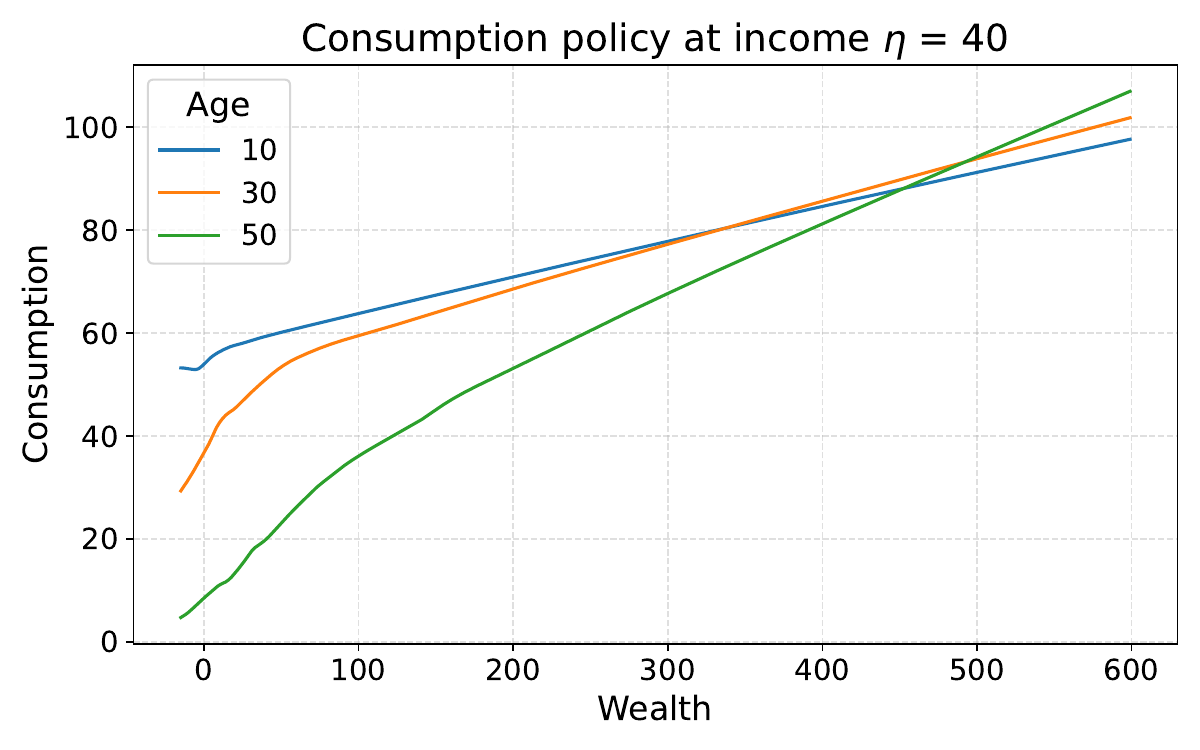}
    \caption{}
    \label{fig:consumption-policy-income-40}
  \end{subfigure}
  \hfill
    \begin{subfigure}[t]{0.47\textwidth}
    \centering
    \includegraphics[width=\linewidth]{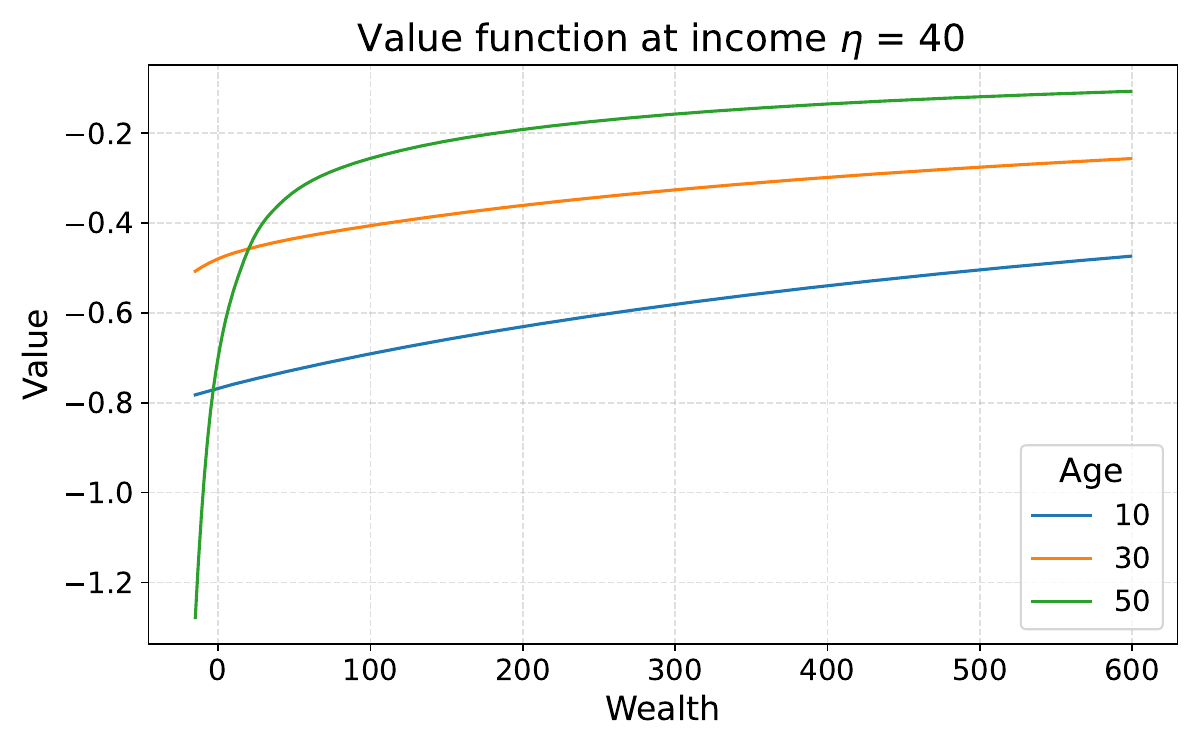}
    \caption{}
    \label{fig:value-function-income-40}
  \end{subfigure}
  \hfill
  \caption{Plots of the consumption policies and value functions at two income levels $\eta = 4.0$ and $\eta = 40.0$ with age $10$, $30$ and $50$.}
  \label{fig:policy-and-value-plots}
\end{figure}
We make a few observations from Figure \ref{fig:policy-and-value-plots}: 
\begin{enumerate}
    \item [1.] At a given income level, the slope of the consumption level is positively correlated to age. In other words, the older the age, the higher the propensity to consume with respect to an increase in wealth. We can also observe a diminished return on the marginal propensity to consume as wealth increases from around $0$ at income level $\eta = 40$.
    \item [2.] At both income levels, there appears to be a wealth level where the consumption levels of different age groups intersect and then reverse their order. We think this can be interpreted as a tipping point where the priority between consumption and bequest flips for different age groups. This tipping point tends to happen at a higher wealth level when the income level is higher.
    \item [3.] At both income levels, the values are higher for groups with older ages at most of the wealth levels except when the wealth level is small or negative (e.g. smaller than $25$ when $\eta = 40$). The slopes of the value curves follow a reversed order, i.e. younger age groups tend to have larger marginal increase of value over wealth, except when the wealth level is small or negative. We think this might be explained by the model parameters we have set which encourages consumption over bequest.    
\end{enumerate}    

\section*{Acknowledgements}
FH gratefully acknowledges the support of the Center for Advanced Studies (CAS) in Oslo, which funded the Signatures for Images project during the academic year 2023/2024, when this work started. AM and C were supported by EPSRC Mathematical Sciences Small Grant UKRI253: FBSDE Methods in Heterogeneous Agent Models. Part of this work was carried out while C was at the University of Bath.

\section*{Funding}
Open Access funding provided by the IReL Consortium.

\section*{Data Availability}
This study does not use any real-world data; it focuses on theoretical developments.
The numerical experiments were conducted using synthetic data, and the procedures for generating these datasets are fully documented in the paper. No external data are required to reproduce the reported results. The code accompanying this paper is available in the versioned \url{https://github.com/CangxiongChen/FBSDE_life_cycle_model/releases/tag/v1.0.0} and is permanently archived on \url{https://doi.org/10.5281/zenodo.21677841}.

\section*{Open Access}
This article is licensed under a Creative Commons Attribution 4.0 International License, which permits use, sharing, adaptation, distribution and reproduction in any medium or format, as long as you give appropriate credit to the original author(s) and the source, provide a link to the Creative Commons licence, and indicate if changes were made. The images or other third party material in this article are included in the article's Creative Commons licence, unless indicated otherwise in a credit line to the material. If material is not included in the article's Creative Commons licence and your intended use is not permitted by statutory regulation or exceeds the permitted use, you will need to obtain permission directly from the copyright holder. To view a copy of this licence, visit \url{http://creativecommons.org/licenses/by/4.0/}.

\newpage
\appendix

\section{Explicit solutions to linear BSDEs}
We recall here a result providing an explicit solution to certain linear BSDEs. It is this representation that plays a central role in the derivation of the optimal consumption profile.
The following proposition is a special case formulation of \cite[Prop. 6.2.1]{pham_09_continuous}. 

\begin{prop}\label{prop: linear BSDE}
Let $\{\beta_t\}_{t\in [0,T]}$ be a Brownian motion supported on a probability space $(\Omega,\cF,\{\cF_t\}_{t\in [0,T]},\PP)$, suppose $a$ is bounded and progressively measurable processes valued in $\RR$, and let $\xi\in \scL^2_\Omega$. Then the unique solution $(y,z)$ to the linear BSDE 
\begin{equation*}
    -\dd y_t =a_t y_t\dd_t + z_t \dd \beta_t,\quad y_T=\xi, 
\end{equation*}
is given by 
\begin{equation*}
    y_t = \EE[\Gamma_T  \Gamma_t^{-1} \xi|\cF_t], 
\end{equation*}
where $\Gamma:[0,T]\rightarrow \RR$ solves the (random) ODE
\begin{equation*}
    \dd \Gamma_t = \Gamma_t a_t \dd t,\quad \Gamma_0=1. 
\end{equation*}
\end{prop}

As illustrated above, the so called stochastic maximum principle guarantees that the solution to the HJB equation is tightly connected to the the solution of certain BSDEs. Thus, when the Hamiltonian in the HJB equation is linear (to be specified later) it turns out that the optimal control can be explicitly found through the solution to the above BSDE.

\section{Technical Lemmas}

We collect some useful results on applications of the Leibniz rule and explicit solutions to linear BSDE.
\begin{lem}\label{lem:diff_two_param_function_integral}
   Let $f\in C^1(\mbR_+\times \mbR_+ ;\mbR)$ and $\bnu = \{\nu_t\}_{t\in \mbR}$ be a flow of \emph{regular demographic measures} in the sense of Definition~\ref{def:dem_flow} with density $\bn \coloneqq \{n(t,\,\cdot\,)\}_{t\in \mbR}$. Then, it holds that
\begin{align*}
    \frac{\dd}{\dd t} \int_{t-L}^t f(t,b) \nu_t(\dd b)  = &\, f(t,t)n(t,t) -f(t,t-L)n(t,t-L) + \int_{t-L}^t \partial_t f(t,b) n(t,b)\dd b \\
    &\, + \int^t_{t-L} f(t,b) \partial_t n(t,b) \dd b.
\end{align*}
\end{lem}

\begin{proof}
     By assumption we can write
     \begin{equation*}
         \int_{t-L}^t f(t,b) \nu_t(\dd b)  = \int_{t-L}^t f(t,b) n(t, b)\dd b, 
     \end{equation*}
     and the formula on the right hand side follows after applying the Leibniz integral rule and product rule for derivatives.
\end{proof}
\begin{prop}[Solution to linear BSDEs]\label{prop:bsde_linear_solution}
Let $\{\beta_t\}_{t\in [0,L]}$ be an $\mbR$-Brownian motion with random initial condition on a filtered probability space $(\Omega, \cF,\{\cF_t\}_{t\in [0,L]},\PP)$, $r$ be a fixed path in $ C([0,L])$ and $g\in \msL^1(\Omega;\mbR)$. Then, the linear BSDE 
\begin{equation}\label{eq:special lin BSDE}
\begin{aligned}
        -\dd y_t &= r_t y_t\dd t +z_t \dd \beta_t, \quad y_L=g.  
\end{aligned}
\end{equation}
 has a unique strong solution adapted to the filtration $\{\cF_t\}_{t\in [0,L]}$ described by the tuple $(y,z)$ to \eqref{eq:special lin BSDE} where
\begin{equation}\label{eq:explicit BSDE}
    y_t = \EE\left[\exp\left(\int_t^L r_s\dd s\right)g|\cF_t\right], 
\end{equation}
and $z$ is obtained by the martingale representation theorem.
\end{prop}
\begin{proof}
See \cite[Prop.~6.2.1.]{pham_09_continuous}. 
\end{proof}

\noindent\textbf{Publisher's Note.} Springer Nature remains neutral with regard to jurisdictional claims in published maps and institutional affiliations.
\end{document}